\documentclass{my_m2an} 
\usepackage{amsmath,amssymb}
\usepackage[mathcal]{euscript} 
\usepackage{tikz}
\usetikzlibrary{patterns} 
\usepackage{fullpage}
\usepackage{hyperref}
\usepackage{url}

\usepackage{subcaption}

\newtheorem{theorem}{Theorem}[section]
\newtheorem{lemma}[theorem]{Lemma}

\theoremstyle{definition}
\newtheorem{definition}[theorem]{Definition}

\theoremstyle{remark}
\newtheorem{remark}[theorem]{Remark}

\allowdisplaybreaks
\usepackage{color}
\definecolor{bfonce}{rgb}{0.,0.,0.8}
\definecolor{bclair}{rgb}{0.87,0.92,1.}
\definecolor{rclair}{rgb}{1.,0.92,0.87}
\definecolor{orangec}{rgb}{1.,0.6,0.}
\definecolor{vertc}{rgb}{0.,0.8,0.6}
\definecolor{vertf}{rgb}{0.,0.7,0.0}
\definecolor{rougef}{rgb}{0.8,0.,0.}
\newcommand{\bfa}{{\boldsymbol a}}
\newcommand{\bfe}{{\boldsymbol e}}
\newcommand{\bfF}{{\boldsymbol F}}
\newcommand{\bff}{{\boldsymbol f}}
\newcommand{\bfG}{{\boldsymbol G}}
\newcommand{\bfn}{\boldsymbol n}
\newcommand{\bfu}{{\boldsymbol u}}
\newcommand{\bfx}{\boldsymbol x}
\newcommand{\bfy}{\boldsymbol y}
\newcommand{\bfvarphi}{{\boldsymbol \varphi}}
\newcommand{\dx}{\ \mathrm{d}\bfx}
\newcommand{\dt}{\ \mathrm{d} t}
\newcommand{\mesh}{{\mathcal M}}
\newcommand{\edge}{{\sigma}}
\newcommand{\edgeperp}{{\tau}}
\newcommand{\edges}{{\mathcal E}}
\newcommand{\edgesint}{{\mathcal E}_{\mathrm{int}}}
\newcommand{\edgesext}{{\mathcal E}_{\mathrm{ext}}}
\newcommand{\edgesinti}{{\mathcal E}^{(i)}_{\mathrm{int}}}
\newcommand{\edgesexti}{{\mathcal E}^{(i)}_{\mathrm{ext}}}
\newcommand{\edgesi}{{\edges\ei}}
\newcommand{\edged}{\epsilon}
\newcommand{\edgesd}{{\widetilde {\edges}}}

\newcommand{\edgesdint}{{\edgesd_{{\rm int}}}}

\newcommand{\edgespart}{{\mathfrak F}}
\newcommand{\ei}{^{(i)}}
\newcommand{\m}{^{(m)}}
\newcommand{\dive}{{\mathrm{div}}}
\newcommand{\gradi}{\boldsymbol \nabla}
\newcommand{\diam}{{\mathrm{diam}}}
\newcommand{\llbracket}{\bigl[ \hspace{-0.55ex} |}
\newcommand{\rrbracket}{| \hspace{-0.55ex} \bigr]}
\newcommand{\characteristic}{{1 \hspace{-0.8ex} 1}}
\newcommand{\mnn}{{m\in\xN}}
\newcommand{\eg}{{\it e.g.}}
 \newcommand{\xL}{\mathrm{L}}
\begin{document}
\title[A quasi - second order staggered scheme for shallow water]
{A consistent  quasi - second order staggered scheme for the two-dimensional shallow water equations}
%

\author{R. Herbin}
\address{I2M UMR 7373, Aix-Marseille Universit\'e, CNRS, Ecole Centrale de Marseille. 
39 rue Joliot Curie. 13453 Marseille, France. \\ (raphaele.herbin@univ-amu.fr)}

\author{J.-C. Latch\'e}
\address{IRSN, BP 13115, St-Paul-lez-Durance Cedex, France (jean-claude.latche@irsn.fr)}

\author{Y. Nasseri}
\address{I2M UMR 7373, Aix-Marseille Universit\'e, CNRS, Ecole Centrale de Marseille. 
39 rue Joliot Curie. 13453 Marseille, France. \\ (youssouf.nasseri@univ-amu.fr)}

\author{N. Therme}
\address{CEA/CESTA 33116, Le Barp, France (nicolas.therme@cea.fr)}.

%


\begin{abstract}
{
A quasi-second order scheme is developed to obtain approximate solutions of the two-dimensional shallow water equations with bathymetry.
The scheme is based on a staggered  finite volume space discretisation: the scalar unknowns are located in the discretisation cells while the vector unknowns are located on the edges of the mesh.
A MUSCL-like interpolation for the discrete convection operators in the water height and momentum balance equations is performed in order to improve the accuracy of the scheme.
The time discretisation is performed either by a first order segregated forward Euler scheme or by the second order Heun scheme.
Both schemes are shown to preserve the water height positivity under a CFL condition and an important state equilibrium known as the lake at rest.
Using some recent Lax-Wendroff type results for staggered grids, these schemes are shown to be \emph{LW-consistent} with the weak formulation of the continuous equations, in the sense that if a sequence of approximate solutions is bounded and strongly converges to a limit, then this limit is a weak solution of the shallow water equations ; besides, the forward Euler scheme is shown to be \emph{LW-consistent} with a weak entropy inequality.
Numerical results confirm the efficiency and accuracy of the schemes.
}

\end{abstract}

\maketitle

\hspace{.8cm}{\bf Keywords} {Finite-volume scheme, MAC grid, shallow water flow.}
\\

%
%
\section{Introduction} \label{sec1:pbcont}

The shallow water equations (SWE) form a hyperbolic system of two conservation laws (mass and momentum balance equations) which models the flow of an incompressible fluid, assuming that the range of the vertical height of the flow is small compared to the horizontal scales.
This model is widely used for the simulation of numerous geophysical phenomena, such as flow in rivers and coastal areas, lava flows or snow avalanches; many other applications may be found as, for instance, in process industries.

The SWE with bathymetry, posed over a space-time domain $\Omega \times (0,T)$,  where $\Omega$ is an open bounded subset of $\xR^2$ of boundary $\partial \Omega$ and $T >0$, read
\begin{subequations} \label{eq:sw}
\begin{align}\label{eq:mass} &
\partial_t h + \dive (h \bfu) =0
&&
\mbox{in } \; \Omega \times (0, T),
\\ \label{eq:mem} &
\partial_t (h \bfu) + \dive(h \bfu \otimes \bfu) + \gradi p + g h \gradi z = 0
&&
\mbox{in }\ \Omega \times (0, T),
\\ \label{bc} &
p = \frac{1}{2} g h^2
&&
\mbox{in }\ \Omega \times (0, T),
\\ \label{bc:ins} &
\bfu \cdot \bfn =0
&&
\mbox{ on }\ \partial\Omega \times (0,T),
\\ \label{eq:ini_cond} &
h(\bfx, 0) = h_0, \,\, \bfu(\bfx,0)=\bfu_0
&&
\mbox{ in }\Omega,
\end{align}
\end{subequations}
where the unknowns are the water height $h$ and the (vector valued) horizontal velocity of the fluid $\bfu=(u_1,u_2)$, averaged over the fluid depth; $g$ is the gravity constant and $z$ the (given) bathymetry, supposed to be regular in this paper.
The initial conditions, featured in \eqref{eq:ini_cond}, are $h_0 \in \xL^\infty(\Omega)$ and $\bfu_0 =(u_{0,1},u_{0,2})\in \xL^\infty(\Omega,\xR^2)$ with $h_0 \ge 0$.
We suppose here that the boundary conditions boil down to \eqref{bc:ins}, {\it i.e.}\ an impermeability boundary condition.

Let us recall that if $(h,\bfu)$ is a regular solution of \eqref{eq:sw}, the following potential energy balance and kinetic energy balance are obtained by manipulations on the mass and momentum equations:
\begin{align} \label{pot_bal} &
\partial_t (\frac{1}{2} g h^2) + \dive (\frac{1}{2} g h^2 \bfu) + \frac{1}{2} g h^2 \dive \bfu = 0,
\\ \label{kin_bal} &
\partial_t (\frac{1}{2} h \vert \bfu \vert^2) + \dive (\frac{1}{2} h  \vert \bfu \vert^2 \bfu) + \bfu \cdot \gradi p  + gh \bfu \cdot \gradi z = 0.
\end{align}
Summing these equations, we obtain an entropy balance equation: $\partial_t E + \dive \Phi = 0$, where the entropy-entropy flux pair $(E, \Phi)$ is given by:
\begin{equation}\label{def:entropy}
E = \frac 1 2  h|\bfu|^2 + \frac 1 2 g h^2 + ghz \text{ and } \Phi =(E +  \frac 1 2 g h^2) \bfu.
\end{equation}
For non regular functions, the above manipulations are no longer valid, and the entropy inequality $\partial_t E + \dive \Phi \leq 0$ is satisfied in a distributional sense.

The system \eqref{eq:sw} has been intensively studied, both theoretically and numerically, and it is impossible to give a comprehensive list of references; we thus refer for an introduction to classical textbooks, \eg\ \cite{tan-92-sha, bou-04-non}, and to more recent reviews \cite{aud-18-hdr, cas-17-wel, xin-17-num} and references therein.
If no dry zone exists, the system is known to be strictly hyperbolic, and its solution may develop shocks, so that the finite volume method is often preferred for numerical simulations.
In such a context, two main approaches for the spatial discretisation are found in the literature: the first one is the colocated approach, where the expression of the numerical fluxes usually relies on (approximate or exact) Riemann solvers, see \eg\ \cite{bou-04-non,cas-17-wel} and references therein; the other one is based on a staggered arrangement of the unknowns on the grid.
This latter approach is quite classical in the hydraulic and ocean engineering community, where it is used on rectangular grids and known as the Arakawa-C discretisation, see \eg\ \cite{ara-81-pot, bon-05-ana, ste-03-sta}; this rectangular staggered arrangement is also known as the Marker-And-Cell (MAC) discretisation \cite{har-65-num, har-71-num}.
With this space discretisation, the development of Riemann solvers is made difficult by the fact that the discrete water height and velocity can no more be considered as piecewise constant on the same partition of the computational domain; in fact, we are not aware of any attempt in this direction.
Instead, the numerical diffusion necessary to the scheme stability is obtained by a simple upwinding of the transport terms with respect to the material velocity; this makes such schemes very simple to implement and efficient (since the fluxes evaluation is straightforward), which probably explains their popularity.
An important feature comforting their use is that, despite this simplicity of the upwinding technique, they may be shown to inherit the stability properties of the continuous problem: non-negativity of the water height, preservation of the so-called "lake at rest" steady state; a careful design of the velocity convection operator, which is an essential ingredient for the entropy consistency, also yields the $\xL^2$-stability of first-order discretisations \cite{her-10-kin}.
Note also that such schemes admit natural semi-implicit variants (pressure correction schemes) h are numerically efficient and unconditionally stable, in the sense that they preserve the positivity of the water height and satisfy an entropy inequality without restrictive assumptions on the time step.
Finally, since the native numerical diffusion only depends on the material velocity, the accuracy is not lost in low Mach number situations (or zones).

While colocated schemes for the SWE have been the object of numerous mathematical studies in the last decades, the theoretical numerical analysis of staggered schemes for the SWE has only been recently undertaken.
A staggered scheme with an upwind choice for the convection operators and a forward Euler time discretisation is proposed and analysed in the case of one space dimension in \cite{doy-14-exp, gun-15-num}; the analysis is based on closely related works on the barotropic Euler equations, see \cite{her-18-con} and references therein.
In particular, in this one-dimensional setting, the Lax-Wendroff consistency  of the scheme (or LW-consistency for short) is shown as well as a the LW-entropy consistency, in the sense that if the scheme is assumed to converge strongly and in a bounded way, then the limit is a weak (or entropy weak) solution of the continuous problem (see \cite{lax-60-sys} for the seminal result).
A staggered scheme, still first order in time and space and with fluxes derived through the kinetic approach, is proposed in \cite{ber-15-kin} for the barotropic Euler equations, and a second order in space and first order in time scheme is studied in \cite{dur-20-ene} for the SWE; in these two works, the consistency issue is not addressed.

We present in this paper staggered schemes for the solution of SWE, and our aim is twofold:
\begin{list}{--}{\itemsep=0.ex \topsep=0.5ex \leftmargin=1.cm \labelwidth=0.7cm \labelsep=0.3cm \itemindent=0.cm}
\item First, we analyze and test a class of \emph{second order in time and space} schemes for the SWE, which were briefly presented in \cite{gal-20-sec}.
In this respect, the originality with previous works on staggered grids first lies in the formulation of the numerical flux for the convection operator, which is general enough to include the first order upwind choice already studied in \cite{her-19-dec} and a quasi-second order MUSCL-like procedure originally introduced in \cite{pia-13-for}.
This latter formulation is based on a purely algebraic limiter which includes several well-known higher order schemes.
The second order in time scheme is obtained by switching from a first order Euler time discretisation to the second order Heun (or RK2) method, thus really improving the accuracy of the approximate solution in regular zones.
Generic properties are shown to be preserved, such as the positivity of the water height and the preservation of the "lake at rest" steady state.
\item Second, the schemes are proven to be  LW-consistent, thanks to a generalised Lax-Wendroff theorem which is recalled in the appendix and was designed specially for this kind of application.
As far as we know, this is the first proof of LW-consistency of a numerical scheme for the SWE in the multidimensional setting.
Note in passing that this proof does not require the $BV$ boundedness that is classically required for the LW-consistency analysis of hyperbolic conservation laws.
Moreover, the first order in time scheme is also shown to be  \emph{LW-entropy consistent} in the sense that under some boundedness assumptions (which unfortunately, now includes a time $BV$ bound that seems difficult to bypass), the limit of any strongly converging sequence of approximate solutions converges to an entropy weak solution of the SWE \eqref{eq:sw}.
\end{list}
The proposed scheme may be easily extended to non-structured discretisations; however, since it copes with dry zones, it is often more efficient in practice to embed the computational domain in a larger domain which may be meshed by a structured discretisation.
We thus restrict the exposition to the case of non-uniform rectangular meshes.

This paper is organized as follows.
In Section 2, we introduce the space and time discretisations.
The discrete stability and well-balanced properties of the approximate solutions are stated and proven in Section \ref{sec:stab}.
Furthermore, under some convergence and boundedness assumptions, the approximate solutions are shown in Section \ref{sec:cons} to converge to a weak solution of the SWE \eqref{eq:sw}.
This proof heavily relies on a generalized Lax-Wendroff theorem \cite{gal-21-wea} which was recently proven to simplify the proofs of consistency of staggered schemes.
It is given in the appendix in a form adapted to the present case (see Theorem \ref{theo:lw}).
In Section \ref{sec:entropy}, we consider the first order time discretisation and show that any possible limit of the scheme satisfies a weak entropy inequality, again using the consistency result.
An in-depth numerical study of the schemes is presented in Section \ref{sec:num}.
%
%
\section{Space and time discretisation}

\subsection{Definitions and notations} \label{sec:discop}

Let $\Omega$ be a connected subset of $\xR^2$ consisting in a union of rectangles whose edges are assumed to be orthogonal to the canonical basis vectors, denoted by $\bfe^{(1)}$ and $\bfe^{(2)}$.

\begin{definition}[MAC discretisation] \label{def:MACgrid}
A discretisation $(\mesh, \edges)$ of $\Omega$ with a staggered rectangular grid (or MAC grid), is defined by:
\begin{list}{--}{\itemsep=0.ex \topsep=0.5ex \leftmargin=1.cm \labelwidth=0.7cm \labelsep=0.3cm \itemindent=0.cm}
\item A primal mesh $\mesh$ which consists in a conforming structured, possibly non uniform, rectangular grid of $\Omega$.
A generic cell of this grid is denoted by $K$, and its mass center by $\bfx_K$.
\item A set $\edges$ of all edges of the mesh, with $\edges= \edgesint \cup \edgesext$, where $\edgesint$ (resp. $\edgesext$) are the edges of $\edges$ that lie in the interior (resp. on the boundary) of the domain.
The set of edges that are orthogonal to $\bfe\ei$ is denoted by $\edgesi$, $i=1,\ 2$.
We then have $\edgesi= \edgesinti \cup \edgesexti$, where $\edgesinti$ (resp. $\edgesexti$) are the edges of $\edgesi$ that lie in the interior (resp. on the boundary) of the domain.

For $\edge\in\edgesint$, we write $\edge = K|L$ if $\edge = \partial K \cap \partial L$.
A dual cell $D_\edge$ associated to an edge $\edge \in\edges$ is defined as follows:
\begin{list}{-}{\itemsep=0.ex \topsep=0.ex \leftmargin=1.cm \labelwidth=0.7cm \labelsep=0.1cm \itemindent=0.cm}
\item if $\edge=K|L \in \edgesint$, then $D_\edge = D_{K,\edge}\cup D_{L,\edge}$, where $D_{K,\edge}$ (resp. $D_{L,\edge}$) is the half-part of $K$ (resp. $L$) adjacent to $\edge$ (see Fig. \ref{fig:mesh});
\item if $\edge \in \edgesext$ is adjacent to the cell $K$, then $D_\edge=D_{K,\edge}$.
\end{list}
For $i=1,2$, the domain $\Omega$ is split up in dual cells: $\Omega = \cup_{\edge \in \edgesi} \overline{D_\edge}$; the $i$-th grid is referred to as the $i$-th dual mesh.
The set of the edges of the $i$-th dual mesh is denoted by $\edgesd\ei$ (note that these edges may be non-orthogonal to $\bfe\ei$).
The dual edge separating two dual cells $D_\edge$ and $D_{\edge'}$ is denoted by $\edged=\edge|\edge'$.
\end{list}
\end{definition}

The discrete velocity unknowns are associated to the velocity cells and are denoted by $(u_{i,\edge})_{\edge\in\edgesi}$, $i=1,\ 2$, while the discrete scalar unknowns (water height and pressure) are associated to the primal cells and are denoted respectively by $(h_K)_{K\in\mesh}$ and $(p_K)_{K\in\mesh}$.

In order to define the scheme, we need some additional notations.
The set of edges of a primal cell $K$ and of a dual cell $D_\edge$ are denoted by $\edges(K) \subset \edges$ and $\edgesd(D_\edge)$ respectively; note that $\edgesd(D_\edge) \subset \edgesd\ei$ if $\edge \in \edgesi$.
For $\edge \in \edges$, we denote by $\bfx_\edge$ the mass center of $\edge$.
The vector $\bfn_{K,\edge}$ stands for the unit normal vector to $\edge$ outward $K$.

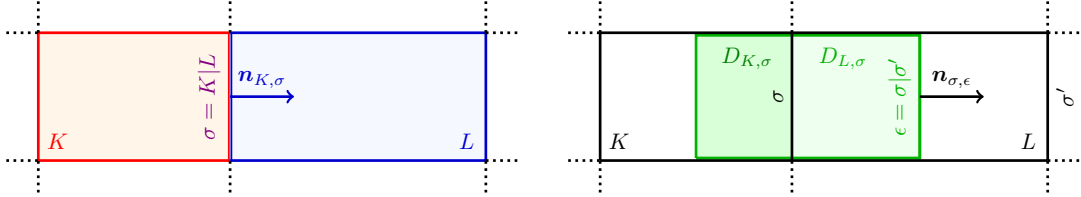
\begin{figure}[htbp]
\centering
\scalebox{0.85}{
\begin{tikzpicture}[scale=1] 
\fill[color=orangec!30!white,opacity=0.3] (0., 0.) -- (3., 0.) -- (3., 2.) -- (0., 2.) -- (0., 0.);
\fill[color=bclair,opacity=0.3] (3., 0.) -- (7., 0.) -- (7., 2.) -- (3., 2.) -- (3., 0.);

\draw[very thick, color=red] (3, 2.) -- (0., 2.) -- (0, 0.) -- (3., 0.);
\draw[thick, color=red] (2.98, 0.) -- (2.98, 2.) node[sloped, midway, left, anchor=south] {\textcolor{violet}{$\edge= K|L$}};
\node[color=red] at (0.3, 0.3){$K$};

\draw[very thick, color=bfonce] (3., 0.) -- (7., 0.) -- (7, 2.) -- (3, 2.);
\draw[thick, color=bfonce] (3.02, 0.) -- (3.02, 2.);
\node[color=bfonce] at (6.7, 0.3){$L$};

\draw[very thick, color=bfonce, ->] (3, 1.) -- (4, 1.) node[midway,above]{$\bfn_{K,\edge}$};

\draw[very thick, dotted, color=black] (-0.5, 0.) -- (0., 0.); \draw[very thick, dotted, color=black] (0., -0.5) -- (0., 0.); 
\draw[very thick, dotted, color=black] (7., 0.) -- (7.5, 0.);  \draw[very thick, dotted, color=black] (7., -0.5) -- (7., 0.); 
\draw[very thick, dotted, color=black] (-0.5, 2.) -- (0., 2.); \draw[very thick, dotted, color=black] (0., 2.) -- (0., 2.5);  
\draw[very thick, dotted, color=black] (7., 2.) -- (7.5, 2.); \draw[very thick, dotted, color=black] (7., 2.) -- (7., 2.5); 
\draw[very thick, dotted, color=black] (3., -0.5) -- (3., 0.); 
\draw[very thick, dotted, color=black] (3., 2.) -- (3., 2.5); 
\end{tikzpicture}
}
\hspace{2mm}
\scalebox{0.85}{
\begin{tikzpicture}[scale=1] 
\fill[color=green!50!white,opacity=0.3] (1.5, 0.) -- (3., 0.) -- (3., 2.) -- (1.5, 2.) -- (1.5, 0.); \node[color=vertf!70!black] at (2.3, 1.6){$D_{K,\edge}$};
\fill[color=green!20!white,opacity=0.3] (3., 0.) -- (5., 0.) -- (5., 2.) -- (3., 2.) -- (3., 0.); \node[color=vertf] at (3.8, 1.6){$D_{L,\edge}$};
\draw[very thick, color=vertf] (1.5, 0.04) -- (5., 0.04) -- (5., 1.96) -- (1.5, 1.96) -- (1.5, 0.04);
\draw[very thick, color=vertf] (5., 0.04) -- (5., 1.96) node[midway, sloped, right, anchor=south]{$\edged=\edge|\edge'$};
\draw[very thick, color=black, ->] (5., 1.) -- (6, 1.) node[midway, anchor=south]{$\bfn_{\edge,\edged}$};

\draw[very thick, color=black] (0., 0.) -- (7., 0.); \draw[very thick, color=black] (0., 2.) -- (7., 2.);
\draw[very thick, color=black] (0., 0.) -- (0., 2.);
\draw[very thick, color=black] (3., 0.) -- (3., 2.) node[midway, sloped, left, anchor=south]{$\edge$};
\draw[very thick, color=black] (7., 0.) -- (7., 2.) node[midway, sloped, right, anchor=north]{$\edge'$};

\draw[very thick, dotted, color=black] (-0.5, 0.) -- (0., 0.); \draw[very thick, dotted, color=black] (0., -0.5) -- (0., 0.); 
\draw[very thick, dotted, color=black] (7., 0.) -- (7.5, 0.);  \draw[very thick, dotted, color=black] (7., -0.5) -- (7., 0.); 
\draw[very thick, dotted, color=black] (-0.5, 2.) -- (0., 2.); \draw[very thick, dotted, color=black] (0., 2.) -- (0., 2.5);  
\draw[very thick, dotted, color=black] (7., 2.) -- (7.5, 2.); \draw[very thick, dotted, color=black] (7., 2.) -- (7., 2.5); 
\draw[very thick, dotted, color=black] (3., -0.5) -- (3., 0.); 
\draw[very thick, dotted, color=black] (3., 2.) -- (3., 2.5); 
\node[color=black] at (0.3, 0.3){$K$}; \node[color=black] at (6.7, 0.3){$L$};
\end{tikzpicture}
}
\caption{Notations for the primal and dual meshes -- Left: primal mesh -- Right: dual mesh for the first component of the velocity.} \label{fig:mesh}
\end{figure}

The size $\delta_\mesh$ of the mesh and its regularity $\theta_\mesh$ are defined by:
\begin{equation}  \label{regmesh}
\delta_\mesh=\max_{K\in\mesh }  \diam(K), \mbox{ and }
\theta_\mesh = \max_{K \in \mesh} \frac{\mathrm{diam}(K)^2}{|K|},
\end{equation}
where, here and in the following, $|\cdot|$ stands for the one or two dimensional measure of a subset of $\xR^2$.
Note that $\theta_\mesh$ is controlled by the maximum value taken by the ratio $|\edge|/|\edge'|$, with $\edge$ and $\edge'$ two edges of a same cell (this maximum is of course obtained when $\edge$ and $\edge'$ are not normal to the same vector of the canonical basis of $\xR^2$).
Since the grid is rectangular, a power (equal to 2 if $\Omega$ is a rectangle) of this maximum ratio in turns controls the maximum value of $|K|/|L|$, with $K$ and $L$ two cells of the mesh; so supposing that $\theta_\mesh$ is bounded is equivalent to formulate a quasi-uniformity condition for the mesh.

\begin{remark}[One-dimensional case]
In the one-dimensional case, these arguments do not hold.
For the proofs of consistency of the scheme presented in the following, we would need to use the additional regularity parameter:
\[
\theta'_\mesh = \max \bigl\{ \frac{|K|}{|L|},\ (K, L) \in \mesh^2,\ K \mbox{ and } L \mbox{ adjacent} \bigr\}.
\]
\end{remark}

For the sake of simplicity, we consider a uniform discretisation $0 = t_0 < t_1 < \cdots < t_N = T$ of the time interval $(0,T)$, and denote the (constant) time step by $\delta t = t_{n+1} - t_n$ for $n = 0, 1, \cdots, N - 1$.
%
%
\subsection{The segregated forward Euler scheme}   \label{sec:Euler-sch}

We first present a first order segregated discretisation in time and MAC discretisation in space of the system \eqref{eq:sw}, with a MUSCL-like technique for the computation of the numerical fluxes, see \cite{pia-13-for}; the scheme is written in compact form as follows:
\begin{subequations}\label{euler:scheme}
\begin{align} \nonumber &
\mbox{{\bf Initialisation}:}
\\ \label{sch-init_h} & \qquad
h^0_K = \frac 1 {|K|} \int_K h_0(\bfx) \dx, \quad p^0_K = \frac 1 2 g\, (h^0_K)^2,\quad \forall K \in \mesh,
\\[1ex] \label{sch-init_u} & \qquad
u^0_{i,\edge} = \frac 1 {|D_\edge|} \int_{D_\edge} u_{i,0} (\bfx) \dx, \quad \forall \edge \in \edgesinti,\mbox{ for } i=1,2.
\\[2ex] \nonumber &
\mbox{\bf For } 0 \leq n \leq N-1, \mbox{ solve for } h^{n+1},\ p^{n+1} \mbox{ and } \bfu^{n+1} = (u_i^{n+1})_{i=1,2}:
\\[1ex] \label{sch-mass} & \hspace{6ex}
\eth_t h_K^n +  \dive_K\,(h^n \bfu^n) = 0, \quad \forall K \in \mesh,
\\[1ex] \label{sch-pressure} & \hspace{6ex}
p_K^{n+1} = \frac 1 2 g\, (h_K^{n+1})^2, \quad \forall K \in \mesh,
\\[1ex] \label{sch-vitesse} & \hspace{6ex}
\eth_t(h u_i)_\edge^n + \dive_{D_\edge} (h^n u_i^n \bfu^n) + \eth_\edge p^{n+1} + g\, h_{\edge,c}^{n+1} \ \eth_\edge z =0, \quad  \forall \edge \in \edgesinti, \mbox{ for } i=1,2,
\end{align}
\end{subequations}
where the different discrete terms and operators introduced here are now defined.

\paragraph{\bf Discrete time derivative of the height.}
The term $\eth_t h^n_K$ is the discrete time derivative of the fluid height in the cell $K$ and over the time interval $(t_n,t_{n+1})$:
\[
\eth_t h^n_K = \frac 1 {\delta t}\ (h^{n+1}_K - h_K^n).
\]

\paragraph{\bf Discrete divergence and gradient operators of scalar unknowns.}
The discrete divergence operator on the primal mesh denoted by $\dive_K$ is defined as follows:
\begin{equation} \label{eq:div} 
\dive_K\, (h \bfu) = \frac{1}{|K|} \ \sum_{\edge\in\edges(K)} |\edge|\ \bfF_\edge \cdot \bfn_{K,\edge},
\end{equation}
where $\bfF_\edge$ is assumed to vanish on the external edges (thanks to the assumed impermeability condition) and, for an internal edge,
\[
\bfF_\edge = h_\edge  \ \bfu_\edge  \mbox{ and }\bfu_\edge = u_{i,\edge}  \ \bfe\ei   \mbox{ for } \edge \in \edges\ei,\ i=1,2.
\]
The value of $h$ at the edge, $h_\edge$, is approximated by a MUSCL-like interpolation technique \cite{pia-13-for}; in the subsequent analysis, we do not need to have an explicit formula for $h_\edge$, we only need the following conditions to be satisfied:
\begin{align} \nonumber &
\forall\ K \in \mesh, \ \forall \edge = K|L \in \edgesint(K),
\\[1ex]	\label{muscl:cons} & \qquad
- \exists\ \lambda_\edge^K \in [0, 1]:\  h_\sigma = \lambda_\edge^K\, h_K + (1-\lambda_\edge^K)\ h_L.
\\ \label{muscl:h:def}	&  \qquad
-  \exists\ \alpha_\edge^K \in [0, 1]  \mbox{ and } M_\edge^K \in \mesh:\   \ h_\edge - h_K =
\left\{ \begin{array}{rcll}
\alpha_\edge^K\ (h_K- h_{M_\edge^K})  & \text{ if } \bfu_\edge \cdot \bfn_{K,\edge} \geq 0,
\\[1ex]
\alpha_\edge^K\ (h_{M_\edge^K} - h_K) & \text{ otherwise}.
\end{array} \right.
\end{align}
By \eqref{muscl:cons}, $h_\edge$ is a convex combination of $h_K$ and $h_L$;  if $\bfu_\edge \cdot \bfn_{K,\edge} < 0$,  the cell $M_\edge^K$ in \eqref{muscl:h:def} can be chosen as $L$ and $\alpha_\edge^K$ as $1 - \lambda_\edge^K$. 
In the case of a discrete divergence free velocity field $\bfu$, this assumption ensures that $h_K^{n+1}$ is a convex combination of the values $h_K^n$  and $(h_M^n)_{M \in \mathcal N_\edge^K}$, where $\mathcal N_\edge^K$ denotes the set of cells $M_\edge^K$ satisfying \eqref{muscl:h:def}, see \cite[Lemma 3.1]{pia-13-for}, for any structured or unstructured mesh.

In practice, there are several ways to choose the value $h_\edge$ so as to satisfy the conditions \eqref{muscl:cons}-\eqref{muscl:h:def}.
For instance in one space dimension and for a uniform mesh (or for uniform Cartesian meshes), if $\edge' = J|K$, $\edge = K|L$ and $\edge'$ and $\edge$ are opposite edges of $K$, with $\bfu_\edge \cdot \bfn_{K,\edge} \ge 0$, the cell $M_\edge^K$ in Relation \eqref{muscl:h:def} can be chosen as the cell $J$ and the value $h_\edge$ may be computed using the following classical Van Leer limitation procedure \cite{van-77-tow}:
\begin{multline*}
h_\edge - h_K = \frac 1 2\ \mathrm{minmod}\ \big(\frac 1 2\, (h_L - h_J),\ \zeta^+\, (h_L -h_K),\ \zeta^{-}\,  (h_K -h_J) \big),
\\
\text{with  } \mathrm{minmod} (a, b, c) =
\begin{cases} \mathrm{sgn}(a) \min(|a|,|b|,|c|) \mbox{ if } a, b, c \mbox{ have the same sign,} \\  0 \mbox{ otherwise.} \end{cases}
\end{multline*}
where the limitation parameters $\zeta^{+}$ and $\zeta^{-}$ are such that $\zeta^{+}, \zeta^{-} \in [0, 2]$.
Observe that, for the discretisation to be (quasi) second order, these parameters must be such that $\zeta \ge 1 - \varepsilon(\delta_\mesh)$ with $\varepsilon(\delta_\mesh) \to 0$ as $\delta_\mesh \to 0$; if $\zeta^{+}=\zeta^{-} = 1$, the two-slopes $\mathrm{minmod}$ limiter ($\mathrm{minmod} (h_L -h_K,  h_K -h_J )$) is recovered.

The numerical tests which are presented in Section \ref{sec:num} below are performed with the freeware code CALIF$^3$S implementation which is designed for any kind of mesh, Cartesian or structured, 2D and 3D, see \cite{califs} for more details;
for Cartesian meshes, the choice of $h_\edge$ with $\edge=K|L$ and $K$ the upwind cell to $\edge$ is the following:
\begin{equation} \label{implementation-calif}
\begin{array}{ll}
i)
&
\mbox{Choose an affine interpolation  }\tilde h_\edge \mbox{ at the mass center of the interface, using the lowest}
\\ &
\mbox{possible number of neighbours of } K,
\\[1ex]
ii)
\quad  & \displaystyle
\mbox{Compute } h_\edge - h_K = \frac 1 2 \mathrm{minmod}\ \bigl(2(\tilde h_\edge -h_K),\ 2(h_K-h_J)\bigr),
\end{array}
\end{equation}
where $J$ is a suitable neighbour of $K$, chosen as the opposite cell to $\edge$ with respect to $K$ for quadrilateral or hexahedric cells.

A local discrete derivative applied to a discrete scalar field $\xi$ (with $\xi = p, h$ or $z$) is defined by:
\begin{equation} \label{eq:grad}
\eth_\edge \xi = \frac{|\edge|}{|D_\edge|}\ (\xi_L - \xi_K) \, \text{ for } \edge = K|L \in \edgesinti, \mbox{ with } (\bfx_K)_i < (\bfx_L)_i, \mbox{ for } i=1,2.
\end{equation}
The above defined discrete divergence and discrete derivatives satisfy the following div-grad duality relationship \cite[Lemma 2.4]{gal-18-conv}:
\begin{equation}\label{eq:div-grad}
\sum_{K\in\mesh} |K|\ \xi_K\ \dive_K(h \bfu) + \sum^2_{i=1}\ \sum_{\edge \in \edgesinti} |D_\edge|\ h_\edge\ u_{i,\edge}\ \eth_\edge \xi = 0.
\end{equation}


\paragraph{\bf Discrete water height for the bathymetry term.}
In equation \eqref{sch-vitesse} the term $\eth_\edge z$ denotes the discrete derivative (in the sense of \eqref{eq:grad}) of the piecewise constant function $z_\mesh = \sum_{K\in \mesh} z(\bfx_K) \characteristic_K$ (with $\characteristic_K(\bfx)=1$ if $\bfx \in K$ and 0 otherwise), that is:
\begin{equation} \label{gradz}
\eth_\edge z = \frac{|\edge|}{|D_\edge|}\ (z(\bfx_L) - z(\bfx_K)) \, \text{ for } \edge = K|L \in \edgesinti, \mbox{ with } (\bfx_K)_i < (\bfx_L)_i, \mbox{ for } i=1,2.
\end{equation}
The value $h_{\edge,c}$ of the water height is defined so as to satisfy:
\begin{equation}\label{p-lake}
\eth_\edge p + g\, h_{\edge,c} \ \eth_\edge z  = 0 \mbox{ if } \eth_\edge(h+z)=0, \mbox{ for }i =1, 2.
\end{equation}
This requirement is fulfilled if $h_{\edge,c}$ is centered, {\it i.e.} if $h_{\edge,c}$ is defined by:
\begin{equation} \label{hedge}
h_{\edge,c} = \frac 1 2\ (h_K+h_L), \text{ for } \edge = K|L \in \edgesint.
\end{equation}
Indeed, if $h_{\edge,c}$ is defined by \eqref{hedge}, since $p_K = \frac 1 2 g h_K^2$ for $K\in\mesh$, one has from the definition of the discrete gradient \eqref{eq:grad}, for $\edge= K|L$,
\[
\eth_\edge p + g\, h_{\edge,c} \ \eth_\edge z = \frac 1 2\ g\ (h_K + h_L) \eth_\edge(h+z),
\]
and therefore \eqref{p-lake} holds, so that the ``lake at rest" steady state is preserved, see Lemma \ref{lem:lake} below.


\paragraph{\bf Discrete momentum convection operator.}
The discrete time derivative $\eth_t(h u_i)_\edge^n$ in Equation \eqref{sch-vitesse} is defined, for $\edge = K|L \in \edgesinti$, by
\[
\eth_t(h u_i)_\edge^n = \frac 1 {\delta t}\ \bigl((h \ u_i)_\edge^{n+1} -(h \ u_i)_\edge^n \bigr)
\]
where
\begin{equation} \label{discrete_water_dual}
(h \ u_i)_\edge^{n+1} = h_{D_\edge}^{n+1} \ u_{i,\edge}^{n+1}, \quad \mbox{with} \quad
h_{D_\edge} = \frac 1 {|D_\edge|}\ \bigl( |D_{K,\edge}|\ h_K + |D_{L,\edge}|\ h_L \bigr).
\end{equation}
The discrete divergence operator on the dual mesh $\dive_{D_\edge}$ is given by:
\begin{equation} \label{eq:conv}
\dive_{D_\edge}(h u_i \bfu) = \frac{1}{|D_\edge|}\ \sum_{\edged \in \edgesd(D_\edge)} |\edged|\ \bfG_{\edged} \cdot \bfn_{\edge,\edged}, \text{ with } \bfG_{\edged} = \bfF_{\edged}\ u_{i,\edged},
\end{equation}
where
\begin{list}{--}{\itemsep=0.ex \topsep=0.5ex \leftmargin=1.cm \labelwidth=0.7cm \labelsep=0.3cm \itemindent=0.cm}
\item the flux $\bfF_{\edged}$ is computed from the primal numerical mass fluxes; following \cite{her-10-kin} (see also \cite{ans-11-anl, her-18-con} for an extension to triangular or quadrangular meshes using low order non-conforming finite element), it is defined as follows (see Figure \ref{fig:convection} for the notations):
\begin{subequations} \label{discrete_dual_flux}
\begin{align} \label{dual_flux_edge_paral} &
\text{for } \edged =\edge|\edge',\ \edged \subset K,
&&
\bfF_{\edged} = \frac 1 2\ \bigl(\bfF_\edge +  \bfF_{\edge'} \bigr),
\\ \label{dual_flux_edge_perp} &
\text{for } \edged =\edge|\edge',\ \edged \not \subset K,\ \edged \subset \edgeperp \cup \edgeperp',
&&
\bfF_{\edged} = \frac 1 {|\edged|}\ \bigl(\frac 1 2 |\edgeperp|\ \bfF_{\edgeperp} + \frac 1 2 |\edgeperp'|\ \bfF_{\edgeperp'} \bigr).
\end{align}
\end{subequations}

\begin{figure}[htbp]
\centering
\begin{tikzpicture}[scale=1]
\fill[color=bclair,opacity=0.3] (1.5, 0.)--(3, 0.)--(3, 2)--(1.5, 2)--cycle;
\fill[color=orangec!30!white,opacity=0.3] (3, 0.)--(4.5, 0.)--(4.5, 2.)--(3., 2)--cycle;

\draw[very thin] (1., 0.0)--(5,0.);
\draw[very thin] (1., 2.)--(5, 2.);
\draw[very thin] (2, -0.5)--(2, 2.5);
\draw[very thin] (4, -0.5)--(4, 2.5);

\path node at (1.2, 0.5) [anchor= west]{$J$};
\path node at (2.5, 0.5) [anchor= west]{$K$};
\path node at (4., 0.5) [anchor= west]{$L$};
\path node at (2.2, 1) {$\sigma'$};
\path node at (4.2, 1) {$\sigma$};

\draw[very thick, red] (3., 0)--(3, 2.);
\path[very thick, red] node at (3.1, 1.) {$\epsilon$};

\fill[color=bclair,opacity=0.3] (6.75, 0.)--(8.25, 0.)--(8.25,1)--(6.75, 1)--cycle;
\fill[color=orangec!30!white,opacity=0.3] (6.75, 1.)--(8.25, 1.)--(8.25,2.)--(6.75, 2)--cycle;

\draw[very thin] (5.5, 0.0)--(9.5,0.);
\draw[very thin] (5.5, 2.)--(9.5, 2.);
\draw[very thin] (6., -0.5)--(6., 2.5);
\draw[very thin] (9, -0.5)--(9, 2.5);
\draw[very thin] (5.5, 1.)--(9.5, 1.);
\draw[very thin] (7.5, -0.5)--(7.5, 2.5);

\path node at (6.25, 0.5) {$K$};
\path node at (8.75, 0.5) {$L$};
\path node at (6.25, 1.5) {$M$};
\path node at (8.75, 1.5) {$N$};
\path node at (7.59, 0.5) {$\sigma$};
\path node at (7.65, 1.6) {$\sigma'$};
\path node at (6.75, 1.2) {$\tau$};
\path node at (8.3, 1.2) {$\tau'$};
\path[very thick, red] node at (7.6, 1.2) {$\epsilon$};
\draw[very thick, red] (6.75, 1)--(8.25,1.);
\end{tikzpicture}
\caption{Notations for the definition of the momentum flux on the dual mesh for the first component of the velocity - Left: $\edged \subset K$ - Right: $\edged \subset \edgeperp \cup \edgeperp'$.}
\label{fig:convection}
\end{figure}
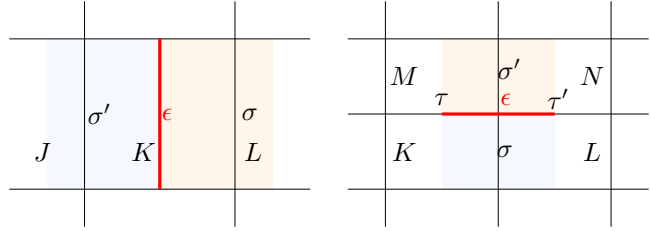

\item the value $u_{i,\edged}$ is expressed in terms of the unknowns $u_{i,\edge}$, for $\edge \in \edges\ei$, again by a second order MUSCL-like interpolation scheme; the values $u_{i,\edge}$ are thus assumed to satisfy the following property:
\begin{align} \nonumber &
\forall\ \edge \in \edgesinti,\ i =1,2,\ \forall\ \edged=\edge|\edge' \in \edgesd(D_\edge),
\\  \label{muscl:u:cons} & \quad
\begin{array}{l}
u_{i,\edged}  \mbox{ is a convex combination of } u_{i,\edge} \mbox{ and } u_{i,\edge'}:
\\[1ex] \hspace{28ex}
\exists \mu^\edge_\edged \in [0,1] :\ u_{i,\edged} = \mu^\edge_\edged\ u_{i,\edge} + (1 - \mu^\edge_\edged)\ u_{i,\edge'},
\end{array}
\\[2ex] \label{muscl:u:def}	& \quad
\begin{array}{l}
\exists \ \alpha_\edged^\edge \in [0, 1] \text{ and } \tau_\edged^\edge  \in \edgesinti :\ u_{i,\edged} - u_{i,\edge} =
\left\{\begin{array}{ll}
\alpha_\edged^\edge\ (u_{i,\edge} - u_{i,\tau_\edged^\edge}) & \text{if } \bfF_{\edged} \cdot \bfn_{\edge,\edged} \geq 0,
\\
\alpha_\edged^\edge\ (u_{i,\tau_\edged^\edge} - u_{i,\edge})  & \text{otherwise}.
\end{array} \right. \end{array}
\end{align}
Again note that in the case $\bfF_{\edged} \cdot \bfn_{\edge,\edged} <0$, the edge $\tau_\edged^\edge$ may be chosen as $\edge'$, and this is the choice made in CALIF$^3$S, the free software used in the numerical tests.
Still for the numerical tests, as for the scalar unknowns, when $\bfF_{\edged} \cdot \bfn_{\edge,\edged} \geq 0$, the choice for $\tau_\edged^\edge$ is restricted to a single dual cell, which is this opposite dual cell to $D_\edge$ with respect to $D_{\edge'}$, and the assumption that the coefficients $\mu^\edge_\edged$ and $\alpha_\edged^\edge$ lie in the interval $[0,1]$ yields the usual two-slopes minmod limiter.
\end{list}
Let us emphasize that, owing to the definitions \eqref{discrete_water_dual} and \eqref{discrete_dual_flux}, the following discrete mass balance version on the dual mesh holds:
\begin{equation} \label{mass_dual}
\dfrac{|D_\edge|}{\delta t}\ (h_{D_\edge}^{n+1} - h_{D_\edge}^n) + \sum_{\edged \in \edgesd(D_\edge)} |\edged|\ \bfF_{\edged}^n  \cdot \bfn_{\edge,\edged} = 0.
\end{equation}
%
%
\subsection{A second order in time Heun scheme} \label{sec:Heun-sch}

We retain here the quasi-second order space discretisation which we just set up, but consider now  a second order time discretisation using the Heun (or Runge-Kutta 2) scheme.
The initialization of the scheme is the same as that of the forward Euler scheme, see Equations \eqref{sch-init_h}-\eqref{sch-init_u}, but the $n$-th step now reads:
\begin{subequations}\label{heun:scheme}
\begin{align} \nonumber &
\mbox{{\bf Step} $n$ :}  \text{ For } h^n \text{ and } \bfu^n =(u_i^n)_{i=1, 2} \text{ known},
\\[2ex]	\label{heun:first_mass} & \hspace{10ex} \displaystyle
\widehat{h}_K^{n+1}  =  h_K^n - \delta t \, \dive_K (h^n  \bfu^n), && \forall K \in \mesh,
\\ \label{heun:first_mom} & \hspace{10ex} \displaystyle
\widehat{h}_{D_\edge}^{n+1}   \ \widehat{u}_{i,\edge}^{n+1}  =  h_{D_\edge}^n u_{i,\edge}^n - \delta t \, \mathcal{F}_{D_\edge} (h^n, u_i^n), && \forall \edge \in \edgesinti,\ i=1,2,
\displaybreak[1] \\[2ex] \label{heun:second_mass} & \hspace{10ex} \displaystyle
\tilde{h}_K^{n+1} =  \widehat{h}_K^{n+1}  - \delta t \, \dive_K (\widehat{h}^{n+1}  \widehat{\bfu}^{n+1} ), && \forall K \in \mesh,
\\ \label{heun:second_mom} & \hspace{10ex} \displaystyle
\tilde{h}_{D_\edge}^{n+1} \ \tilde{u}_{i,\edge}^{n+1} =  \widehat{h}_{D_\edge}^{n+1}  \widehat{u}_{i,\edge}^{n+1}  - \delta t \, \mathcal{F}_{D_\edge} (\widehat{h}^{n+1} , \widehat{u}_i^{n+1} ),
&& \forall \edge \in \edgesinti,\ i=1,2
\displaybreak[1] \\[2ex] \label{heun:last_mass} & \hspace{10ex}
h_K^{n+1} =  \frac{1 }{2}\, (h_K^n + \tilde{h}_K^{n+1} ), && \forall K \in \mesh,
\\  \label{heun:last_mom} & \hspace{10ex}
h_{D_\edge}^{n+1} \ u_{i,\edge}^{n+1}  =  \frac{1 }{2} \big(h_{D_\edge}^n u_{i,\edge}^n + \tilde{h}_{D_\edge}^{n+1} \ \tilde{u}_{i, \edge}^{n+1}  \big ),  && \forall \edge \in \edgesinti,\ i=1,2,
\end{align}
\end{subequations}
where
\begin{equation} \label{totalflux}
\mathcal{F}_{D_\edge} (h, u_i) = \dive_{D_\edge} (h u_i \bfu) + g\, h_{\edge,c}\ ( \eth_\edge h + \eth_\edge z)
\end{equation}
and the dual cell values $\widehat{h}_{D_\edge}^{n+1}, $  $\tilde{h}_{D_\edge}^{n+1}$ and  ${h}_{D_\edge}^{n+1}$ are computed from the corresponding cell values by the analogue of the formula \eqref{discrete_water_dual}; since, in equations \eqref{heun:first_mom} and \eqref{heun:second_mom}, the convection operator is derived from the associated mass balance in the same way as in the Euler scheme, a dual mass balance of the type \eqref{mass_dual} is satisfied for these two steps.
The steps \eqref{heun:second_mass}-\eqref{heun:last_mom} of the above scheme  \eqref{heun:scheme} may be replaced by the more compact form
\begin{subequations}
\begin{align} \label{heun:equiv1} & \quad \displaystyle
\eth_t h_K^n = -\frac 1 2\ \Big(\dive_K (h^n  \bfu^n) + \dive_K (\widehat{h}^{n+1}  \widehat{\bfu}^{n+1} ) \Big), && \forall K \in \mesh,
\\[1ex] \label{heun:equiv2} & \quad \displaystyle
\eth_t (h_{D_\edge} \ u_i)^n_\edge = - \frac{1}{2}  \, \Big(\mathcal{F}_{D_\edge} (h^n, u_i^n) + \mathcal{F}_{D_\edge} (\widehat{h}^{n+1} , \widehat{u}_i^{n+1} ) \Big), && \forall \edge \in \edgesi,\ i=1,2,
\end{align}
\end{subequations}
where the dual cell values $(h_{D_\edge}^{n+1})$ are computed from the primal cells values $(h_K^{n+1})$ by the formula \eqref{discrete_water_dual}; hence, once again, a dual mass balance of the type \eqref{mass_dual} is satisfied with the half-sum of the mass fluxes used in the two terms of the momentum convection operator of Equation \eqref{heun:equiv2}.
%
%
\section{Stability of the schemes}\label{sec:stab}
 
The positivity of the water height under a CFL like condition is ensured by both the schemes \eqref{euler:scheme} and \eqref{heun:scheme}; it is a consequence of the property \eqref{muscl:h:def} of the MUSCL choice for the interface values.
Indeed, the proof of the positivity in \cite[Lemma 3.1]{pia-13-for} remains valid even if the discrete velocity field is not divergence free, as is the case here.

\begin{lemma}[Positivity of the water height]\label{lem:positivity}
Let $n \in \llbracket 0, N-1 \rrbracket$, let $(h_K^n)_{K \in \mesh} \subset \xR_+^\ast$ and $(\bfu_\edge^n)_{\edge \in \edges} \subset \xR^2$ be given, and let $h_K^{n+1}$ be computed by the forward Euler scheme, step \eqref{sch-mass}. 
Then $h_K^{n+1} > 0$, for all $K \in \mesh$ under the following CFL condition,
\begin{equation} \label{cfl:posit:euler}
2 \delta t  \displaystyle \sum_{\edge \in \edges(K)} |\edge|\ |\bfu^n_\edge \cdot \bfn_{K,\edge}| \leq \ |K|, \quad \forall K \in \mesh.
\end{equation}
If \eqref{cfl:posit:euler} is fulfilled and if furthermore
\begin{equation} \label{cfl:posit:heun}
2\delta t {\displaystyle \sum_{\edge \in \edges(K)} |\edge|\ |\widehat{\bfu}^{n+1} _\edge \cdot \bfn_{K,\edge} |}\leq   |K|, \quad \forall K \in \mesh,
\end{equation}
then $h_K^{n+1}$ computed by the Heun scheme \eqref{heun:scheme} is positive.
\end{lemma}

Secondly, thanks to the choice \eqref{hedge} for the reconstruction of the water height, the property \eqref{p-lake} holds, so that the so-called "lake at rest" steady state is preserved by both schemes.

\begin{lemma}[Steady state "lake at rest"]\label{lem:lake}
Let $n \in \llbracket 0, N-1 \rrbracket$, $C \in \xR_+$; let $(h_K^n)_{K \in \mesh} \subset \xR$ be such that $h^n_K+z_K=C$ for all $K\in \mesh$ and $(\bfu_\edge^n)_{\edge \in \edges}$ be such that $\bfu_\edge^n=0$ for all $\edge \in \edges$.
Then the solution $(h_K^{n+1})_{K \in \mesh}$, $(\bfu_\edge^{n+1})_{\edge \in \edges}$ of the forward Euler scheme \eqref{euler:scheme}  (resp. Heun scheme \eqref{heun:scheme}) satisfies $h_K^{n+1}+z=C$ for all $K\in \mesh$ and $\bfu_\edge^{n+1}=0$ for $\edge \in \edges$.
\end{lemma}

As a consequence of the careful discretisation of the convection term, the segregated forward Euler scheme satisfies a discrete kinetic energy balance, as stated in the following lemma. 
The proof of this result is an easy adaptation of \cite[Lemma 3.2]{her-18-con}.

\begin{lemma}[Discrete kinetic energy balance, forward Euler scheme] \label{lem:disc_kin}
Let, for $n \in \llbracket 0, N \rrbracket$, $\edge\in\edgesi$ and $i=1,2$, the $i$-th part of the kinetic energy be defined by $(E_{k,i})_\edge^n=\frac 1 2\, h_{D_\edge}^n\, (u^n_{i,\edge})^2$.
The solution to the scheme \eqref{euler:scheme} satisfies the following equality, for $i=1,2$, $\edge \in \edges\ei$ and $n \in \llbracket 0, N-1 \rrbracket$:
\begin{equation} \label{disc-kinet}
(\eth_t E_{k,i})_\edge^n
+ \dfrac{1}{2\ |D_\edge|} \sum_{\edged \in \edgesd(D_\edge)} |\edged| \ (u^n_{i,\edged})^2  \  \bfF^n_{\edged} \cdot \bfn_{\edge,\edged}
+ u_{i,\edge}^{n+1}\ \eth_\edge p^{n+1}
+ g\, h_{\edge,c}^{n+1} \, u_{i,\edge}^{n+1}\ \eth_\edge z = - R_{i,\edge}^{n+1},
\end{equation}
with $(\eth_t E_{k,i})_\edge^n=\dfrac 1 {\delta t}\,\bigl((E_{k,i})_\edge^{n+1}-(E_{k,i})_\edge^n\bigr)$ and
\begin{multline*}
R_{i, \edge}^{n+1} = \dfrac 1 {2\,\delta t} \ h_{D_\edge}^{n+1} \ \big(u^{n+1}_{i,\edge} - u^n_{i,\edge}\big)^2
- \frac 1 {2\ |D_\edge|} \sum_{\edged \in \edgesd(D_\edge)} |\edged| \ \bfF_\edged^n \cdot \bfn_{\edge,\edged} \ \big(u^n_{i,\edged} - u_{i,\edge}^n \big)^2
\\
+ \frac 1 {|D_\edge|}\sum_{\edged \in \edgesd(D_\edge)} |\edged| \ \bfF_\edged^n \cdot \bfn_{\edge,\edged} \ \big(u^n_{i,\edged} - u_{i,\edge}^n \big) \big(u^{n+1}_{i,\edge} - u_{i,\edge}^n \big).
\end{multline*}
\end{lemma}

The scheme also satisfies the following discrete potential energy balance.

\begin{lemma}[Discrete potential energy balance, forward Euler scheme] \label{lem:disc-pot}
Let the local (in space and time) discrete potential energy be defined by $(E_p)_K^n=\frac 1 2 g\,(h_K^n)^2 + g h_K^n z_K$, for $K \in \mesh$ and $n \in \llbracket 0, N \rrbracket$.
The solution to the scheme \eqref{euler:scheme} satisfies the following equality, for $K \in \mesh$ and $n \in \llbracket 0, N-1 \rrbracket$:
\begin{equation}\label{disc-pot}
(\eth_t E_p)_K^n + \dive_K(\frac 1 2 g (h^n)^2 \bfu^n ) + g\, z_K\ \dive_K(h^n \bfu^n ) +  p_K^n\ \dive_K (\bfu^n) = - r_K^{n+1},
\end{equation}
with $(\eth_t E_p)_K^n = \dfrac 1 {\delta t} \bigl((E_p)_K^{n+1} - (E_p)_K^n\bigr)$, $\displaystyle \dive_K(\frac 1 2 g (h^n)^2 \bfu^n ) = \frac 1 {|K|} \sum_{\edge \in \edges(K)} |\edge|\ \frac 1 2  g (h^n_\edge)^2\ \bfu_\edge \cdot \bfn_{K,\edge}$ and
\begin{multline}\label{residu:pot:bound}
r_K^{n+1} = \frac{1}{2\,\delta t} g (h_K^{n+1} - h_K^n)^2
- \frac g {2\,|K|}  \sum_{\edge \in \edges(K)} |\edge|\ (h^n_\edge - h_K^n)^2 \  \bfu_\edge^n \cdot \bfn_{K,\edge}
\\
+ \frac g {|K|} \sum_{\edge \in \edges(K)} |\edge|\ (h^{n+1}_K - h_K^n) \ h_\edge^n \ \bfu_\edge^n \cdot \bfn_{K,\edge}.
\end{multline}
\end{lemma}
\begin{proof}
Applying \cite[Lemma A1]{her-18-con}, (re-stated in Lemma \ref{lem:A1} below for the sake of completeness), with $P=K$, 	$\psi : x \mapsto \frac 1 2 g  x^2$ , $\rho_P = h_K^{n+1}$, $\rho_P^\ast = h_K^n$, $\eta=\sigma$, $\rho_\eta^\ast = h_\edge^n$ and $V_\eta^\ast = |\sigma| \bfu_\sigma^n \cdot \bfn_{K, \sigma}$, and  $ R_K^{n+1}=|K|\ r_K^{n+1}$, we get that
\begin{multline*}
\frac g 2\, (\eth_t h^2)_K^n + \dive_K(\frac g 2 (h^n)^2 \bfu^n ) + p_K^n\ \dive_K (\bfu^n) = - \frac  g {2\delta t} (h_K^{n+1} - h_K^n)^2
\\
+ \frac 1 {|K|} \frac g 2 \sum_{\edge \in \edges(K)} |\edge|\ (h^n_\edge - h_K^n)^2 \ \bfu_\edge^n \cdot \bfn_{K,\edge}
- \frac 1 {|K|} g \sum_{\edge \in \edges(K)} |\edge|\ (h^{n+1}_K - h_K^n) \   h_\edge^n \ \bfu_\edge^n \cdot \bfn_{K,\edge}.
\end{multline*}
Then, multiplying the discrete mass balance equation \eqref{sch-mass} by $g z_K$ yields
\[
g\, (\eth_t hz)_K^n + g\, z_K\ \dive_K(h^n\ \bfu) = 0.
\]
Summing the two above equations yields \eqref{residu:pot:bound}.
\end{proof}

Since the discrete kinetic and potential energies are computed on the dual and primal meshes respectively, deriving a discrete entropy inequality is not straightforward.
In \cite{her-19-dec}, a kinetic energy inequality on the primal cells is obtained from the inequality \eqref{bc:ins} to get a discrete local entropy inequality on the primal cells.
Here, we proceed differently to show that the first order in time scheme is entropy consistent: indeed, we pass to the limit in each discrete energy inequality on its respective mesh, see Section \ref{sec:entropy} below.
%
%
\section{LW-consistency of the schemes}\label{sec:cons}

We now wish to prove the consistency of the proposed schemes in the Lax-Wendroff sense (following the seminal paper \cite{lax-60-sys}), namely to prove that if a sequence of solutions is controlled in suitable norms and converges to a limit, the limit necessarily satisfies a weak formulation of the continuous problem.

The pair of functions $(\bar h, \bar \bfu) \in \xL^1(\Omega\times[0,T))\times \xL^1(\Omega\times[0,T))^2$ is a weak solution to the continuous problem if it satisfies, for any $\varphi \in C^\infty_c \bigl(\Omega \times [0,T)\bigr)$ and $\bfvarphi \in C^\infty_c \bigl(\Omega \times [0,T)\bigr)^2$:
\begin{subequations} \begin{align}\label{weak-mass} &
\int_0^T \int_\Omega \Bigl[ \bar h \, \partial_t \varphi +\bar h\,\bar u \, \cdot \gradi \varphi \Bigr]\dx \dt
+\int_\Omega h_0(\bfx) \, \varphi(\bfx,0) \dx= 0,
\\ \label{weak-mom} &
\int_0^T \int_\Omega \Bigl[\bar h\, \bar\bfu \cdot \partial_t \bfvarphi + (\bar h \bar\bfu \otimes \bar\bfu) : \gradi \bfvarphi
+ \frac 1 2\ g \,\bar h^{2} \dive \bfvarphi + g \, \bar h \, \gradi z \cdot \bfvarphi \Bigr] \dx \dt
\\ & \hspace{6cm} \nonumber
+ \int_\Omega h_0(\bfx) \, \bfu_0(\bfx)\cdot \varphi(\bfx,0) \dx = 0.
\end{align}\label{weak-form}
\end{subequations}
A weak solution of \eqref{weak-form} is an entropy weak solution if, for any nonnegative test function $\varphi \in C^\infty_c \bigl(\Omega \times [0,T), \mathbb R_+\bigr)$:
\begin{equation}\label{eq:weakentropy}
\int_0^T \int_\Omega \Bigl[\bar E \, \partial_t \varphi + {\bf \bar \Phi} \cdot \gradi \varphi \Bigr] \dx \dt + \int_\Omega E_0(\bfx)\, \varphi(\bfx,0) \dx \geq 0,
\end{equation}
with
\[
\bar E = \frac 1 2  \bar h\,|\bar\bfu|^2 + \frac 1 2 g \bar h^2 + g \bar h z,\quad E_0 = \frac 1 2  h_0\,|\bfu_0|^2 + \frac 1 2 g h_0^2 + g h_0 z \quad \text{and } \bar \Phi =(\bar E +  \frac 1 2 g \bar h^2)\, \bar \bfu.
\]

Let $(\mesh\m,\edges\m)_\mnn$ be a sequence of meshes in the sense of Definition \ref{def:MACgrid} and let ($h\m$, $\bfu\m)_\mnn$  be the associated sequence of solutions of the scheme \eqref{euler:scheme} defined almost everywhere on $\Omega \times [0, T)$ by:
\begin{equation} \label{eq:def_unk}
\begin{array}{l} \displaystyle
u_i\m(\bfx, t) = \sum_{n=0}^{N\m-1}\ \sum_{\edge \in \edges^{(m,i)}} (u_i\m)_\edge^n\ \characteristic_{D_\edge}(\bfx)\ \characteristic_{[t_n, t_{n+1})}(t),\text{ for } i=1,\ 2,
\\[2ex] \displaystyle
h\m(\bfx, t) =  \sum^{N\m-1}_{n=0}\ \sum_{K \in \mesh\m} (h\m)_K^n\ \characteristic_K(\bfx)\ \characteristic_{[t_n, t_{n+1})}(t),
\end{array}
\end{equation}
where $\characteristic_A$ is the characteristic function of a given set $A$, that is $\characteristic_A (y) = 1 $ if $y \in A$, $\characteristic_A (y) = 0 $ otherwise, and $\edges^{(m,i)}$ stands for the set of the edges of $\mesh\m$ orthogonal to $\bfe\ei$ (in other words, this notation replaces $(\edges\m)\ei$ for short).

\textbf{Assumed estimates }- Some boundedness and compactness assumptions on the sequence of discrete solutions $\displaystyle (h\m, \ \bfu\m)_{m \in \xN}$ are needed in order to prove the LW-consistency.
Here, we assume that:
\begin{list}{--}{\itemsep=0.ex \topsep=0.5ex \leftmargin=1.cm \labelwidth=0.7cm \labelsep=0.3cm \itemindent=0.cm}
\item  the water height $h\m$ and its reciprocal $1/h\m$ are uniformly bounded in $\xL^\infty(\Omega \times (0,T))$, {\it i.e.}\ there exists  $C^h  \in \xR_+^\ast$ such that for $m \in \xN$ and $0 \leq n < N\m$:
\begin{equation}\label{bound-h}
\frac 1 {C^h} \leq (h\m)^n_K \leq C^h, \quad \forall K \in \mesh\m,
\end{equation} 
\item the velocity $\bfu\m$ is also uniformly bounded in $\xL^\infty(\Omega \times (0,T))^2$, {\it i.e.}\  there exists  $C^u  \in \xR_+^\ast$ such that
\begin{equation}\label{bound-u}
|(\bfu\m)^n_\edge| \leq C^u, \quad \forall \edge \in \edges\m.
\end{equation}
\end{list}

 \begin{theorem}[LW-consistency of the schemes] \label{theo:cons-Euler}
Let $(\mesh\m,\edges\m)_\mnn$ be a sequence of meshes and $(\delta t\m)_\mnn$ be a sequence of time steps such that $\delta_{\mesh\m}$ and $\delta t\m$ tend to zero as $m \to +\infty$~; assume that there exists $\theta >0$ such that $\theta_{\mesh\m} \le \theta$ for any $m\in \xN$ (with $\theta_{\mesh\m}$ defined by \eqref{regmesh}).

Let $(h\m,\bfu\m)_\mnn$ be the associated sequence of solutions to the scheme \eqref{euler:scheme}, and suppose that $(h\m,\bfu\m)_\mnn$ satisfies \eqref{bound-h} and \eqref{bound-u}, and converges to $(\bar h, \bar \bfu)$ in $\xL^1(\Omega \times (0,T)) \times \xL^1(\Omega \times (0,T))^2$.
Then $(\bar h, \bar \bfu)$ satisfies the weak formulation \eqref{weak-form} of the SWE.

Similarly, if  $(h\m,\bfu\m)_\mnn,  (\widehat h\m,\widehat \bfu\m)_\mnn$ are the sequences of solutions to the scheme \eqref{heun:scheme} both uniformly bounded in the sense of \eqref{bound-h} and \eqref{bound-u} and converging to $(\bar h, \bar \bfu)$ in $\xL^1(\Omega \times (0,T)) \times \xL^1(\Omega \times (0,T))^2$, then the limit $(\bar h, \bar \bfu)$ satisfies  \eqref{weak-form}.
\end{theorem}

The proof of this theorem is the object of the following paragraphs.
It relies on some general consistency lemmas proven in \cite{gal-21-wea}, which generalize the results of \cite{gal-19-wea} to staggered meshes; for the sake of completeness, these results are recalled in the Appendix.
The proof of the consistency of the schemes is given in Section \ref{subsec:cons-euler} for the forward Euler time discretisation and in Section \ref{subsec:cons-heun} for the Heun time discretisation.

Note that, because the convergence and boundedness of the approximate solutions are assumed, no CFL condition is required in Theorem \ref{theo:cons-Euler}.
However, recall that a CFL condition is for instance already needed to show the positivity of the water height (Lemma \ref{lem:positivity}), which is assumed in the theorem.

Finally, note that the boundedness and convergence of the sequence $(\widehat h\m,\widehat \bfu\m)_\mnn$ may be proven to be a consequence of the boundedness and convergence of the sequence $(h\m, \bfu\m)_\mnn$ (and so allows to remove this convergence hypothesis from the assumptions of the theorem), under a CFL  condition which is only slightly more restrictive than the condition \eqref{cfl:posit:euler}.
This result may be found in \cite{nas-21-phd}.
%
%
\subsection{Proof of consistency of the forward Euler scheme}\label{subsec:cons-euler}

\subsubsection{Consistency, mass balance equation} \label{subsubsec:cons-mass-euler}

Under the assumptions of Theorem \ref{theo:cons-Euler}, the aim here is to prove that the limit $(\bar h, \bar \bfu)$ of the scheme \eqref{euler:scheme} satisfies the weak form of the mass equation \eqref{weak-mass}.
In order to do so, we apply the consistency result of Theorem \ref{theo:lw} in the appendix, which is a slightly weaker and simpler version (sufficient in our case) of \cite[Theorem 2.1]{gal-21-wea}; we apply it here with $U = (h,\bfu)$, $\beta(U) = h$, $\bff(U) = h \bfu$, $\mathcal P\m = \mathcal \mesh\m, \edgespart\m = \mathcal \edges\m$, and
\begin{align} \nonumber
\mathcal C\m_{\mbox{\tiny{MASS}}}(U\m) : &\quad \Omega\times(0,T) 	\to \xR,
\\ \label{eqdef:convmass} & \quad
(\bfx,t) \mapsto (\eth_t h\m)_K^n + \dive_K\,\bigl((h\m)^n (\bfu\m)^n\bigr)
\\ & \nonumber \hspace{18ex}
\mbox{for } \bfx \in K,\ K \in \mesh\m \mbox{ and } t \in [t_n,t_{n+1}),\ n \in \llbracket 0, N\m-1 \rrbracket.
\end{align}
The boundedness assumptions \eqref{bound-h}  and \eqref{bound-u} imply that \eqref{lemgen:linfbound} holds.
Furthermore, the assumption of Theorem \ref{theo:cons-Euler} that $(h\m, \bfu\m)_\mnn$ is a sequence of solutions to the scheme \eqref{euler:scheme} converging to $(\bar h, \bar \bfu)$ in $\xL^1(\Omega \times (0,T)) \times \xL^1(\Omega \times (0,T))^2$ implies that \eqref{lemgen:l1conv} holds.

By the initialisation \eqref{sch-init_h}-\eqref{sch-init_u} of the scheme, it is clear that
\[
\sum_{K \in \mathcal \mesh\m} \int_K \bigl| (h\m)_K^0  - h_0(\bfx) \bigr| \dx \to 0 \quad \mbox{as } m \mbox{ tends to } +\infty,
\]
so that the assumption \eqref{hyp:condi} is satisfied.

From the definition \eqref{eq:def_unk} of the discrete unknowns, for any $n \in \llbracket 0, N\m -1\rrbracket$ and $K \in \mesh\m$, we have $\beta(U\m(\bfx,t)) = h_K^n$ for any $(\bfx,t) \in K\times[t_n, t_{n+1})$.
Furthermore, from the definition \eqref{eqdef:convmass} of the convection operator in the discrete mass balance equation, with the notations of the appendix, $(\beta\m)_K^n =h_K^n$.
Hence,
\[
\sum_{n=0}^{N\m-1}  \sum_{K \in \mathcal \mesh\m}\ \int_{t_n}^{t_{n+1}} \int_K \bigl |(h\m)_K^n  - h\m(\bfx,t) \bigr| \dx \dt =0,
\]
and the assumption \eqref{hyp:t} is also clearly satisfied.
Now, in the expression \eqref{eqdef:convmass} of the convection operator, the discrete flux through an edge $\edge$ reads $(\bfF\m)_\edge^n = h_\edge^n \bfu_\edge^n$. and, because the velocity components are piecewise constant on different grids, over a cell $K\in\mesh$ and the time interval $[t_n, t_{n+1})$, $n \in \llbracket 0,  N\m-1 \rrbracket$,
\begin{align*}	&
\bff(U^m(\bfx,t)) = (f_1(U^m(\bfx,t)),f_2(U^m(\bfx,t))), \mbox{ where }
\\ & \qquad
f_i(U^m(\bfx,t)) =
\begin{cases} h_K^n u_{i,\edge}^n  \mbox{ if }\bfx \in D_{K,\edge}  \\[1ex]
              h_K^n u_{i,\edge'}^n \mbox{ if } \bfx \in D_{K,\edge'},
\end{cases}
\mbox{ with } \edge\mbox{ and } \edge' \mbox{ the edges of } K \mbox{ perpendicular to } \bfe\ei.
\end{align*}
For $\bfx \in K$, $\edge=K|L$ and $t \in [t_n, t_{n+1})$, since $h_\edge^n$ is defined as a convex combination of $h_K^n$ and $h_L^n$,
\begin{align*}
\Bigl| \Bigl((\bfF\m)_\edge^n - \bff(U^m(\bfx,t)\Bigr)\cdot \bfn_{K,\edge} \Bigr|
&
= \Bigl| \Big(h_\edge^n  \bfu_\edge^n -  h_K^n  \bfu_\edge^n +  h_K^n  \bfu_\edge^n  -  h_K^n \bfu(\bfx,t) \Big) \cdot \bfn_{K,\edge} \Bigr|
\\	&
\le C^u \ | h_K^n - h_L^n | + C^h \ |\bfu_\edge^n - \bfu_{\edge'}^n|.
\end{align*}
We thus have
\[
\sum_{n=0}^{N\m-1}  \sum_{K \in \mesh\m}  \frac{\mathrm{diam}(K)}{|K|} \int_{t_n}^{t_{n+1}} \int_K
\sum_{\edge \in \edges(K)} |\edge|\ \Bigl| \Bigl((\bfF\m)_\edge^n - \bff(U^m(\bfx,t)\Bigr)\cdot \bfn_{K,\edge} \Bigr| \dx \dt \leq
R_1\m + R_2\m,
\]
with
\begin{eqnarray*} &&
R_1\m = C^u \sum_{n=0}^{N\m-1}  \delta t \sum_{K \in \mesh\m} \mathrm{diam}(K) \sum_{\edge \in \edges(K),\ \edge=K|L} |\edge|\ | h_K^n - h_L^n |,
\\ &&
R_2\m = C^h \sum_{n=0}^{N\m-1}  \delta t \sum_{K \in \mesh\m} \mathrm{diam}(K) \sum_{\edge \in \edges(K)} |\edge|\ |\bfu_\edge^n - \bfu_{\edge'(\edge,K)}^n|,
\end{eqnarray*}
where $\edge'(\edge,K)$ stands for the edge of $K$ opposite to $\edge$.
The proof that these remainder terms tend to zero relies on Lemma \ref{lem:translates} of the appendix, which states the convergence to zero of such translations applied to a convergent sequence of functions in $\xL^1$.
The application of this result only requires to collect the jumps involving an unknown (here $h_K^n$ or $u_\edge^n$) and to check that the sum of the associated weights is bounded by the measure of "the support of the unknown", \ie\ the part of the domain where the discrete function takes the value of the unknown (so, here, $K$ and $D_\edge$).
For a given cell $K\in\mesh\m$, $h_K^n$ appears twice the number of the edges of a cell in $R_1\m$, and the total weight is
\[
\omega_K^n = C^u\ \sum_{\edge \in \edges(K),\ \edge=K|L} \bigl( \mathrm{diam}(K) + \mathrm{diam}(L) \bigr)\ |\edge|.
\]
This quantity is bounded by $C\ |K|$, where the real number $C$ only depends on the parameter $\theta$ measuring the regularity of the mesh.
For a velocity unknown at the edge $\edge=K|L$, we get:
\[
\omega_\edge^n = 2\,C^h\ (\mathrm{diam}(K) + \mathrm{diam}(L))\ |\edge|
\]
which is once again bounded by $ C\ |D_\edge|$ with $C$ independent of the mesh thanks to the regularity assumption.
Hence, the assumption \eqref{hyp:x} of Theorem \ref{theo:lw} is also satisfied, and we get
\begin{multline*}
\forall \varphi \in C^\infty_c \bigl(\Omega \times [0,T)\bigr), \qquad \int_0^T \int_\Omega \mathcal C\m_{\mbox{\tiny{MASS}}}(U\m)\ \varphi (\bfx,t) \dx \dt \quad \to
\\
-\int_\Omega h_0(\bfx) \, \varphi(\bfx,0) \dx -\int_0^T \int_\Omega \Bigl[ \bar h(\bfx,t) \, \partial_t \varphi(\bfx,t) + \bar h(\bfx,t)\, \bar\bfu(\bfx,t) \, \cdot \gradi \varphi(\bfx,t) \Bigr] \dx \dt
\quad \mbox{ as } m \to + \infty.
\end{multline*}
Therefore, we conclude that the limit $(\bar h, \bar \bfu)$ of the approximate solutions defined by the forward Euler scheme \eqref{euler:scheme} satisfies the weak form \eqref{weak-mass} of the mass balance equation.
%
%
\subsubsection{Consistency, momentum equation} \label{subsubsec:cons-mom-euler}

Let $\bfvarphi = (\varphi_1, \varphi_2) \in (C_{c}^{\infty}(\Omega \times [0,T)))^2$ be a test function.
Multiplying Equation \eqref{sch-vitesse} by $\varphi_i(\bfx,t)$, integrating over $D_\edge\times(t_n,t_{n+1})$ and summing the result over $\edge \in \edges^{(m,i)}$ and over $n \in \llbracket 0, N\m-1 \rrbracket$ yields:
\begin{equation} \label{sumQ}
Q_{1,i}\m + Q_{2,i}\m + Q_{3,i}\m + Q_{4,i}\m = 0,\quad i=1,2,
\end{equation}
with (dropping the superscripts $(m)$ in the summations for the sake of simplicity)
\begin{align}  \label{eq:q1} &
Q_{1,i}\m = \sum_{n=0}^{N\m-1} \sum_{\edge \in \edges^{(m,i)}} \int_{t_n}^{t_{n+1}} \int_{D_\edge} \ \eth_t(h u_i)_\edge^n\ \ \varphi_i(\bfx,t) \dx \dt,
\\ \label{eq:q2} &
Q_{2,i}\m = \sum_{n=0}^{N\m-1} \sum_{\edge \in \edges^{(m,i)}} \int_{t_n}^{t_{n+1}} \int_{D_\edge} \ \dive_{D_\edge} (h^n \bfu^n u_i^n)\ \varphi_i(\bfx,t) \dx \dt,
\\  \label{eq:q3} &
Q_{3,i}\m = \sum_{n=0}^{N\m-1} \sum_{\edge \in \edges^{(m,i)}} \int_{t_n}^{t_{n+1}} \int_{D_\edge}\ \eth_\edge p^{n+1}\ \varphi_i(\bfx,t) \dx \dt,
\\ \label{eq:q4} &
Q_{4,i}\m = \sum_{n=0}^{N\m-1} \sum_{\edge \in \edges^{(m,i)}} \int_{t_n}^{t_{n+1}} \int_{D_\edge} g\ h_{\edge,c}^{n+1}\ \ \varphi_i(\bfx,t) \dx \dt.
\end{align}
%
%

\paragraph{\bf The nonlinear convection operator.}
In order to study the limit of the discrete non linear convection operator defined by $Q_i\m = Q_{1,i}\m + Q_{2,i}\m$, we apply Theorem \ref{theo:lw} with $U = (h,\bfu)$, $\beta(U) = h u_i$, $\bff(U) = h u_i \bfu $, with $\mathcal P\m$ the set of dual cells associated with $u_i$ (that is with the cells corresponding to the vertical edges for $i=1$ and the horizontal edges for $i = 2$), with $\edgespart = \edgesd^{(m,i)}$ (\ie\ the dual edges associated to the $i$-th dual mesh of $\mesh\m$) and with the dual fluxes $(\bfG)_\edged^n$  defined by \eqref{discrete_dual_flux}.
The discrete non linear convection operator thus reads
\begin{align*}
[\mathcal C\m_{\mbox{\tiny{MOM}}}(U\m)]_i :
& \quad
\Omega\times(0,T) \to \xR,
\\ & \quad
(\bfx,t) \mapsto (\eth_t h u_i)_\edge^n -\ \dive_{D_\edge}\,(h^n\, u_i\, \bfu^n)
\\ & \hspace{13ex}
\mbox{ for } \bfx \in D_\edge,\ \edge \in \edges^{(m,i)} \mbox{ and } t \in (t_n,t_{n+1}),\ n \in \llbracket 0, N\m -1 \rrbracket
\end{align*}
(again dropping the superscripts $\m$ for the sake of simplicity).

Let us first check the assumption \eqref{hyp:condi}.
Since the initial condition and its discrete approximation defined by the initialisation of the scheme \eqref{sch-init_h}-\eqref{sch-init_u} are bounded in $\xL^\infty(\Omega)$, using the identity $2(ab-cd)=(a-c)(b+d)+(a+c)(b-d)$, we have
\begin{multline*} \hspace{5ex}
\sum_{\edge \in \edgesint^{(m,i)}}\ \int_{D_\edge} \bigl| (h\m\bfu\m)_{i,\edge}^0  -  h_0(\bfx)\, u_{i,0}(\bfx) \bigr| \dx
\\
\leq C \sum_{\edge \in \edgesint^{(m,i)}}\ \int_{D_\edge} \Bigr( \bigl| (h\m)^0  -  h_0(\bfx) \bigr| + \bigl| (\bfu\m)_{i,\edge}^0  -  u_{i,0}(\bfx) \bigr| \Bigl)\dx = T\m,
\hspace{5ex} \end{multline*}
with $C$ independent of $m$.
For a given function $\psi \in \xL^1(\Omega)$, and any subset $A$ of $\Omega$, let us denote by $\langle\psi\rangle_A$ the mean value of $\psi$ on $A$.
By the initialisation of the scheme, we have:
\begin{multline*} \hspace{5ex}
T\m = C \sum_{\substack{\edge \in \edgesint^{(m,i)},\\[0.5ex] \edge=K|L}}\ \Bigl( \int_{D_{K,\edge}} \Bigr( \bigl| <h_0>_K  -  h_0(\bfx) \bigr| \dx + \int_{D_{L,\edge}} \Bigr( \bigl| <h_0>_L  -  h_0(\bfx) \bigr| \dx
\\
+ \int_{D_\edge} \bigl| <u_{i,0}>_{D_\edge}  -  u_{i,0}(\bfx) \bigr| \Bigl)\dx \Bigr).
\hspace{5ex} \end{multline*}
Reordering the sum, we get
\[
T\m \leq C \sum_{K \in \mesh\m} \int_K \bigl| <h_0>_K  -  h_0(\bfx) \bigr| \dx
+ C \sum_{\edge \in \edgesint^{(m,i)}} \int_{D_\edge} \bigl| <u_{i,0}>_{D_\edge}  -  u_{i,0}(\bfx) \bigr| \Bigl)\dx,
\]
and $T\m$ tends to zero when $m$ tends to $+\infty$ by standard arguments.

For the time derivative term, thanks to definition \eqref{discrete_water_dual} of the edge water height as a weighted average of the water height in the adjacent cells, we have
\[
\sum_{n=0}^{N\m-1} \sum_{\edge \in \edges^{(m,i)}}\ \int_{t_n}^{t_{n+1}} \int_{D_\edge} \bigl| (h\bfu)_{i,\edge}^n - h(\bfx,t)\, u_i(\bfx,t) \bigr| \dx \dt=0, \quad i = 1, 2.
\]
so that the assumption\eqref{hyp:t} is satisfied.

In order to show that the assumption \eqref{hyp:x} is satisfied, we need to show that
\begin{multline} \label{eq:cons-qdm}
R\m=\sum_{n=0}^{N\m-1}  \sum_{\edge \in \edges^{(m,i)}} \  \frac {\mathrm{diam} (D_\edge)}{|D_\edge|}\\ \int_{t_n}^{t_{n+1}} \int_{D_\edge}
\sum_{\edged \in \edgesd(D_\edge)}\  |\edged|\ \Big| \bigl((\bfG\m)_{\edged}^n - h(\bfx,t)\, u_i(\bfx,t)\, \bfu(\bfx,t) \bigr)\cdot \bfn_{\edge,\edged}  \Bigr| \dx \dt \to 0
\mbox{ as } m \to + \infty.
\end{multline}
Let us estimate, for any $\edged \in \edgesd^{(m,i)}$, $n \in \llbracket 0, N\m -1\rrbracket$ and $\bfx \in D_\edge$, with $D_\edge$ a dual cell adjacent to $\edged$, the quantity $Y_\edged^n$ defined by:
\[
Y_\edged^n(\bfx) = \Big|\bigl( (\bfG\m)_\edged^n - h(\bfx,t)\, u_i(\bfx,t)\, \bfu(\bfx,t) \bigr) \cdot \bfn_{\edge,\edged}\Big|.
\]
Let $K$ and $L$ be the (primal) cells such that $\edge = K|L$.
\begin{list}{--}{\itemsep=0.ex \topsep=0.5ex \leftmargin=1.cm \labelwidth=0.7cm \labelsep=0.3cm \itemindent=0.cm}
\item If the edge $\edged$ is parallel to $\edge$, $\edged=\edge|\edge' \subset K$, then $(\bfG\m)_{\edged}^n$ is defined by \eqref{dual_flux_edge_paral}.
By the triangle inequality and thanks to the assumptions \eqref{muscl:cons}, \eqref{muscl:u:cons}, \eqref{bound-h} and \eqref{bound-u}, we get that
\begin{equation} \label{eq:bound_y1}
Y_\edged^n(\bfx) \le \frac 1 2 (C^u)^2 |h_K -h_L| + \frac 1 2 (C^u)^2 |h_K - h_J| + C^h C^u |u_{\edge,i} - u_{\edge',i}|,\quad \forall \bfx \in D_\edge,
\end{equation}
where  $J$ is the (primal) cell such that $\edge' = J|K$, see Figure \ref{fig:convection}, left.
\item If $\edged$ is orthogonal to $\edge$, $\edged=\edge|\edge'$, then $(\bfG\m)_{\edged}^n$ is defined by  \eqref{dual_flux_edge_perp}.
The computation of the quantity $Y_\edged^n(\bfx)$ is cumbersome, and we only give here an example, on Figure \ref{fig:cons_conv}.
The important point is that, $\forall \bfx \in D_\edge$, $Y_\edged^n(\bfx)$ is bounded by an expression of the form:
\begin{equation} \label{eq:bound_y2}
Y_\edged^n(\bfx) \le C\ \Bigl[ (C^u)^2 \sum_{M,N \in \mathcal N_\edged^2} |h_N^n -h_M^n|
+ C^h C^u \sum_{\eta,\eta' \in \mathcal N_\edge} |u_{\eta,3-i} - u_{\eta',3-i}|
+ C^h C^u |u_{\edge,i} - u_{\edge',i}| \Bigr],
\end{equation}
where $\mathcal N_\edged \subset \mesh\m$ is the set of the four primal cells adjacent to $\edged$, $\mathcal N_\edge \subset (\edges\m)^{(3-i)}$ is the set of the four edges orthogonal to $\edge$ and sharing a vertex with $\edge$ and $C$ is a given real number.
\end{list}
From Expression \eqref{eq:cons-qdm}, we get that
\[
|R\m| \leq \sum_{n=0}^{N\m-1} \sum_{\edge \in \edges^{(m,i)}}\ \delta t\ \mathrm{diam}(D_\edge) \sum_{\edged \in \edgesd(D_\edge)}\  |\edged| \ M_{\edge,\edged}^n,
\]
where $M_{\edge,\edged}^n$ is an upper bound for $Y_\edged^n(\bfx)$ over $D_\edge$.
We have seen that $\theta\m \le \theta$ implies a quasi-uniformity condition for the sequence of meshes, and thus the product $\mathrm{diam}(D_\edge) |\edged|$ is controlled by $C(\theta)\ |K|$ for any edge $\edge$, dual edge $\edged$ and cell $K$ of any mesh of the sequence, with $C(\theta)$ only depending on $\theta$; similarly, the product $\mathrm{diam}(D_\edge) |\edged|$ is controlled by $C(\theta)\ |D_{\edge'}|$, for any edges $\edge$ and $\edge'$ and dual edge $\edged$ of any mesh of the sequence.
Using the estimates \eqref{eq:bound_y1} and \eqref{eq:bound_y2}, we get a bound for $|R\m|$ as a collection of jumps of the height and the velocity, where each unknown appears only a bounded number of times; as for the mass balance equation, the total weight obtained by gathering the jumps involving a given unknown is thus bounded be the measure of the support of this unknown, multiplied by a real number only depending on $\theta$.
In addition, these jumps involve cells and edges in a stencil of bounded width (with respect to $m$).
We are thus in position to apply Lemma \ref{lem:translates} to obtain that $R\m$ tends to zero.
Hence, owing to Theorem \ref{theo:lw}, we get that
 \begin{multline} \label{limq1}
Q_i\m = Q_{1,i}\m + Q_{2,i}\m =	\int_0^T \int_\Omega \mathcal  C\m_{\mbox{\tiny{MOM}}}(U\m)\ \varphi_i (\bfx,t) \dx \dt \to
\\
-\int_0^T \int_\Omega \Bigl[ \bar h(\bfx,t)\ \bar u_i(\bfx,t) \ \partial_t \varphi(\bfx,t) + \bar h(\bfx,t)\ \bar u_i(\bfx,t)\ \bar \bfu(\bfx,t) \cdot \gradi \varphi_i(\bfx,t) \Bigr] \dx \dt
\\
- \int_\Omega h_0(\bfx)\ u_{i,0}(\bfx) \ \varphi_i(\bfx,0) \dx \mbox{ as } m \to + \infty.
\end{multline}

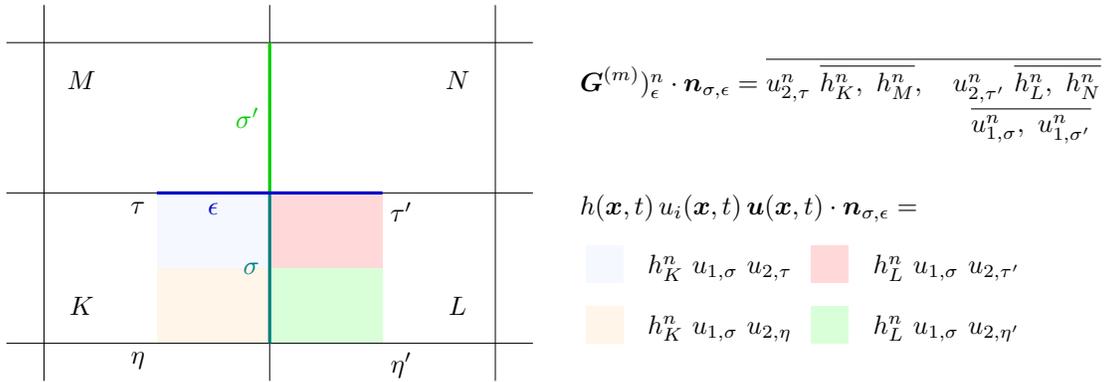
\begin{figure}[htbp]
\centering
\begin{tikzpicture}[scale=1]
\fill[color=orangec!30!white,opacity=0.3] (1.5, 0.)--(3., 0.)--(3.,1.)--(1.5, 1.)--cycle;
\fill[color=bclair,opacity=0.3] (1.5, 1.)--(3., 1.)--(3.,2.)--(1.5, 2.)--cycle;
\fill[color=green!50!white,opacity=0.3] (3, 0.)--(4.5, 0.)--(4.5,1.)--(3., 1.)--cycle;
\fill[color=red!50!white,opacity=0.3] (3, 1.)--(4.5, 1.)--(4.5,2.)--(3., 2.)--cycle;

\draw[very thin] (-0.5, 0.0)--(6.5,0.) node[below, near start] {$\eta$} node[below, near end] {$\eta'$};
\draw[very thin] (-0.5, 2.)--(6.5, 2.) node[below, near start] {$\edgeperp$} node[below, near end] {$\edgeperp'$};
\draw[very thin] (-0.5, 4.)--(6.5, 4.);
\draw[very thin] (0., -0.5)--(0., 4.5);
\draw[very thin] (3., -0.5)--(3., 4.5);
\draw[very thin] (6., -0.5)--(6., 4.5);

\draw[very thick, green!50!blue] (3,0)--(3,2) node[left, midway] {$\edge$};
\draw[very thick, green!80!black] (3,2)--(3,4) node[left, midway] {$\edge'$};
\draw[very thick, bfonce] (1.5,2.)--(4.5,2.) node[below, near start] {$\epsilon$};

\path node at (0.5, 0.5) {$K$};
\path node at (5.5, 0.5) {$L$};
\path node at (0.5, 3.5) {$M$};
\path node at (5.5, 3.5) {$N$};

\path[anchor=west] node at (7, 3.5) {$\bfG\m)_\edged^n \cdot \bfn_{\edge,\edged}
= \overline{\rule[-0.7ex]{0ex}{3.2ex}u_{2,\tau}^n\ \overline{\rule[-0.7ex]{0ex}{2.6ex} h_K^n,\ h_M^n},\quad u_{2,\tau'}^n\ \overline{\rule[-0.7ex]{0ex}{2.6ex}h_L^n,\ h_N^n}}$};
\path[anchor=west] node at (12.2, 2.9) {$\overline{\rule[-0.7ex]{0ex}{2.4ex}u_{1,\edge}^n,\ u_{1,\edge'}^n}$};
\path[anchor=west] node at (7, 1.8) {$h(\bfx,t)\, u_i(\bfx,t)\, \bfu(\bfx,t) \cdot \bfn_{\edge,\edged} =$};
\fill[color=bclair,opacity=0.3]           (7.2, 0.8)--(7.7, 0.8)--(7.7, 1.3)--(7.2, 1.3)--cycle; \path[anchor=west] node at (7.9, 1.){$h_K^n\ u_{1,\edge}\ u_{2,\edgeperp}$};
\fill[color=red!50!white,opacity=0.3]     (10.2, 0.8)--(10.7, 0.8)--(10.7, 1.3)--(10.2, 1.3)--cycle; \path[anchor=west] node at (10.9, 1.){$h_L^n\ u_{1,\edge}\ u_{2,\edgeperp'}$};
\fill[color=orangec!30!white,opacity=0.3] (7.2, 0)--(7.7, 0)--(7.7, 0.5)--(7.2, 0.5)--cycle; \path[anchor=west] node at (7.9, 0.2){$h_K^n\ u_{1,\edge}\ u_{2,\eta}$};
\fill[color=green!50!white,opacity=0.3]   (10.2, 0)--(10.7, 0)--(10.7, 0.5)--(10.2, 0.5)--cycle; \path[anchor=west] node at (10.9, 0.2){$h_L^n\ u_{1,\edge}\ u_{2,\eta'}$};
\end{tikzpicture}
\caption{Definition of the numerical momentum convection flux and of its expression as a function of the discrete piecewise constant functions. Particular case of the flux of a horizontal velocity $u_{1,\edge}$ through an horizontal face $\edged$. The notation $\overline{a,\ b}$ stands for a convex combination of $a$ and $b$.}
\label{fig:cons_conv}
\end{figure}
%
%

\paragraph{\bf The pressure gradient.}
Let us now turn to the term $Q_{3,i}\m$ defined by \eqref{eq:q3}.
Even if the discrete pressure gradient term is derived from its duality with the discrete divergence, it also takes the form of a conservative finite volume operator, which reads, for $\edge \in \edgesint^{(m,i)}$, $i=1,2$,
\[
\eth_\edge p^{n+1} = \frac 1 {|D_\edge|}\ \sum_{\substack{\edged \in \edgesd(D_\edge) \cap \edgesd^{(m,i)},\\[0.3ex] \edged \subset K}} |\edged|\ p^{n+1}_K \ \bfe\ei \cdot \bfn_{\edge,\edged},
\]
this expression being the discrete counterpart of $\partial_i p = \dive(p \,\bfe\ei)$.
To apply Theorem \ref{theo:lw}, we have to show that the assumption \eqref{hyp:x} is satisfied, that is to check that the term
\[
R\m = \sum_{n=0}^{N\m-1} \sum_{\edge \in \edges^{(m,i)}}\ \frac {\mathrm{diam} (D_\edge)}{|D_\edge|} \int_{t_n}^{t_{n+1}} \int_{D_\edge}
\sum_{\substack{\edged \in \edgesd(D_\edge) \cap \edgesd^{(m,i)},\\[0.3ex] \edged \subset K}}\ |\edged|\ \Big| (p_K^{n+1}\ \bfe\ei - p(\bfx,t)\ \bfe\ei)\cdot \bfn_{\edge,\edged} \Bigr| \dx \dt
\]
tends to zero when $m$ tends to $+\infty$.
Let $K$ and $L$ be the two primal cells adjacent to $\edge$; over $D_\edge \times (t_n,t_{n+1})$, $p(\bfx,t)$ takes two possible values: $p(\bfx,t) = p_K^n$ if $\bfx \in D_{K,\edge}$ and $p(\bfx,t) = p_L^n$ if $\bfx \in D_{L,\edge}$.
In addition, since $\edge$ and $\edged$ are parallel, we have $|\edged|=|\edge|$.
We thus have
\[
R\m = \sum_{n=0}^{N\m-1} \delta t\m \sum_{\substack{\edge \in \edges^{(m,i)},\\[0.2ex]\edge=K|L}}\ \frac {\mathrm{diam} (D_\edge)}{|D_\edge|}
\sum_{\substack{\edged \in \edgesd(D_\edge) \cap \edgesd^{(m,i)},\\[0.3ex] \edged \subset K}}\ |\edge|\ \bigl(|D_{K,\edge}|\ | p_K^{n+1}-p_K^n |+|D_{L,\edge}|\ | p_K^{n+1}-p_L^n |\bigr).
\]
This yields
\begin{multline*} \hspace{10ex}
|R\m| \leq \sum_{n=0}^{N\m-1}  \delta t\m \sum_{K \in \mesh\m} \Bigl(\sum_{\edge \in \edges(K) \cap \edges^{(m,i)}} |D_\edge|\Bigr)\ | p_K^{n+1}-p_K^n |
\\
+ \sum_{n=0}^{N\m-1}  \delta t\m
\sum_{\substack{\edge \in \edges^{(m,i)},\\[0.2ex]\edge=K|L}}\ \mathrm{diam} (D_\edge)\ |\edge|\ | p_K^n-p_L^n |.
\hspace{10ex}\end{multline*}
Thanks to the quasi-uniformity of the mesh, Lemma \ref{lem:translates} implies that both terms at the right-hand side tend to zero, and we get:
\begin{equation} \label{limq3}
Q_{3,i}\m \to - \int_0^T\int_{\Omega}  \bar p(\bfx,t) \ \partial_i \varphi_i(\bfx,t) \dx \dt \mbox{ as } m \to +\infty,\ \mbox{ for } i=1,\ 2.
\end{equation}
%
%

\paragraph{\bf The bathymetry.}
Let us now turn to the bathymetry term  $Q_{4,i}\m$ given by Equation \eqref{eq:q4}, which may be written
\[
Q_{4,i}\m = \int_0^T \int_\Omega \widetilde h\m(\bfx,t)\ \eth\m_i z (\bfx)\ \varphi_i (\bfx,t) \dx \dt,
\]
where
\begin{list}{--}{\itemsep=0.ex \topsep=0.5ex \leftmargin=1.cm \labelwidth=0.7cm \labelsep=0.3cm \itemindent=0.cm}
\item the function $\widetilde h\m : \Omega \times (0,T) \to \xR$ is defined by
$\widetilde h (\bfx,t) = h_{\edge,c}^{n+1} = \frac 1 2 (h_K^{n+1} + h_L^{n+1})$ for $\bfx \in D_\edge$ and $t \in (t_n,t_{n+1})$; the sequence $(\widetilde h\m)_\mnn$ is therefore bounded in $\xL^\infty(\Omega \times (0,T))$ and, thanks to the regularity of the mesh, converges to $\bar h$ in $\xL^1(\Omega \times (0,T))$;
\item by \eqref{gradz}, the function $\eth\m_i z: \Omega   \to \xR$ is defined by
\[
\eth_i\m z  = \sum_{\substack{\edge \in \edgesint^{(m,i)} \\[0.2ex] \edge = K|L,\ \bfx_K < \bfx_L}} \frac {|\edge|}{|D_\edge|}\ \bigl( z(\bfx_L) - z(\bfx_K) \bigr)\ \characteristic_{D_\edge}.
\]
Since $z$ is a regular function, the sequence of functions $(\eth\m_i z)_\mnn$ converges uniformly to the derivative $\partial_i z$ of $z$ with respect to the $i$-th variable as $m\to +\infty$.
\end{list}
Hence,
\begin{equation}\label{limq4}
Q_{4,i}\m \to  \int_0^T \int_\Omega \bar h(\bfx,t)\ \partial_i z (\bfx)\ \varphi_i (\bfx,t) \dx \dt \mbox{ as } m \to \infty.
\end{equation}
%
%

\paragraph{\bf Limit of the momentum balance equation.}
Passing to the limit in \eqref{sumQ} as $m\to +\infty$, using \eqref{limq1}, \eqref{limq3} and \eqref{limq4}, we get that the limit $(\bar h, \bar \bfu)$ of the approximate solutions defined by the forward Euler scheme \eqref{euler:scheme} satisfies \eqref{weak-mom}, which concludes the first part of the proof of  Theorem \ref{theo:cons-Euler}.
%
%
\subsection{Proof of the LW-consistency of the Heun scheme}\label{subsec:cons-heun}

\subsubsection{Mass balance}

Under the assumptions of Theorem \ref{theo:cons-Euler}, the aim here is to prove that the limit $(\bar h, \bar \bfu)$ of the scheme \eqref{heun:first_mass}-\eqref{heun:last_mom} satisfies the weak form of the mass equation \eqref{weak-mass}. 
In order to do so, we consider the equivalent mass equation \eqref{heun:equiv1}.
Because of the structure of the scheme, namely the fact that the divergence part of the convection operator is split in two terms, we cannot use here Theorem \ref{theo:lw} straightforwardly as in the case of the forward Euler scheme.
However, its building bricks, \ie\ Lemma\ \ref{lem:time-cons} which states the consistency of a discrete time derivative term and Lemma \ref{lem:space-cons} which addresses the consistency of a space divergence term, still apply.
We thus invoke Lemma \ref{lem:time-cons} with $U = (h,\bfu)$, $\beta(U) = h$, $\bff(U) = h \bfu$, $\mathcal P\m = \mathcal \mesh\m, \edgespart\m = \mathcal \edges\m$ for the time derivative and then Lemma \ref{lem:space-cons} twice, for the two divergence terms: once with $U = (h,\bfu)$,  $\bff(U) = h \bfu$, $\mathcal P\m = \mathcal \mesh\m, \edgespart\m = \mathcal \edges\m$, and then with $U = (\widehat h, \widehat \bfu)$,  $\bff(U) = \widehat h \widehat \bfu$.

Thanks to the arguments developed in Section \ref{subsubsec:cons-mass-euler}, it is easy to check that in each case, the assumptions of the lemmas are satisfied, so that we can conclude that $(\bar h, \bar \bfu)$ satisfies \eqref{weak-mass}.

\subsubsection{Momentum balance}
Still under the assumptions of Theorem \ref{theo:cons-Euler}, we now prove that the limit $(\bar h, \bar \bfu)$ of the scheme \eqref{heun:first_mass}-\eqref{heun:last_mom} satisfies the weak form of the mass equation \eqref{weak-mom}. 
Again we consider the equivalent momentum equation \eqref{heun:equiv2}. 
Multiplying the equation \eqref{heun:equiv2} by $|D_\edge| \varphi_{i,\edge}^{n+1}$, summing the result over $\edge \in \edges^{(m,i)}$ and then summing over $n \in \llbracket 0, N-1 \rrbracket$ and $i= 1,2$ yields:
  
\begin{equation} \label{sumQheun}
\sum_{i=1}^2 \Big[ Q_{1,i}\m + \frac 1 2 ( Q_{2,i}\m + \widehat Q_{2,i}\m +  Q_{3,i}\m +  \widehat  Q_{3,i}\m) + Q_{4,i}\m + \widehat Q_{4,i}\m)\Big] = 0,
\end{equation}
where $Q_{1,i}\m, \ldots,  Q_{4,i}\m$ are defined by \eqref{eq:q1}-\eqref{eq:q4}, and $\widehat Q_{2,i}\m, \widehat Q_{3,i}\m, \widehat Q_{4,i}\m$  are defined by \eqref{eq:q2}-\eqref{eq:q4}, replacing the unknowns $h,p,u$ by $\widehat h, \widehat p, \widehat u$.   

Again, because of the structure of the scheme, we cannot use  Theorem \ref{theo:lw} directly: we use Lemma \ref{lem:time-cons} for the time derivative term $Q_{1,i}\m$ and Lemma \ref{lem:space-cons} for the terms $Q_{2,i}\m$ and $\widehat Q_{2,i}\m$, with $\mathcal P\m$  the set of dual cells associated with $u_i$ (that is with the vertical edges for $i=1$ and the horizontal edges for $i = 2$),  with $\edgespart = \edgesd^{(m,i)}$ and with the dual fluxes $(\bfG)_{\edged}^n$  defined by \eqref{discrete_dual_flux}.
We first apply Lemma \ref{lem:time-cons} with $U = (h,\bfu)$, $\beta(U) = h u_i$, $\bff(U) = h \bfu  u_i$,  and  then Lemma \ref{lem:space-cons}, once with  $U = (h,\bfu)$, $\beta(U) = h u_i$, $\bff(U) = h \bfu  u_i$ and then with $U = (\widehat h,\widehat \bfu)$, $\beta(U) = \widehat h \widehat u_i$, $\bff(U) = \widehat h \widehat \bfu \widehat u_i$. 
Thanks to the arguments developed in Section \ref{subsubsec:cons-mom-euler}, it is easy to check that in each case, the assumptions of the lemmas are satisfied, so that 
 \begin{multline} \label{limq1q2heun}
 Q_{1,i}\m +\frac 1 2( Q_{2,i}\m + \widehat Q_{2,i}\m) \to  
-\int_0^T \int_\Omega \Bigl[ \bar h(\bfx,t)\bar u_i(\bfx,t) \, \partial_t \varphi_i(\bfx,t) + \bar h(\bfx,t) \bar u_i(\bfx,t) \bar \bfu(\bfx,t) \, \cdot \gradi \varphi_i(\bfx,t) \Bigr] \dx \dt  \\
- \int_\Omega h_0(\bfx) u_{i,0}(\bfx) \, \varphi_i(\bfx,0) \dx \mbox{ as } m \to + \infty.  
\end{multline}

The proof of convergence of the pressure gradient and bathymetry terms $Q_{3,i}\m$, $Q_{4,i}\m$, $\widehat  Q_{3,i}\m$ and $\widehat Q_{4,i}\m$ follow the exact same lines as that of the terms $Q_{3,i}\m$ and $Q_{4,i}\m$ in Section \ref{subsubsec:cons-mom-euler}. 
Hence,
\begin{multline}\label{limq3q4heun}
\sum_{i=1}^2\! \frac 1 2 \Big(Q_{3,i}\m +\widehat Q_{3,i}\m  + Q_{4,i}\m +\widehat Q_{4,i}\m \!\Big)
\to  - \int_0^T \int_\Omega  \Big(\bar p(\bfx,t) \ \dive \bfvarphi(\bfx,t) + \bar h(\bfx,t) \gradi z (\bfx) \cdot \bfvarphi  (\bfx,t) \Big) \dx \dt \\ \mbox{ as } m \to+ \infty.
\end{multline}
Therefore, owing to \eqref{limq1q2heun} and \eqref{limq3q4heun}, we may pass to the limit in \eqref{sumQheun} and conclude that $(\bar h, \bar \bfu)$ satisfies \eqref{weak-mom}.
This concludes the proof of Theorem \ref{theo:cons-Euler}.
%
%
\section{LW-entropy consistency of the forward Euler scheme} \label{sec:entropy}

\begin{theorem}[LW-entropy consistency of the forward Euler scheme]
Let $(\mesh\m,\edges\m)_\mnn$ be a sequence of meshes and $(\delta t\m)_\mnn$ be a sequence of time steps such that $\delta_{\mesh\m}$ and $\delta t\m$ tend to zero as $m \to +\infty$~; assume that there exists $\theta >0$ such that $\theta_{\mesh\m} \le \theta$ for any $m\in \xN$ (with $\theta_{\mesh\m} $ defined by \eqref{regmesh}).
Let $(h\m,\bfu\m)_\mnn$ be the associated sequence of solutions to the scheme \eqref{euler:scheme}, and suppose that $(h\m,\bfu\m)_\mnn$ satisfies \eqref{bound-h}, \eqref{bound-u} and converges to $(\bar h, \bar \bfu)$ in $\xL^1(\Omega \times (0,T)) \times \xL^1(\Omega \times (0,T))^2$.
Assume that the sequence of solutions satisfies the following stability assumptions
\begin{equation} \label{hyp:BVt}
\begin{array}{ll}
\exists\ C_{BVt} \in \xR_+ \, : \,
& \displaystyle
\sum_{n=0}^{N\m-1} \sum_{K\in\mesh\m} |K|\ \bigl| (h\m)^{n+1}_K - (h\m)_K^n \bigr|  \le C_{BVt},\quad \forall m \in \xN,
\\[3ex] & \displaystyle
\sum_{n=0}^{N\m-1} \sum_{\edge\in\edges^{(i,m)}} |D_\edge|\ \bigl| (u\m)^{n+1}_{i,\edge} - (u\m)_{i,\edge}^n \bigr|  \le C_{BVt},\quad \forall m \in \xN,\ i=1,2,
\end{array}
\end{equation}
that the sequence of meshes and time-steps satisfies the condition
\begin{equation} \label{hyp:deltat}
\dfrac {\delta t\m} {\inf_{\edge \in \edges\m} |\edge|}  \to 0  \mbox{ as } m \to +\infty,
\end{equation}
and that the coefficients $\lambda_\edge^K$ and $\mu^\edge_\edged$ in \eqref{muscl:cons} and \eqref{muscl:u:cons} satisfy:
\begin{equation}\label{muscl:entropie}
\lambda_\edge^K  \in [\frac 1 2 , 1] \mbox { if } \bfF_\edge \cdot \bfn_{K, \edge} \ge 0,
\mbox{ and } \mu^\edge_\edged \in [\frac 1 2 , 1]  \mbox { if } \bfF_\edged \cdot \bfn_{\edge, \edged} \ge 0.
\end{equation}
Then $(\bar h, \bar \bfu)$ satisfies the entropy inequality \eqref{eq:weakentropy}.
\end{theorem}

Note that the condition \eqref{hyp:deltat} is stronger than a CFL condition.
This inequality, together with the BV-stability \eqref{hyp:BVt} are used twice in the proof of the theorem, to prove that the remainder terms for the kinetic and potential energy balances tend to a non-negative quantity.
The condition \eqref{muscl:entropie} is also rather restrictive.
Indeed, it is satisfied by the usual two slopes  minmod limiter \cite{god-96-num} only in the case of a uniform Cartesian mesh \cite{pia-13-for}, and it is not satisfied by the three slopes minmod limiter.
For a non-uniform mesh, this condition does not allow to obtain a quasi second-order approximation.

\bigskip
\begin{proof}
Let $\varphi \in C_c^\infty(\Omega\times[0,T), \xR_+)$, and for a given discretisation $(\mesh\m,\edges\m)$ let $\varphi_K^n$ (resp. $\varphi_\edge^n$) denote the mean value of $\varphi$ on $K\times(t_n,t_{n+1})$ (resp. $D_\edge\times(t_n,t_{n+1})$), for any $K\in \mesh\m$ (resp. $\edge \in \edges\m$) and $n \in \llbracket 0,N\m -1\rrbracket$.
Let us multiply the part of the discrete kinetic energy balance associated to the $i$-th velocity component \eqref{disc-kinet} by $\delta t\ |D_\edge|\ \varphi_\edge^n$ and sum over $\edge \in \edges^{(m,i)}$ and $i=1,\ 2$; let us then multiply the discrete potential energy balance \eqref{disc-pot} by $\delta t\ |K|\ \varphi_K^n$ and sum over $K \in \mesh\m$.
Summing the two resulting equations and summing over $n \in \llbracket 0,N\m -1\rrbracket$, we get
\begin{multline} \label{ineq:entropy-weak-rest} \hspace{3ex}
\int_0^T \int_\Omega \mathcal C\m_{\mbox{\tiny{KIN}}}(U\m) \varphi (\bfx,t) \dx \dt + \int_0^T \int_\Omega \mathcal C\m_{\mbox{\tiny{POT}}}(U\m) \varphi (\bfx,t) \dx \dt
+ \mathcal P\m + \mathcal Z\m \\= -\mathcal R_k\m - \mathcal R_p\m,
\hspace{3ex} \end{multline}
where the different terms in this equation satisfy:
\begin{align*} &
\mathcal C\m_{\mbox{\tiny{KIN}}}(U\m) = \sum_{i=1}^2 \mathcal C\m_{\mbox{\tiny{KIN}},i}(U\m)\quad \mbox{with, for } i=1,2 \mbox{ and } \edge \in \edgesint^{(i,m)},
\\ & \hspace{5ex}
\mathcal C\m_{\mbox{\tiny{KIN}},i}(U\m)|_{D_\edge} = (\eth_t E_{k,i})_\edge^n
+ \frac 1 {|D_\edge|} \sum_{\edged \in \edgesd(D_\edge)} |\edged| \frac {(u_{i,\edged}^n)^2} 2 \bfF_\edged^n \cdot \bfn_{\edge, \edged}
\mbox{ and } (E_{k,i})_\edge^n = \frac 1 2 h_{D_\edge}^n (u_{i,\edge}^n)^2,
\displaybreak[1]\\[2ex] &
\mathcal C\m_{\mbox{\tiny{POT}}}(U\m)|_K = \frac 1 2 g\, (\eth_t h^2)_K^n + \dive_K\big(\frac 1 2 g (h^n)^2 \bfu^n\big),
\displaybreak[1]\\[2ex] &
\mathcal P\m = \sum_{n=0}^{N\m-1} \delta t\m \ \Bigl[\sum_{\edge \in \edges\m} |D_\edge|\ u_\edge^{n+1}\, \eth_\edge p^{n+1}\, \varphi_\edge^n + \sum_{K \in \mesh\m} |K|\ p_K^n\ \dive_K(\bfu^n)\, \varphi_K^n \Bigr],
\displaybreak[1]\\[3ex] &
\mathcal Z\m = \sum_{n=0}^{N\m-1} \delta t\m \Bigl[ g \sum_{\edge \in \edges\m} |D_\edge|\ h_{\edge,c}^{n+1} u_\edge^{n+1}\, \eth_\edge z\ \varphi_\edge^n
+ g \sum_{K\in \mesh\m} |K|\ z_K\ \Big( (\eth_t h)_K^n + \dive_K(h^n \bfu^n) \Big)\varphi_K^n \Bigr],
\displaybreak[1]\\[3ex] &
\mathcal R_k\m \ge \sum_{i=1}^2 \sum_{n=0}^{N\m-1} \delta t\m \sum_{\edge \in \edgesint^{(i,m)}}
\sum_{\edged \in \edgesd(D_\edge)} |\edged| \ \bfF_\edged^n \cdot \bfn_{\edge,\edged}
  \Bigl(-\frac 1 2 (u^n_{i,\edged} - u_{i,\edge}^n)^2 + (u^n_{i,\edged} - u_{i,\edge}^n) (u^{n+1}_{i,\edge} - u_{i,\edge}^n) \Bigr) \ \varphi_\edge^n,
\displaybreak[1]\\[2ex] &
\mathcal R_p\m \ge \sum_{n=0}^{N\m-1} \delta t\m \sum_{K\in \mesh\m}
\Big[
- \frac 1 2 g \sum_{\edge \in \edges(K)} |\edge|\ (h^n_\edge - h_K^n)^2\ \bfu_\edge^n \cdot \bfn_{K,\edge}
\\ & \hspace{40ex}
+ g \sum_{\edge \in \edges(K)} |\edge|\ (h^{n+1}_K - h_K^n)\ h_\edge^n\ \bfu_\edge^n \cdot \bfn_{K,\edge}
\Big] \varphi_K^n.
\end{align*}
In the terms $\mathcal P\m$ and $\mathcal Z\m$, the quantity $u_\edge^{n+1}$ stands for $\bfu_\edge^{n+1} \cdot \bfe^{(\edge)}$, where $\bfe^{(\edge)}$ is the vector of the canonical basis of $\xR^2$ normal to $\edge$.
%
%

\paragraph{\bf Kinetic energy convection term.}
Let us check that the above defined convection operator $\mathcal C^{(m,i)}_{\mbox{\tiny{KIN}}}$ satisfies the hypotheses \eqref{hyp:condi}--\eqref{hyp:x} of the LW-consistency theorem \ref{theo:lw} given in the appendix.
In fact, we check the consistency of $\mathcal C^{(m,i)}_{\mbox{\tiny{KIN}},i}$ for $i=1,2$ and, to this purpose, apply Theorem \ref{theo:lw} with $d=2$, $\mathcal P\m$ and $\mathfrak F\m$ the $i$-th dual mesh and its set of edges, $U=(h,\bfu)$ and $\beta(U) = E_{k,i}(U)= \frac 1 2 h u_i^2$.

The assumption \eqref{hyp:condi} is easily checked following the technique used to prove the same assumption for the initialisation of the momentum balance equation.
The assumption \eqref{hyp:t} compares the integrals of the quantities $(E_{k,i}\m)_\edge^n$ and of the expression of the kinetic energy applied to the discrete unknown $E_{k,i}(U\m(\bfx,t))$, \ie\ states that $\mathcal D\m$ tends to zero, with $\mathcal D\m$ defined by:
\[
\mathcal D\m=\sum_{n=0}^{N\m-1} \sum_{\edge \in \edges^{(m,i)}}\ \int_{t_n}^{t_{n+1}} \int_{D_\edge} \bigl| (E_{k,i}\m)_\edge^n - E_{k,i}(U\m(\bfx,t)) \bigr| \dx \dt.
\]
Because of the definition of $(h\m)^n_{D_\edge}$ as an average of $(h\m)^n_K$ and $(h\m)^n_L$, for $\edge=K|L$, weighted by the area of the half-diamond cells, $(E_{k,i}\m)_\edge^n$ may be integrated as a piecewise constant function over the half-diamond cells, and this piecewise constant function is exactly $E_{k,i}(U\m(\bfx,t))$, so $\mathcal D\m$ vanishes, and Assumption \eqref{hyp:t} is trivially satisfied.

Let us now turn to the assumption \eqref{hyp:x}, which reads
\[\sum_{n=0}^{N\m-1} \sum_{\edge \in \edges^{(m,i)}}\ \frac{\mathrm{diam} (D_\edge)}{|D_\edge|} \int_{t_n}^{t_{n+1}} \int_{D_\edge} \Big| \sum_{\edged \in \edgesd(D_\edge)}
|\edged|\Big((\bfF\m)_{\edged}^n -\bff(U^m(\bfx,t) \Big)\cdot \bfn_{\edge, \edged} \Big| \dx \dt \to 0
\]
as $m$ tends to $+\infty$, with $(\bfF\m)_{\edged}^n = \frac 1 2 (u_{i,\edged^n})^2 \bfF_\edged^n$ and $\bff(U) = \frac 1 2 h |u_i|^2 \bfu$.
This is obtained by setting the left-hand side of this expression as a set of jumps, invoking the assumed $\xL^\infty$ bounds for the discrete solutions and the quasi-uniformity of the meshes of the sequence, and applying Lemma \ref{lem:translates}.
As for the convection term in the momentum balance equation, the most intricate case is faced when $\edged$ is orthogonal to $\edge$, and we give the expression of both the discrete flux and the flux applied to the discrete function on Figure \ref{fig:cons_ek}.
A comparison of Figures \ref{fig:cons_conv} and \ref{fig:cons_ek} shows that the convergence proof searched here is very close to the same proof for the momentum convection, and we do not detail it further.

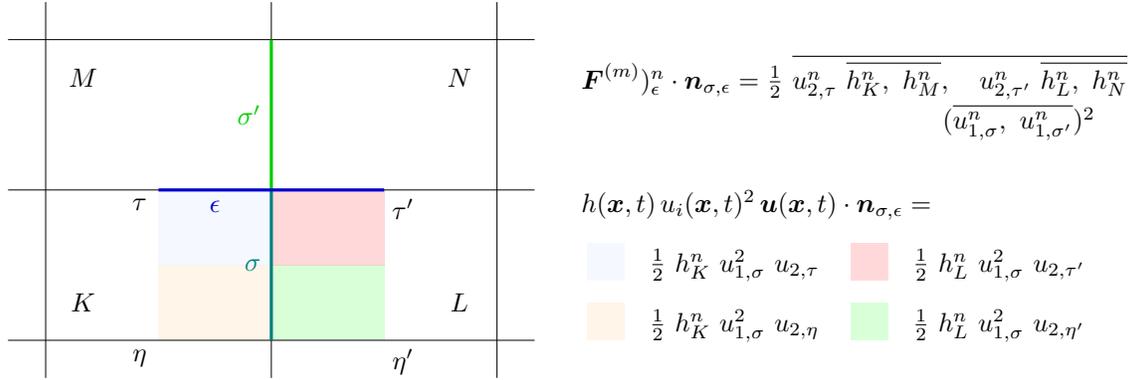
\begin{figure}[htb]
\centering
\begin{tikzpicture}[scale=1]
\fill[color=orangec!30!white,opacity=0.3] (1.5, 0.)--(3., 0.)--(3.,1.)--(1.5, 1.)--cycle;
\fill[color=bclair,opacity=0.3] (1.5, 1.)--(3., 1.)--(3.,2.)--(1.5, 2.)--cycle;
\fill[color=green!50!white,opacity=0.3] (3, 0.)--(4.5, 0.)--(4.5,1.)--(3., 1.)--cycle;
\fill[color=red!50!white,opacity=0.3] (3, 1.)--(4.5, 1.)--(4.5,2.)--(3., 2.)--cycle;

\draw[very thin] (-0.5, 0.0)--(6.5,0.) node[below, near start] {$\eta$} node[below, near end] {$\eta'$};
\draw[very thin] (-0.5, 2.)--(6.5, 2.) node[below, near start] {$\edgeperp$} node[below, near end] {$\edgeperp'$};
\draw[very thin] (-0.5, 4.)--(6.5, 4.);
\draw[very thin] (0., -0.5)--(0., 4.5);
\draw[very thin] (3., -0.5)--(3., 4.5);
\draw[very thin] (6., -0.5)--(6., 4.5);

\draw[very thick, green!50!blue] (3,0)--(3,2) node[left, midway] {$\edge$};
\draw[very thick, green!80!black] (3,2)--(3,4) node[left, midway] {$\edge'$};
\draw[very thick, bfonce] (1.5,2.)--(4.5,2.) node[below, near start] {$\epsilon$};

\path node at (0.5, 0.5) {$K$};
\path node at (5.5, 0.5) {$L$};
\path node at (0.5, 3.5) {$M$};
\path node at (5.5, 3.5) {$N$};

\path[anchor=west] node at (7, 3.5) {$\bfF\m)_\edged^n \cdot \bfn_{\edge,\edged}
= \frac 1 2\ \overline{\rule[-0.7ex]{0ex}{3.2ex}u_{2,\tau}^n\ \overline{\rule[-0.7ex]{0ex}{2.6ex} h_K^n,\ h_M^n},\quad u_{2,\tau'}^n\ \overline{\rule[-0.7ex]{0ex}{2.6ex}h_L^n,\ h_N^n}}$};
\path[anchor=west] node at (11.8, 2.9) {$(\overline{\rule[-0.7ex]{0ex}{2.4ex}u_{1,\edge}^n,\ u_{1,\edge'}^n})^2$};
\path[anchor=west] node at (7, 1.8) {$h(\bfx,t)\, u_i(\bfx,t)^2\, \bfu(\bfx,t) \cdot \bfn_{\edge,\edged} =$};
\fill[color=bclair,opacity=0.3]           (7.2, 0.8)--(7.7, 0.8)--(7.7, 1.3)--(7.2, 1.3)--cycle; \path[anchor=west] node at (7.9, 1.){$\frac 1 2\ h_K^n\ u_{1,\edge}^2\ u_{2,\edgeperp}$};
\fill[color=red!50!white,opacity=0.3]     (10.7, 0.8)--(11.2, 0.8)--(11.2, 1.3)--(10.7, 1.3)--cycle; \path[anchor=west] node at (11.4, 1.){$\frac 1 2\ h_L^n\ u_{1,\edge}^2\ u_{2,\edgeperp'}$};
\fill[color=orangec!30!white,opacity=0.3] (7.2, 0)--(7.7, 0)--(7.7, 0.5)--(7.2, 0.5)--cycle; \path[anchor=west] node at (7.9, 0.2){$\frac 1 2\ h_K^n\ u_{1,\edge}^2\ u_{2,\eta}$};
\fill[color=green!50!white,opacity=0.3]   (10.7, 0)--(11.2, 0)--(11.2, 0.5)--(10.7, 0.5)--cycle; \path[anchor=west] node at (11.4, 0.2){$\frac 1 2\ h_L^n\ u_{1,\edge}^2\ u_{2,\eta'}$};
\end{tikzpicture}
\caption{Definition of the numerical kinetic energy convection flux and of its expression as a function of the discrete piecewise constant functions. Particular case of the flux of a horizontal velocity $u_{1,\edge}$ through an horizontal face $\edged$. The notation $\overline{a,\ b}$ stands for a convex combination of $a$ and $b$.}
\label{fig:cons_ek}
\end{figure}

By Theorem \ref{theo:lw}, we thus get that, when $m$ tends to $+\infty$,
\begin{multline*} \hspace{5ex}
\int_0^T \int_\Omega \mathcal C\m_{\mbox{\tiny{KIN}},i}(U\m)\ \varphi (\bfx,t) \dx \dt \to
\\
-\int_\Omega E_{k,i}(U_0)\ \varphi(x,0) \dx - \int_0^T\int_\Omega E_{k,i}(\bar U)\ \partial_t \varphi + E_{k,i}(\bar U)\ \bar \bfu \cdot \gradi \varphi \dx \dt,
\hspace{5ex} \end{multline*}
with $E_{k,i}(\bar U) = \frac 1 2\, \bar h\, \bar u_i^2$.
Summing over $i=1,2$, we get that
\begin{multline} \label{convergence-Ckin} \hspace{5ex}
\int_0^T \int_\Omega \mathcal C\m_{\mbox{\tiny{KIN}}}(U\m)\ \varphi (\bfx,t) \dx \dt \to
\\
-\int_\Omega E_k(U_0)\ \varphi(x,0) \dx - \int_0^T\int_\Omega E_k(\bar U)\ \partial_t \varphi + E_k(\bar U)\ \bar \bfu \cdot \gradi \varphi \dx \dt,
\hspace{5ex} \end{multline}
with $E_k(\bar U) = \frac 1 2\, \bar h\,|\bar \bfu|^2$.
%
%

\paragraph{\bf Potential energy convection terms.}
Let us now check that the above defined convection operator $\mathcal C\m_{\mbox{\tiny{POT}}}$ satisfies the hypotheses \eqref{hyp:condi}--\eqref{hyp:x} of Theorem \ref{theo:lw} which we now apply with $d=2$, $\mathcal P\m$ and $\mathfrak F\m$ the primal mesh $\mesh\m$ and its set of edges $\edges\m$, $U=(h,\bfu)$, $\beta(U) = \frac 1 2 g h^2$ and $\bff(U) =\frac 1 2 g h^2 \bfu$.
Indeed, the initial condition for $h^2$ is $(h_K^0)^2 = (\langle h_0 \rangle_K)^2$, $\forall K \in \mesh\m$ and $m \in \xN$, and, since $h_0$ is assumed to be bounded in $\xL^\infty$,
\[
\sum_{K\in \mesh\m} \int_K \bigl| (\langle h_0 \rangle_K)^2 - h_0(x)^2 \bigr| \dx \to 0 \mbox{ as } m\to +\infty,
\]
so that the hypothesis \eqref{hyp:condi} is satisfied.
Next, for $m \in \xN$, $K \in \mesh\m$ and $n \in \llbracket 0, N\m-1 \rrbracket$, the term used in the discretisation of the time derivative, namely $((h\m)_K^n)^2$ is equal to $(h\m(\bfx,t))^2$, $\forall \bfx \in K$ and $\forall t \in (t_n,t_{n+1})$, so Assumption \eqref{hyp:t} is trivially satisfied.
Finally, the left hand side of \eqref{hyp:x} reads
\[
\mathcal D\m = \sum_{n=0}^{N\m-1} \sum_{K\in \mesh\m} \frac{\diam(K)}{|K|} \int_{t_n}^{t_{n+1}} \int_K
\sum_{\edge \in \edges(K)} |\edge|\ \bigl| \mathcal D_{\edge,K}^n (\bfx,t) \bigr| \dx \dt,
\]
with, for $\bfx \in K$ and $t \in (t_n,t_{n+1})$,
\[
\mathcal D_{\edge,K}^n (\bfx,t) = \frac 1 2 g\ \bigl((h_\edge^n)^2\ \bfu_\edge^n - h^2(\bfx,t)\ \bfu(\bfx,t) \bigr) \cdot \bfn_{K,\edge}.
\]
Let $i$ be the index such that $\edge$ is orthogonal to $\bfe\ei$ and let $\edge'$ be the opposite edge to $\edge$ in $K$.
We have:
\[
\mathcal D_{\edge,K}^n (\bfx,t) = \frac 1 2 g\ \begin{cases}
\bigl( (h_\edge^n)^2 -(h_K^n)^2 \bigr)\ u_{i,\edge}^n & \quad \mbox{if } \bfx \in D_{K,\edge},
\\[1ex]
(h_\edge^n)^2\ u_{i,\edge}^n - (h_K^n)^2\ u_{i,\edge'}^n & \quad \mbox{if } \bfx \in D_{K,\edge'}.
\end{cases}
\]
Thanks to the fact that $h_\edge^n$ is a convex combination of $h_K^n$ and $h_L^n$ for $\edge=K|L$ and to the $\xL^\infty$ bound on $h$ and $\bfu$, the quantity $\mathcal D_{\edge,K}^n (\bfx,t)$ may be bounded independently of $\bfx$ and $t$ by a weighted sum of the jumps $|h_K^n - h_L^n|$ and $|u_{i,\edge}^n - u_{i,\edge'}^n|$, and we conclude that $\mathcal D\m$ tends to zero when $m$ tends to $+\infty$ by Lemma \ref{lem:translates}, using the uniform bounds of the unknowns, the $L^1$ convergence of $h$ and $\bfu$ and the regularity of the meshes of the sequence.
We thus have
\begin{multline} \label{convergence-Cpot} \hspace{5ex}
\int_0^T \int_\Omega C\m_{\mbox{\tiny{POT}}}(U\m)\ \varphi (\bfx,t) \dx \dt \to
\\
-\int_\Omega \frac 1 2 g\ h_0(\bfx)^2\ \varphi(x,0) \dx - \int_0^T\int_\Omega \bigl(\frac 1 2 g\ \bar h^2\ \partial_t \varphi + \frac 1 2 g\ \bar h^2\ \bar \bfu \cdot \gradi \varphi\bigr) \dx \dt.
\hspace{5ex} \end{multline}
%
%

\paragraph{\bf Pressure terms.}
Let us rewrite $\mathcal P\m$ as
\begin{align*} &
\mathcal P\m = \mathcal D\m + \sum_{n=0}^{N\m-1} \delta t\m ( A^{n+1} + B^{n+1}) \ - \delta t\m B^0,\quad \mbox{with}
\\ & \hspace{10ex}
\mathcal D\m = \sum_{n=0}^{N\m-1} \delta t\m \ \sum_{\edge \in \edges\m} |D_\edge|\ u_\edge^{n+1}\, \eth_\edge p^{n+1}\, (\varphi_\edge^n - \varphi_\edge^{n+1}),
\\ & \hspace{10ex}
A^{n+1} = \sum_{\edge \in \edgesint\m} |D_\edge|\ u_\edge^{n+1}\ \eth_\edge p^{n+1}\ \varphi_\edge^{n+1} \quad \mbox{and} \quad B^{n+1} = \sum_{K \in \mesh\m} |K|\ p_K^{n+1}\ \dive_K(\bfu^{n+1})\ \varphi_K^{n+1}.
\end{align*}
By Lemma \ref{lem:u-p-phi} below,
\[
A^{n+1} + B^{n+1} = \sum_{K \in \mesh\m}\ \sum_{\edge \in \edges(K)} |D_{K,\edge}|\ p_K^{n+1}\, u^{n+1}_\edge\ \dfrac{|\edge|\ (\varphi_K^{n+1} - \varphi_\edge^{n+1})}{|D_{K,\edge}|}\ \bfe^{(\edge)} \cdot \bfn_{K,\edge}.
\]
On each subcell $D_{K,\edge}$ the quantity
\[
\partial_\edge \varphi = \frac{|\edge|\ (\varphi_K^{n+1} - \varphi_\edge^{n+1})}{|D_{K,\edge}|}\ \bfe^{(\edge)} \cdot\bfn_{K,\edge}
\]
is, up to higher order terms, a discrete differential quotient of $\varphi$ between $\bfx_K$ and $\bfx_\edge$, in the direction $i$ if $\edge \in \edgesi$, which uniformly converges to $-\partial_i \varphi$ in the case of a rectangular grid, and therefore, 
\[
\sum_{n=0}^{N\m-1} \delta t\m\ ( A^{n+1} + B^{n+1}) \to - \int_0^T \int_\Omega \bar p(\bfx,t)\ \bar \bfu(\bfx,t) \cdot \gradi \varphi (\bfx,t) \dx \dt \quad \mbox{as } m \to +\infty.
\]
Now,
\[
\delta t\m |B^{0}| =\delta t\m\ \Bigl| \sum_{K \in \mesh\m} |K|\ p_K^0\ \dive_K(\bfu^0)\ \varphi_K^0 \Bigr|
\le g\ \frac {\delta t\m} {\inf_{\edge\in \edges\m} |\edge|} \Vert h_0\Vert_\infty^2\ \Vert \varphi \Vert_\infty \bigl(\sum_{\edge\in \edges} |D_\edge|\bigr)\ \Vert \bfu_0 \Vert_{\infty},
\]
so that, by the assumption \eqref{hyp:deltat}, $B^0 \to 0\mbox{ as } m\to +\infty$.
Note that the assumption \eqref{hyp:deltat} could be avoided if we assume $u_0 \in W^{1,1}(\Omega)^2$ or $u_0 \in BV(\Omega)^2$; indeed, in this case we have
\[
|B^{0}| \le g\ \Vert h_0\Vert_\infty^2\ \Vert \varphi \Vert_\infty\ \Vert \bfu_0 \Vert_{W^{1,1}(\Omega)}.
\]
However, the assumption \eqref{hyp:deltat} seems unavoidable to deal with the remainder term appearing in the discrete potential and kinetic energy balances.
Concerning the term $\mathcal D\m$, thanks to the regularity of $\varphi$ which implies that $|\varphi_\edge^n - \varphi_\edge^{n+1}|\leq c_\varphi\ \delta t\m$,
\[
|\mathcal D\m| \leq c_\varphi\ \sum_{n=0}^{N\m-1} \delta t\m\ \frac{\delta t\m}{\inf_{\edge\in \edges\m} |\edge|}
\ \sum_{\substack{\edge \in \edgesint\m,\\[0.3ex] \edge=K|L}} |D_\edge|\ u_\edge^{n+1}\, |p_K^{n+1} - p_L^{n+1}|,
\]
and we conclude that $\mathcal D\m$ tends to zero as $m$ tends to $+\infty$ thanks to the fact that the unknowns are assumed to be uniformly bounded.
Note that, for this convergence to hold, thanks to Lemma \ref{lem:translates}, we only need the ratio $\delta t\m/(\inf_{\edge\in \edges\m} |\edge|)$ to be bounded (and not tending to zero).
Combining these convergence results, we have
\begin{equation} \label{lim-p-entropie}
\mathcal P\m \to - \int_0^T \int_\Omega \bar p(\bfx,t) \ \bar \bfu(\bfx,t) \cdot \gradi \varphi (\bfx,t) \dx \dt \quad \mbox{as } m \to +\infty.
\end{equation}

%
%

\paragraph{\bf Bathymetry terms.}
As for the passage to the limit in the momentum balance equation, we introduce the following piecewise constant functions:
\begin{list}{--}{\itemsep=0.ex \topsep=0.5ex \leftmargin=1.cm \labelwidth=0.7cm \labelsep=0.3cm \itemindent=0.cm}
\item $\widetilde h\m$ is the piecewise constant function equal to $h_{\edge,c}^{n+1}=\frac 1 2 (h_K^{n+1}+h_L^{n+1})$ on each set $D_\edge \times (t_n, t_{n+1})$, for $\edge = K|L \in \edgesint\m$ and $n \in \llbracket 0,N\m -1\rrbracket$; thanks to the regularity of the mesh, the function $\widetilde h\m$ converges to $\bar h$ in $\xL^1(\Omega \times(0,T))$, and so, thanks to the assumed uniform bound on the discrete water heights, in $\xL^q(\Omega \times(0,T))$ for $1 \leq q < +\infty$.
\item the function $\eth\m_i z: \Omega \to \xR$ is defined by
\[
\eth_i\m z = \sum_{\substack{\edge \in \edgesint^{(m,i)} \\[0.2ex] \edge = K|L,\ \bfx_K < \bfx_L}} \frac {|\edge|}{|D_\edge|}\ \bigl( z(\bfx_L) - z(\bfx_K) \bigr)\ \characteristic_{D_\edge}.
\]
Since $z$ is a regular function, the sequence of functions $(\eth\m_i z)_\mnn$ converges uniformly to the derivative $\partial_i z$ of $z$ with respect to the $i$-th variable as $m\to +\infty$.
We denote $\gradi\m z = (\eth\m_1 z,\ \eth\m_2 z)^t$.
\end{list}
With these notations, we get that
\begin{multline} \label{limit-terme-z-kinetic}
\sum_{n=0}^{N\m-1} \delta t\m\ \sum_{\edge \in \edges\m} |D_\edge|\ h_{\edge,c}^{n+1}\ u_\edge^{n+1}\ \eth_\edge z\ \varphi_\edge^n
\\
= \int_0^{T-\delta t\m} \int_\Omega \widetilde h\m(\bfx,t)\ \bfu\m(\bfx,t+\delta t\m) \cdot \gradi\m z\m(\bfx)\ \varphi\m(\bfx,t) \dx \dt
\\
\to \int_0^T \int_\Omega \bar h(\bfx,t)\ \bar\bfu(\bfx,t) \cdot \gradi z (\bfx)\ \varphi(\bfx,t) \dx \dt \mbox{ as } m \to +\infty,
\end{multline}
since $\bfu\m$ converges to $\bar u$ in $\xL^q(\Omega \times(0,T))^2$, for $1 \leq q < +\infty$.

The second part of the term $\mathcal Z\m$ reads:
\begin{multline*}
\mathcal Z_2\m = g \sum_{n=0}^{N\m-1} \delta t\m \sum_{K\in \mesh\m} |K|\ z_K\ \big( (\eth_t h)_K^n + \dive_K(h^n \bfu^n) \big) \varphi_K^n
\\
= \int_0^T \int_\Omega \mathcal C\m_{\mbox{\tiny{MASS}}}(U\m)(\bfx,t)\ \bigl(z(\bfx)\ \varphi(\bfx,t)\bigr) \dx \dt + \mathcal R\m_z,
\end{multline*}
where $\mathcal C\m_{\mbox{\tiny{MASS}}}(U\m)$ is the convection operator of the mass balance equation and $\mathcal R\m_z$ reads
\[
\mathcal R\m_z = g \sum_{n=0}^{N\m-1} \delta t\m \sum_{K\in \mesh\m} |K|\ z_K\ \big( (\eth_t h)_K^n + \dive_K(h^n \bfu^n) \big)\ <(z_K-z)\,\varphi>_{K\times(t_n,t_{n+1})}.
\]
We have $<(z_K-z)\,\varphi>_{K\times(t_n,t_{n+1})}=<(z_K-z)\,\varphi_K^n>_{K\times(t_n,t_{n+1})} + <(z_K-z)\,(\varphi-\varphi^n_K)>_{K\times(t_n,t_{n+1})}$, where we recall that $\varphi_K^n = <\varphi_K>_{K\times(t_n,t_{n+1})}$; since $z_K$ is the value of the regular function $z$ at the mass center of $K$, the first term is bounded by $c\, h_K^2$, and the second one is bounded by $c\, h_K\, ( h_K+\delta t\m)$, with $c$ only depending on $\varphi$ and $z$.
The term $\mathcal R\m_z$ may thus be shown to tend to zero when $m$ tends to $+\infty$, using only the $\xL^\infty$ bounds of the discrete solutions and the CFL condition.
Thanks to the weak convergence of $C\m_{\mbox{\tiny{MASS}}}(U\m)$ which has already been proved when studying the consistency of the Euler scheme and thanks to the regularity of $z$, we get:
\begin{multline*}
\mathcal Z_2\m \to
- \int_\Omega g\ z(\bfx)\ h_0(\bfx)\ \varphi(\bfx, 0) \dx
-\int_0^T \int_\Omega g\ z(\bfx)\ \bar h(\bfx,t)\ \partial_t \varphi(\bfx,t) \dx \dt
\\
- \int_0^T \int_\Omega g \ \bar h(\bfx,t)\ \bar\bfu (\bfx,t) \cdot \gradi(z \varphi)(\bfx, t) \dx \dt \mbox{ as } m \to +\infty.
\end{multline*} 
Adding the assertion \eqref{limit-terme-z-kinetic} to this relation yields:	
\begin{multline} \label{limit-terme-z}
\mathcal Z\m \to
- \int_\Omega g\ z(\bfx)\ h_0(\bfx)\ \varphi(\bfx, 0) \dx
-\int_0^T \int_\Omega g\ z(\bfx)\ \bar h(\bfx,t)\ \partial_t \varphi(\bfx,t) \dx \dt
\\
- \int_0^T \int_\Omega g \ z(\bfx)\ \bar h(\bfx,t)\ \bar\bfu (\bfx,t) \cdot \gradi(\varphi)(\bfx, t) \dx \dt \mbox{ as } m \to +\infty.
\end{multline}
%
%

\paragraph{\bf Remainder terms.}
The remainder term $\mathcal R_k\m$ in \eqref{ineq:entropy-weak-rest} satisfies $\mathcal R_k\m \ge \mathcal R_{k,1}\m +\mathcal R_{k,2}\m$ with
\begin{align*} &
\mathcal R_{k,1}\m = - \frac 1 2\ \sum_{i=1}^2 \sum_{n=0}^{N\m-1} \delta t\m \sum_{\edge \in \edgesinti}\ \sum_{\edged \in \edgesd(D_\edge)} |\edged|\ \bfF_\edged^n \cdot \bfn_{\edge,\edged}
\ \bigl(u^n_{i,\edged} - u_{i,\edge}^n \bigr)^2\ \varphi_\edge^n,
\\ &
\mathcal R_{k,2}\m = \sum_{i=1}^2\sum_{n=0}^{N\m-1} \delta t\m \sum_{\edge \in \edgesinti}\ \sum_{\edged \in \edgesd(D_\edge)} |\edged|\ \bfF_\edged^n \cdot \bfn_{\edge,\edged}
\ \bigl(u^n_{i,\edged} - u_{i,\edge}^n \bigr) \bigl(u^{n+1}_{i,\edge} - u_{i,\edge}^n \bigr)\ \varphi_\edge^n.
\end{align*}
Reordering the sums in the first term, we get
\[
\mathcal R_{k,1}\m = - \frac 1 2\ \sum_{i=1}^2 \sum_{n=0}^{N\m-1} \delta t\m \sum_{\substack{\edged \in \edgesdint\ei,\ \edged = \edge|\edge'}}\ |\edged|\ \bfF_\edged^n \cdot \bfn_{\edge,\edged}
\ \Bigl( \bigl(u^n_{i,\edged} - u_{i,\edge}^n \bigr)^2 - \bigl(u^n_{i,\edged} - u_{i,\edge'}^n \bigr)^2 \Bigr)\ \varphi_\edge^n.
\]
Supposing, without loss of generality, that the pair $(\edge,\edge')$ is ordered in such a way that $\bfF_\edged^n \cdot \bfn_{\edge,\edged} \geq 0$, we have $u_{i,\edged} -u_{i,\edge} = (1-\mu^\edge_\edged)\ (u_{i,\edge'} -u_{i,\edge})$ and 
$u_{i,\edged} -u_{i,\edge'} = \mu^\edge_\edged\ (u_{i,\edge} -u_{i,\edge'})$ with $\mu^\edge_\edged \in [\frac 1 2 , 1]$ by the assumption \eqref{muscl:entropie}.
Hence,
\[
\bigl(u^n_{i,\edged} - u_{i,\edge}^n \bigr)^2 - \bigl(u^n_{i,\edged} - u_{i,\edge'}^n \bigr)^2 = (1-2\,\mu^\edge_\edged)\ (u_{i,\edge'} -u_{i,\edge})^2 \leq 0
\]
and $\mathcal R_{k,1}\m \geq 0$.
Furthermore,
\[
\begin{array}{ll}
|\mathcal R_{k,2}\m| 
& \displaystyle
\leq \Vert \varphi \Vert_\infty\ \Vert h\m \Vert_\infty\ \vert\bfu\m \Vert_\infty^2\ \sum_{i=1}^2\ \sum_{n=0}^{N\m-1} \delta t\m \sum_{\edge \in \edgesinti}\ \bigl(\sum_{\edged \in \edgesd(D_\edge)} |\edged|\bigr)
\ \bigl| u^{n+1}_{i,\edge} - u_{i,\edge}^n \bigr|
\\[4ex] & \displaystyle
\leq \Vert \varphi \Vert_\infty\ \Vert h\m \Vert_\infty\ \Vert \bfu\m \Vert_\infty^2\ \frac{\delta t\m}{\min_{\edge \in \edges\m}|\edge|}\ C_\theta
\ \sum_{i=1}^2\ \sum_{n=0}^{N\m-1} \sum_{\edge \in \edgesinti}\ |D_\edge|\ \bigl| u^{n+1}_{i,\edge} - u_{i,\edge}^n \bigr|,
\end{array}
\]
where $C_\theta$ only depends on the parameter $\theta$ measuring the regularity of the sequence of meshes.
Thanks to the BV estimate \eqref{hyp:BVt} on the discrete velocities and the assumption \eqref{hyp:deltat} on the time step, the remainder term $\mathcal R_{k,2}\m$ tends to zero when $m$ tends to $+\infty$, and
\begin{equation} \label{err-kinlim}
\lim_{m\to +\infty} \mathcal R_{k} \ge 0.
\end{equation}
The remainder $\mathcal R_p\m$ satisfies $\mathcal R_p\m \ge \mathcal R_{p,1}\m + \mathcal R_{p,2}\m$, with
\[
\begin{array}{l} \displaystyle
\mathcal R_{p,1}\m = - \frac g 2\ \sum_{n=0}^{N\m-1} \delta t\m \sum_{K\in \mesh\m}\ \sum_{\edge \in \edges(K)} |\edge|\ (h^n_\edge - h_K^n)^2\ \bfu_\edge^n \cdot \bfn_{K,\edge}\ \varphi_K^n,
\\[2ex] \displaystyle
\mathcal R_{p,2}\m = g\ \sum_{n=0}^{N\m-1} \delta t\m \sum_{K\in \mesh\m}\ \sum_{\edge \in \edges(K)}
|\edge|\ (h^{n+1}_K - h_K^n)\ h_\edge^n\ \bfu_\edge^n \cdot \bfn_{K,\edge}\ \varphi_K^n.
\end{array}
\]
By the same arguments as for $\mathcal R_{k,1}\m$, with $\lambda^K_\edge \in [\frac 1 2 , 1]$ instead of $\mu^\edge_\edged \in [\frac 1 2 , 1]$, we get that $\mathcal R_{p,1}\m \geq 0$.
Similarly, the remainder $\mathcal R_{p,2}\m$ is shown to tend to zero when $m$ tends to $+\infty$ following the same lines as for $\mathcal R_{k,2}\m$, using the assumed BV estimate \eqref{hyp:BVt} for the discrete heights instead of the velocities and, once again, Assumption \eqref{hyp:deltat}.
Hence,
\begin{equation} \label{limRP}
\lim_{m\to +\infty} \mathcal R_p\m \ge 0.
\end{equation}
%
%

\paragraph{\bf Conclusion of the proof.} --
Owing to \eqref{err-kinlim} and \eqref{limRP}, passing to the limit in \eqref{ineq:entropy-weak-rest} as $m\to +\infty$ yields, together with \eqref{convergence-Ckin}, \eqref{convergence-Cpot}, \eqref{lim-p-entropie} and \eqref{limit-terme-z}, that the limit $(\bar h, \bar u)$ satisfies the weak entropy inequality \eqref{eq:weakentropy}.
\end{proof}

The next lemma, used to pass to the limit in the pressure terms of the entropy balance, is the discrete equivalent, on a staggered grid, of the formal equality $\displaystyle \int_\Omega (\bfu \cdot \gradi p \ \varphi + p \ \dive \bfu \ \varphi ) \dx= -\int_\Omega p \ \bfu \cdot \gradi \varphi \dx$.

\begin{lemma} \label{lem:u-p-phi}
Let $(\mesh, \edges)$ be a $d$-dimensional MAC discretisation of $\Omega$ in the sense of Definition \ref{def:MACgrid}.
Let $(p_K)_{K\in \mesh}$ and $(\bfu_\edge)_{\edge \in \edges}$ be the associated pressure and velocity discrete unknowns, and let $(\varphi_K)_{K\in\mesh}$ and $(\varphi_\edge)_{\edge\in\edges}$ be two families of real number.
Recall that an edge $\edge$ is orthogonal to a vector of the canonical basis of $\xR^2$, which we denote by $\bfe^{(\edge)}$, that $\bfu_\edge$ is colinear to $\bfe^{(\edge)}$ and let us define $u_\edge = \bfu_\edge \cdot \bfe^{(\edge)}$.
Then
\[
\sum_{\substack{\edge \in \edgesint,\\ \edge = K|L}} |D_\edge|\ u_\edge \ \eth_\edge p \ \varphi_\edge
+ \sum_{K \in \mesh} |K| \ p_K \ \dive_K \bfu \ \varphi_K
= \sum_{K \in \mesh} \sum_{\edge \in \edges(K)} |D_{K,\edge}| \ p_K \ u_\edge\ \dfrac{|\edge|(\varphi_K - \varphi_\edge)}{|D_{K,\edge}|}\ \bfn_{K,\edge} \cdot \bfe^{(\edge)}.
\]
\end{lemma}

\begin{proof}
Let us denote by $A$ and $B$ the first and second terms of the left hand side.
Then, with the definition \eqref{eq:grad} of the discrete gradients,
\[
A = - \sum_{K \in \mesh} p_K \sum_{\edge \in \edges(K)} |\edge|\ u_\edge\ \varphi_\edge\ \bfe^{(\edge)} \cdot \bfn_{K,\edge}.
\]
By the definition of the discrete divergence,
\[
B = \sum_{K \in \mesh} p_K \sum_{\edge \in \edges(K)} |\edge|\ u_\edge\ \varphi_K\ \bfe^{(\edge)} \cdot \bfn_{K,\edge}.
\]
Adding these two relations yields
\[
A + B = \sum_{K \in \mesh} p_K \sum_{\edge \in \edges(K)} |\edge|\ u_\edge \ (\varphi_K-\varphi_\edge)\ \bfe^{(\edge)} \cdot \bfn_{K,\edge},
\]
and the proof is complete.
\end{proof}
%
%
\section{Numerical results} \label{sec:num}

This section is devoted to numerical tests: we first check the order of convergence of the proposed scheme on a two-dimensional regular solution (Section \ref{subsec:num_smooth}); then we turn to one-dimensional and two-dimensional shock solutions on a plane topography (Sections \ref{subsec:1D_Riemann} and \ref{subsec:2D_Riemann}); in Section \ref{subsec:part_dam_break}, we address a two-dimensional dam-break problem in a closed computational domain with a variable topography, which, in particular, shows the ability of staggered schemes to "natively" cope with reflection boundary conditions; finally, we compute the motion of a liquid slug over a partly dry support (Section \ref{subsec:drop}).

We compare three schemes: the second-order scheme developed here, the scheme referred to in Section \ref{sec:Euler-sch} as the segregated forward Euler scheme (combining a segregated forward Euler scheme in time and the proposed MUSCL-like discretisation \eqref{implementation-calif} of the convection fluxes) and a first order scheme which still features the segregated forward Euler scheme in time but with first-order upwind convection fluxes.
These schemes are referred to in the following as the {\em second-order}, {\em segregated} and {\em first-order} scheme respectively.
Even though we only have a theoretical proof of weak entropy consistency for the first order in time schemes, throughout this numerical study, we never observed any sign of a possible convergence of the second order in time scheme to a non-entropy weak solution. 

The schemes have been implemented within the CALIF$^3$S open-source software \cite{califs} of the French Institut de S\^uret\'e et de Radioprotection Nucl\'eaire (IRSN); this software is used for the following tests.
%
%
\subsection{A smooth solution} \label{subsec:num_smooth}

We begin here by checking the accuracy of the scheme on a known regular solution consisting in a travelling vortex.
This solution is obtained through the following steps: we first derive a compact-support $H^2$ solution consisting in a standing vortex, then we make it time-dependent by adding a constant velocity translation to the reference frame.
The velocity field of the standing vortex and the pressure are sought under the form:
\[
\hat \bfu = f(\xi) \begin{bmatrix} -x_2 \\ x_1 \end{bmatrix},
\quad \hat p=\wp(\xi),
\]
with $\xi = x_1^2 + x_2^2$.
A simple derivation of these expressions yields:
\[
\hat \bfu \cdot \gradi \hat \bfu = -f(\xi)^2 \begin{bmatrix} x_1 \\ x_2 \end{bmatrix}
\]
and
\[
\quad \gradi \hat p = 2\, \wp'(\xi) \begin{bmatrix} x_1 \\ x_2 \end{bmatrix}.
\]
Using the relation $p = \frac 1 2 g h^2$, we thus obtain a stationary solution of the SWE \eqref{eq:sw} with a topography $z=0$ if $\wp$ satisfies $8\,g\, \wp=(F +c)^2$, where $F$ is such that $F'=f^2$, $F(0)=0$ and $c$ is a positive real number.
For the present numerical study, we choose $f(\xi)=10\,\xi^2 (1-\xi)^2$ if $\xi \in (0,1)$, $f=0$ otherwise, which indeed yields an $H^2(\xR^2)$ velocity field (note that, consequently, the pressure and the water height are also regular), and $c=1$.
The problem is made unsteady by adding a uniform translation: given a constant vector field $\bfa$, the pressure $p$ and the velocity $\bfu$ are deduced from the steady state solution $\hat p$ and $\hat \bfu$:
\[
h (\bfx,t)=\hat h(\bfx-\bfa t), \qquad \bfu(\bfx,t)=\hat \bfu(\bfx-\bfa t)+\bfa.
\]
The center of the vortex is initially located at $\bfx_0=(0,0)^t$, the translation velocity $\bfa$ is set to $\bfa=(1,1)^t$, the computational domain is $\Omega=(-1.2,\,2.)^2$ and the computation is run on the time interval $(0,0.8)$.

Computations are performed with successively refined meshes with square cells, and the time step is $\delta t = \delta_\mesh /8$, and corresponds to a Courant (or CFL) number with respect to the celerity of the fastest waves close to $1/3$.
The discrete $\xL^1$-norm of the difference between the exact solution and the solution obtained by the second-order scheme is given in Table \ref{tab:err}.
The observed order of convergence over the whole sequence is $2$ for the water height and $1.5$ for the velocity.
Results with the first-order scheme are given in Table \ref{tab:err_first}; one observes that the second-order scheme is much more accurate.
Finally, the segregated scheme yields good results on coarse meshes (it is the most accurate scheme on the $32\times32$ mesh); unfortunately, when refining the mesh, oscillations appear, and the convergence is lost.
This results confirms a behaviour already observed for the transport operator in \cite{pia-13-for}: for multi-dimensional problems, the smoothing produced by the Heun time-stepping seems to be necessary to compensate the oscillatory character of the MUSCL scheme (which, for the transport operator, does not lead, of course, to violate the local maximum principle warranted by construction of the limitation process).

\begin{table}[htb]
\centering
\begin{tabular}{|c|c|c|c|c|} \hline
 mesh & error($h$) & ord($h$) & error($u$) & ord($u$)
\\ \hline \rule[-0.7ex]{0ex}{3.3ex}
$32 \times 32$   & $3.61\,10^{-3}$ & /      & $2.93\,10^{-1}$ & /      \\
$64 \times 64$   & $1.15\,10^{-3}$ & $1.65$ & $1.14\,10^{-1}$ & $1.36$ \\
$128 \times 128$ & $2.58\,10^{-4}$ & $2.16$ & $4.06\,10^{-2}$ & $1.49$ \\
$256 \times 256$ & $5.85\,10^{-5}$ & $2.14$ & $1.49\,10^{-2}$ & $1.45$ \\
$512 \times 512$ & $1.53\,10^{-5}$ & $1.93$ & $4.67\,10^{-3}$ & $1.68$ \\
\hline \end{tabular}\\[3ex]
\centering\caption{Measured numerical errors for the travelling vortex -- Discrete $\xL^1$-norm of the difference between the numerical and exact solution at $t=0.8$, for the height and the velocity, and corresponding order of convergence.} \label{tab:err}
\end{table}

\begin{table}[htb]
\centering
\begin{tabular}{|c|c|c|c|c|} \hline
 mesh & error($h$) & ord($h$) & error($u$) & ord($u$)
\\ \hline \rule[-0.7ex]{0ex}{3.3ex}
$32 \times 32$   & $8.04\,10^{-3}$ & /      & $6.55\,10^{-1}$ & /      \\
$64 \times 64$   & $5.56\,10^{-3}$ & $0.53$ & $4.84\,10^{-1}$ & $0.44$ \\
$128 \times 128$ & $3.53\,10^{-3}$ & $0.66$ & $3.22\,10^{-1}$ & $0.59$ \\
$256 \times 256$ & $2.08\,10^{-3}$ & $0.76$ & $1.96\,10^{-1}$ & $0.72$ \\
$512 \times 512$ & $1.15\,10^{-3}$ & $0.85$ & $1.16\,10^{-1}$ & $0.76$ \\
\hline \end{tabular}\\[3ex]
\centering\caption{Measured numerical errors for the travelling vortex with the first order scheme - Discrete $\xL^1$-norm of the difference between the numerical and exact solution at $t=0.8$, for the height and the velocity, and corresponding order of convergence.} \label{tab:err_first}
\end{table}

\begin{table}[htb]
\centering
\begin{tabular}{|c|c|c|} \hline
 mesh & error($h$) & error($u$)
\\ \hline \rule[-0.7ex]{0ex}{3.3ex}
$32 \times 32$   & $2.06\,10^{-3}$ & $2.33\,10^{-1}$  \\
$64 \times 64$   & $1.37\,10^{-3}$ & $1.18\,10^{-1}$  \\
$128 \times 128$ & $1.24\,10^{-3}$ & $8.50\,10^{-2}$  \\
$256 \times 256$ & $1.26\,10^{-3}$ & $6.16\,10^{-2}$  \\
$512 \times 512$ & $1.56\,10^{-3}$ & $4.85\,10^{-2}$  \\
\hline \end{tabular}\\[3ex]
\centering\caption{Measured numerical errors for the travelling vortex with the segregated scheme - Discrete $\xL^1$-norm of the difference between the numerical and exact solution at $t=0.8$, for the height and the velocity.} \label{tab:err_stagg}
\end{table}
%
%
%
\subsection{A Riemann problem} \label{subsec:1D_Riemann}

We now turn to a one-dimensional shock solution, corresponding to a Riemann problem posed over $\Omega = (0,1)$.
The initial height is $h=1$ if $x< 0.5$ and $h=0.2$ otherwise, and the topography $z$ is set to zero over the computational domain; the fluid is initially at rest.
The solution consists in a 1-rarefaction wave and a 2-shock.

We plot on Figure \ref{fig:Rpb1} and Figure \ref{fig:Rpb2} the results obtained a $t=0.1$ with the second-order scheme, the segregated scheme and the first-order scheme.
The space step is $\delta x= 1/200$ and the time step is chosen as $\delta t= \delta x/10$, which corresponds to a CFL number lower than 0.5 with respect to the waves celerity (the maximal speed of sound is close to 3 and the maximal velocity is close to 2).
As expected, the first order scheme is more diffusive than the other ones.
As in the previous test, the segregated forward Euler scheme (with MUSCL fluxes) exhibits some oscillations, which are damped by the Heun time discretisation (see the Figure \ref{fig:Rpb2}).
In this test case, for both the second-order and the segregated scheme, the shock is captured with only one intermediate cell between the left and the right state.

\begin{figure}[htbp]
\includegraphics[width=0.75\textwidth]{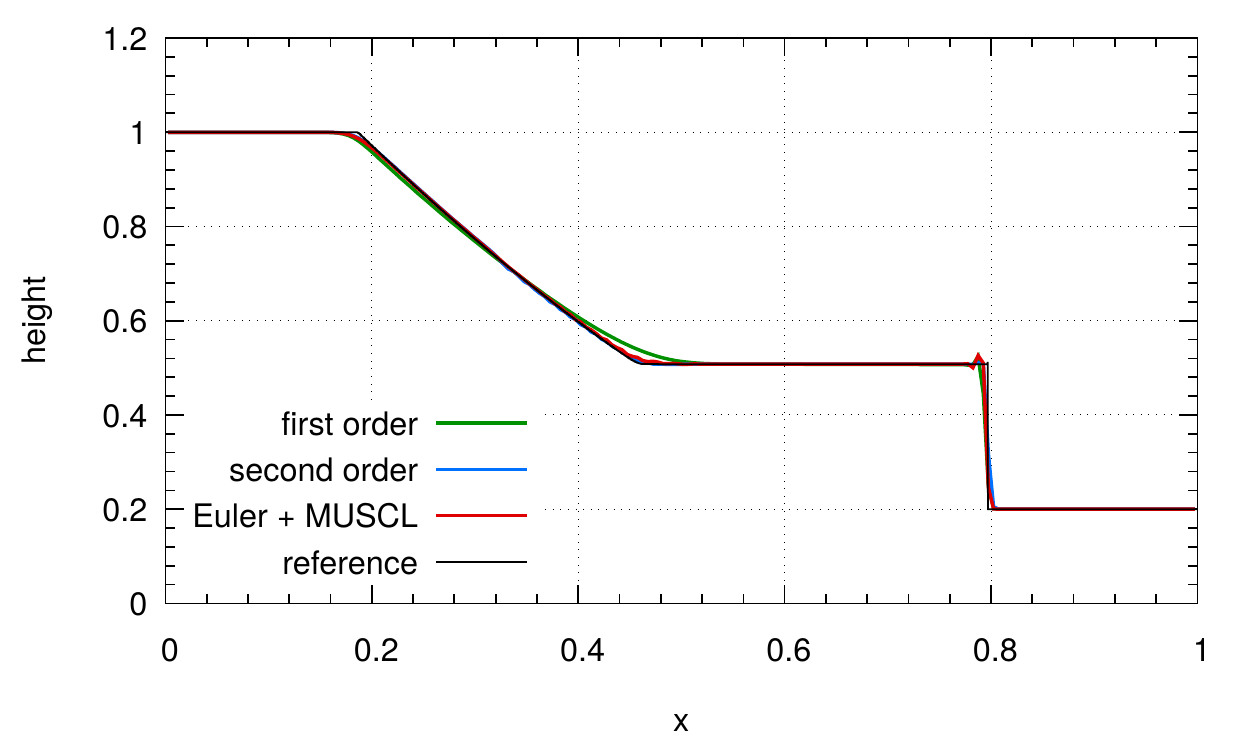}\\
\includegraphics[width=0.75\textwidth]{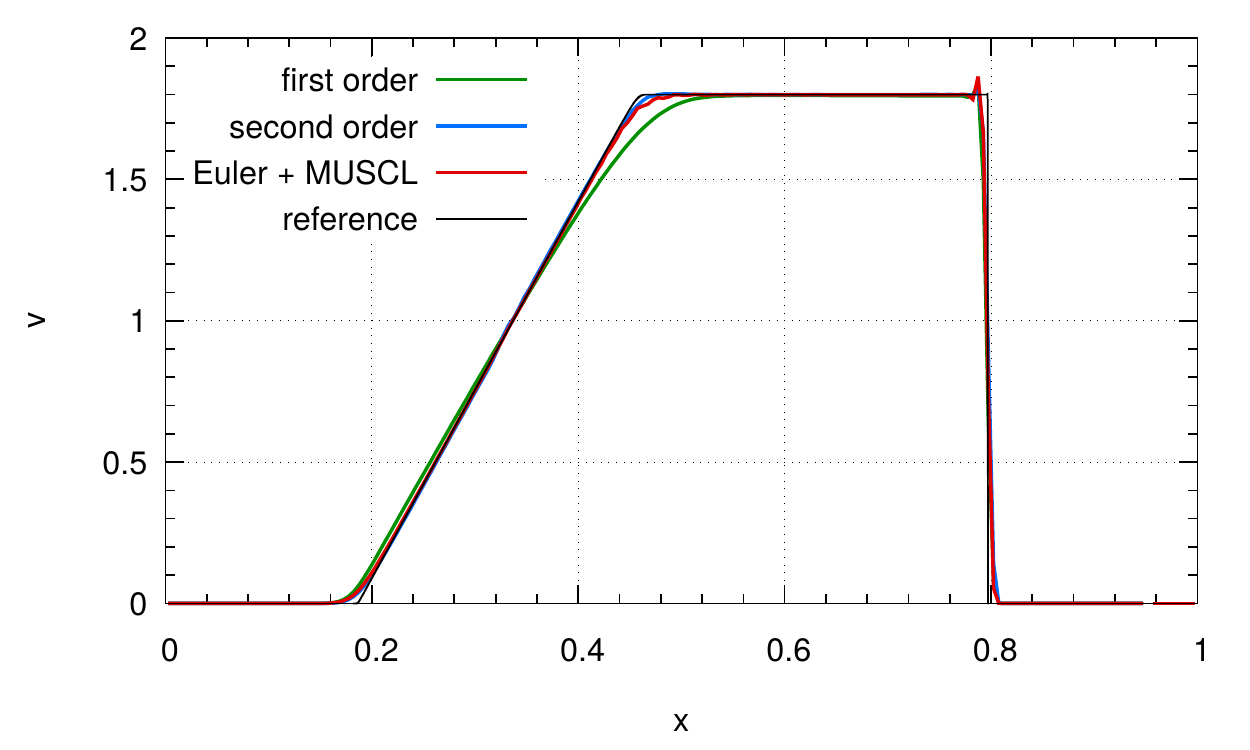}
\caption{Riemann problem. Top: flow height -- Bottom: velocity.}
\label{fig:Rpb1}
\end{figure}

\begin{figure}[htbp]
\includegraphics[width=0.65\textwidth]{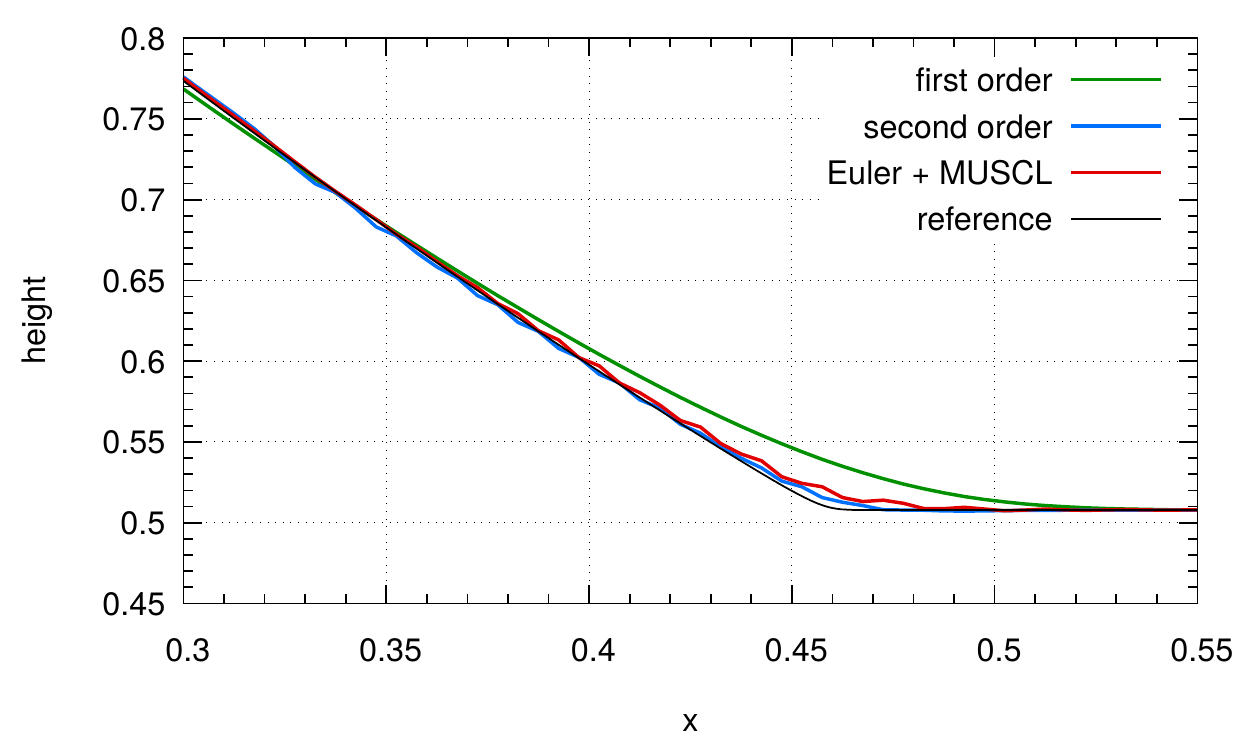}\\
\includegraphics[width=0.65\textwidth]{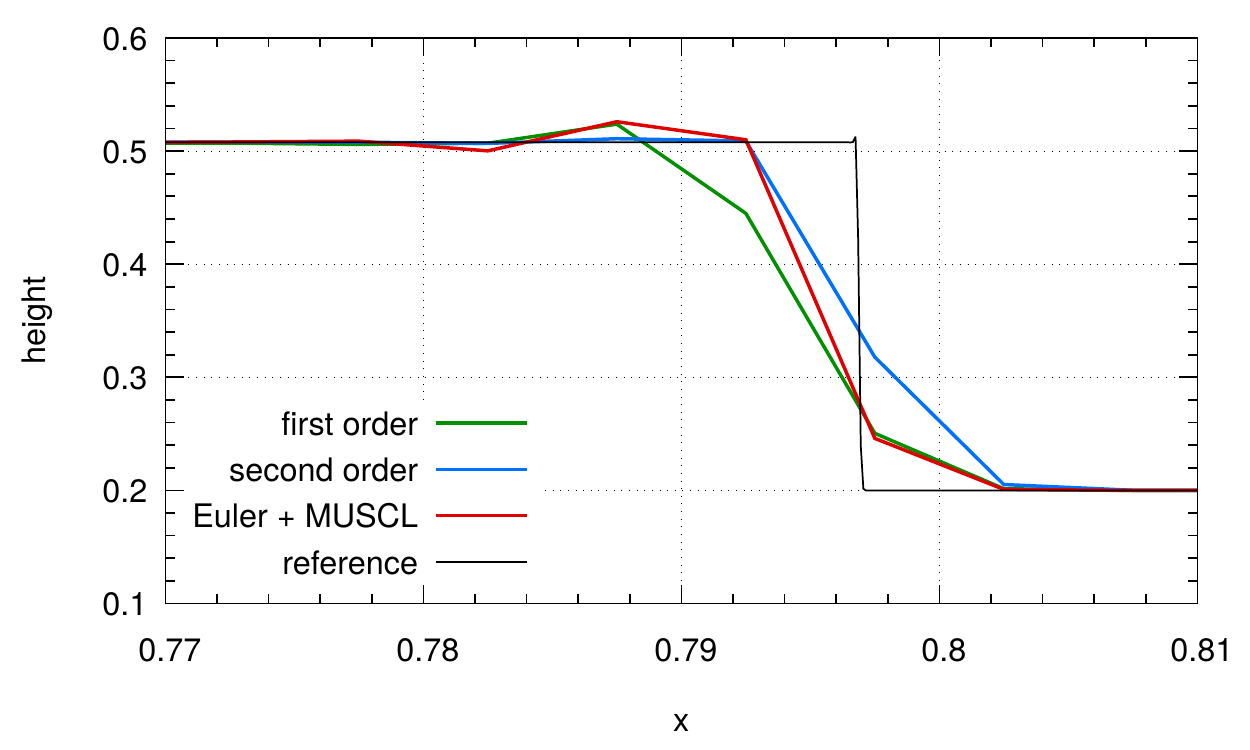}
\caption{Riemann problem. Details of the flow height.}
\label{fig:Rpb2}
\end{figure}
%
%

\subsection{A circular dam break problem} \label{subsec:2D_Riemann}

The objective of this test-case is to check the capability of the scheme to capture a multi-dimensional shock solution.
The fluid is initially at rest and the height is given by:
\[
h=2.5 \mbox{ if } r < 2.5,\ h=0.5 \mbox{ otherwise, with } r^2=x_1^2+x_2^2.
\]
The computational domain is $\Omega=(-20,20)\times(-20,20)$ and the final time is $T=4.7$.

We plot on Figure \ref{fig:cdb1} the results obtained with a $800 \times 800$ uniform mesh, with the second-order scheme.
The time-step is $\delta t=h_\mesh/10$ (with a maximal velocity in the range of $3.5$ and a maximal speed of sound in the range of $5$).
In addition, to cure some oscillations (see Figure \ref{fig:cdb3}), we add a slight stabilization in the momentum balance equation which consists in adding to the discrete momentum equation associated to an edge $\edge$ the following flux through any dual edge of $D_\edge$, with $\edged = D_\edge |D_{\edge'}$:
\[
F_{{\rm stab},\edge, \edged}= \zeta\ h_\edged\ \delta_\edged^{d-1}\ (u_\edge-u_\edge'),
\]
where $\zeta$ is a user-defined parameter, $h_\edged$ and $\delta_\edged$ are quantities representative of the fluid height and of the space step in the neighbourhood of $\edged$, respectively, and $d$ is the space dimension ($d=2$ in this test).
Here, $\zeta=0.1$, which is significantly lower than the diffusion generated by the use of an upwind scheme in the momentum balance equation; indeed, the upwind scheme may be seen as the centered one complemented by a diffusion taking the same expression as $F_{{\rm stab},\edge, \edged}$ with $\zeta\ h_\edged$ replaced by $|F_{\edge,\edged}|/2$.
The interest of this stabilization stems from the fact that the numerical diffusion introduced in the present family of schemes depends on the material velocity (and not on the waves celerity as, for instance, in colocated schemes based on Riemann solvers), and is sometimes too low in the zones where the fluid is almost at rest \cite{her-18-con}.
Note that, as a counterpart, the scheme does not become overdiffusive for low-Mach number flows.
For the same computation, we give on Figure \ref{fig:cdb2} the height and the radial velocity along the axis $x_2=0$ ({\it i.e.}\ the first component of the velocity) at different times.

This computation is also used as "reference computation" on Figure \ref{fig:cdb3}, where we compare the results obtained at $t=3T/5$ with a $200 \times 200$ mesh with the second-order scheme, the second-order scheme with stabilization and the first-order scheme.
This latter is significantly more diffusive, and we observe how the stabilization (even if added to the momentum balance only and not on the mass balance) damps the oscillations obtained with the second-order scheme for both the flow height and the velocity.

\begin{figure}[htbp]
\includegraphics[width=0.45\textwidth, clip=true, bb= 5cm 1cm 33.5cm 26.cm]{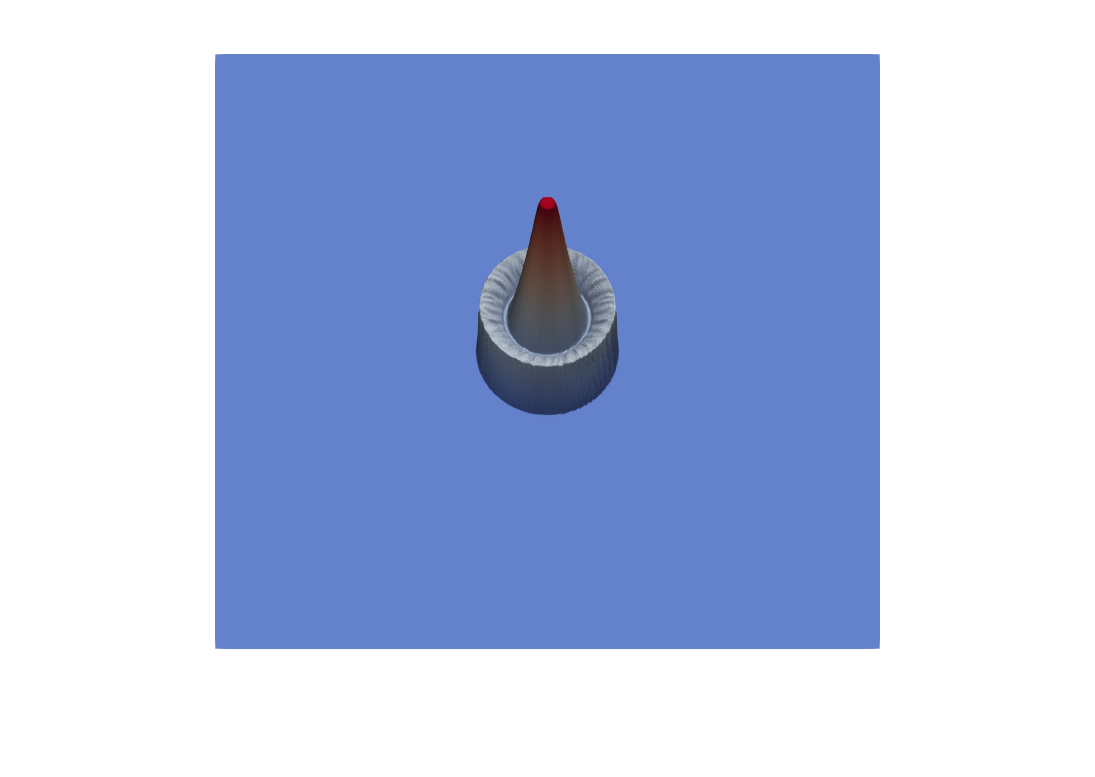}
\includegraphics[width=0.45\textwidth, clip=true, bb= 5cm 1cm 33.5cm 26.cm]{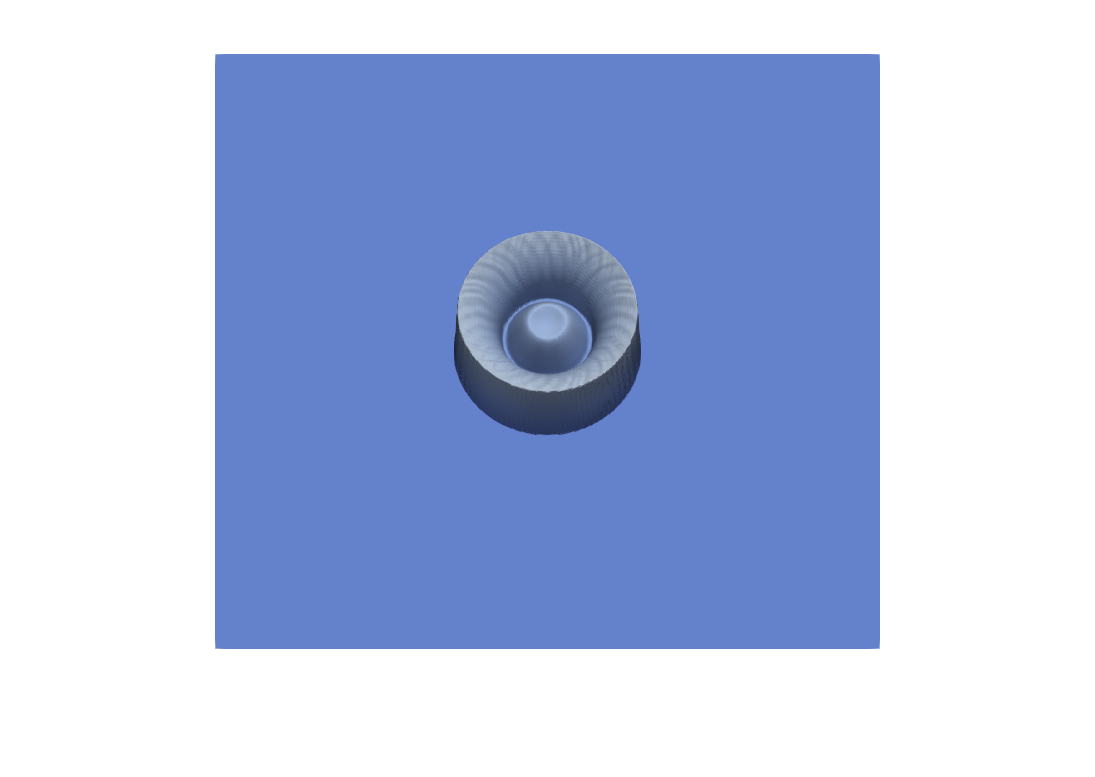}\\[1ex]
\includegraphics[width=0.45\textwidth, clip=true, bb= 5cm 1cm 33.5cm 26.cm]{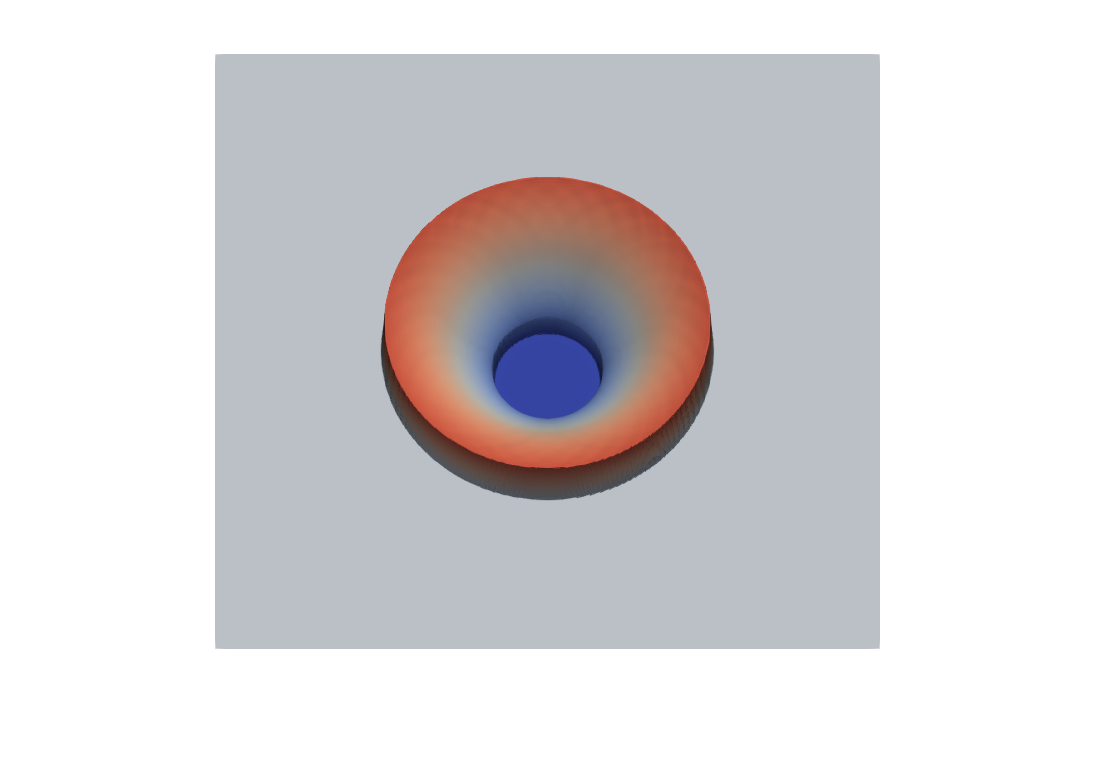}
\includegraphics[width=0.45\textwidth, clip=true, bb= 5cm 1cm 33.5cm 26.cm]{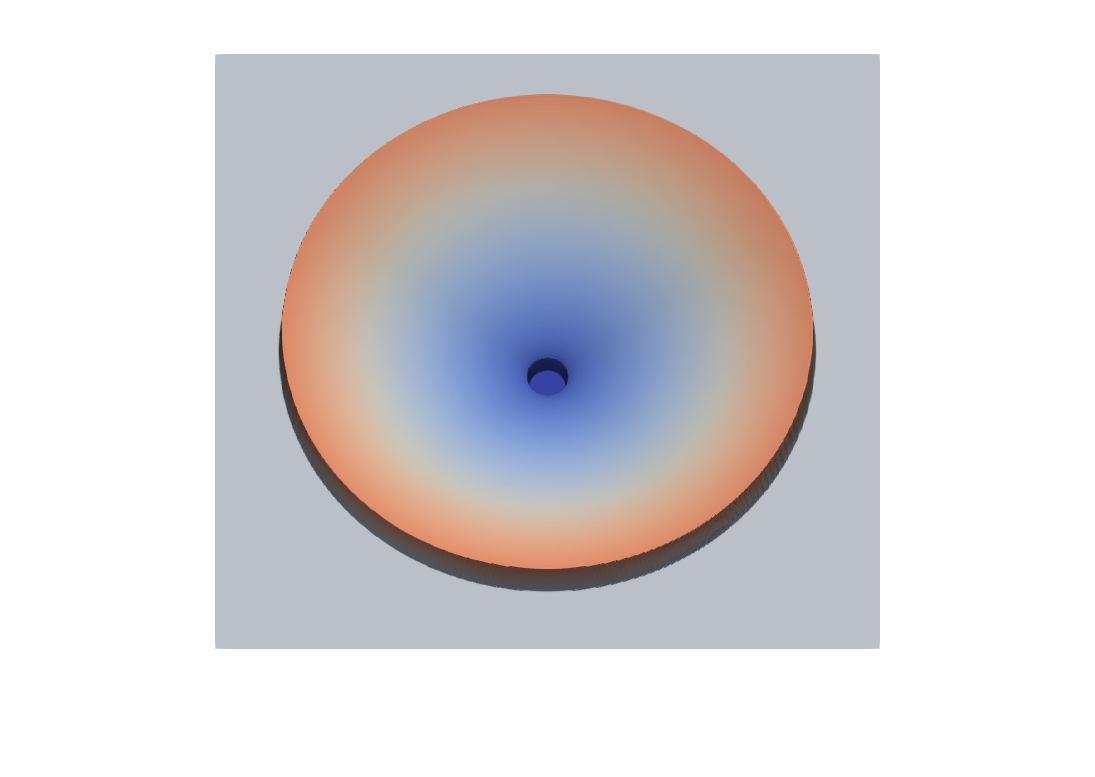}\\[1ex]
\includegraphics[width=0.45\textwidth, clip=true, bb= 5cm 1cm 33.5cm 26.cm]{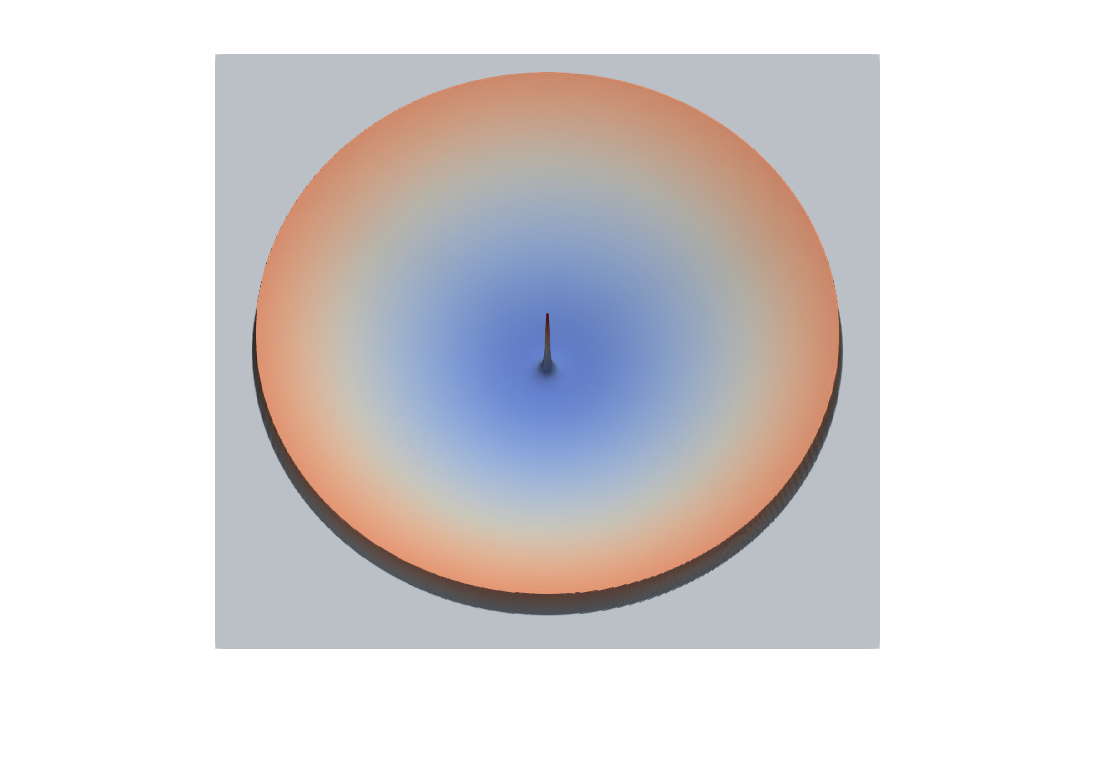}
\includegraphics[width=0.45\textwidth, clip=true, bb= 5cm 1cm 33.5cm 26.cm]{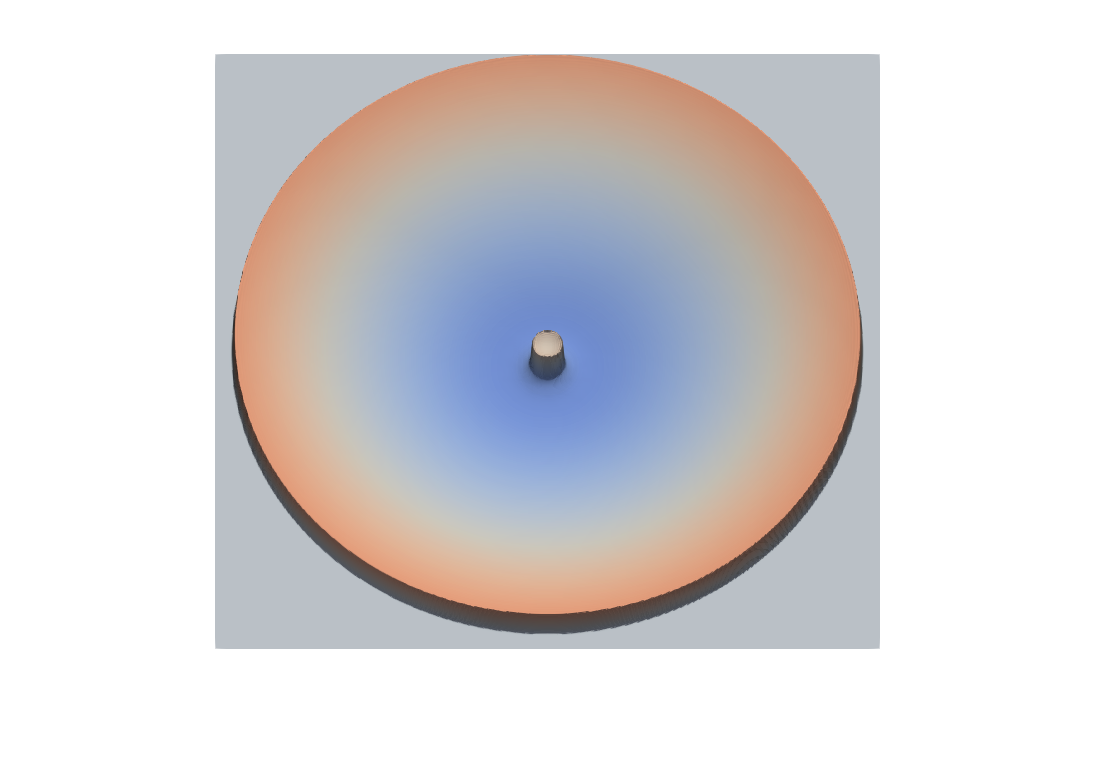}
\caption{Circular dam-break problem. Height obtained at $t=0.38$, $t=0.705$, $t=1.88$, $t=3.76$, $t=4.28$ and $t=T=4.7$ with the stabilized second-order scheme and a $800 \times 800$ mesh.
The color range corresponds to the $(0.1,2.5)$ interval for the first two plots, and to the $(0.1,1)$ interval for the last four ones.}
\label{fig:cdb1}
\end{figure}

\begin{figure}[htbp]
\includegraphics[width=0.75\textwidth]{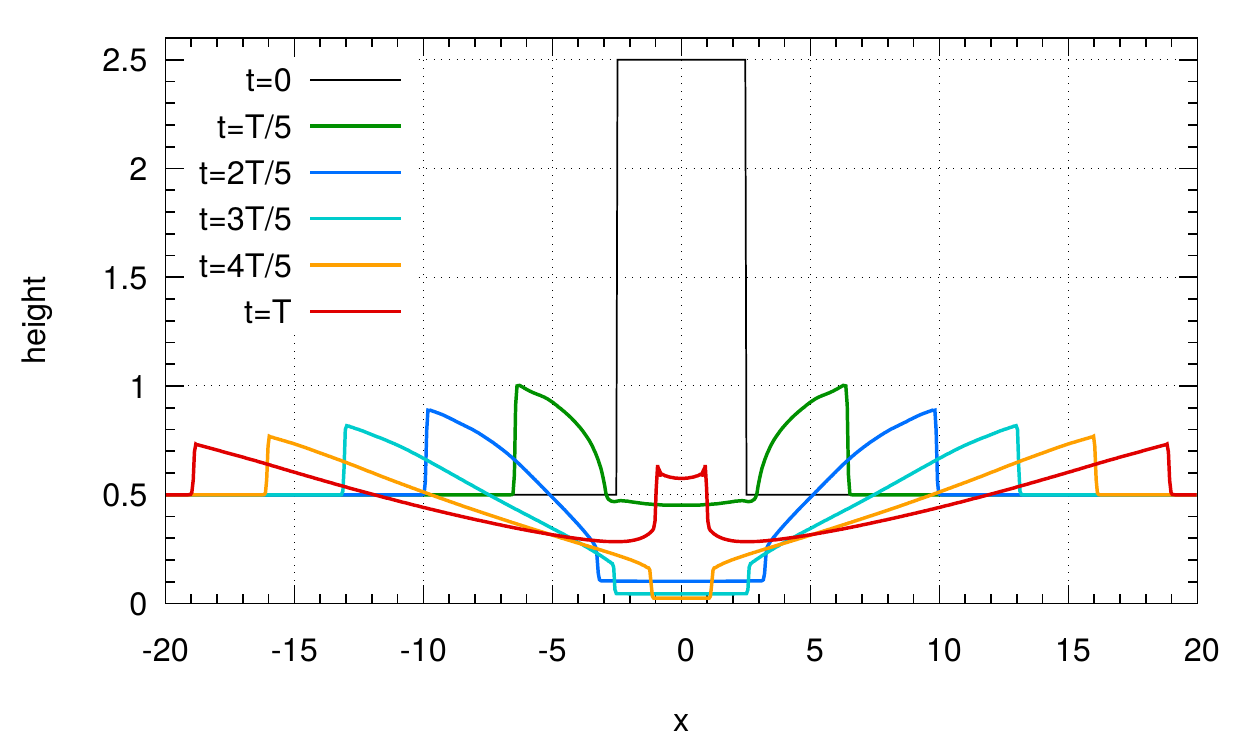}\\
\includegraphics[width=0.75\textwidth]{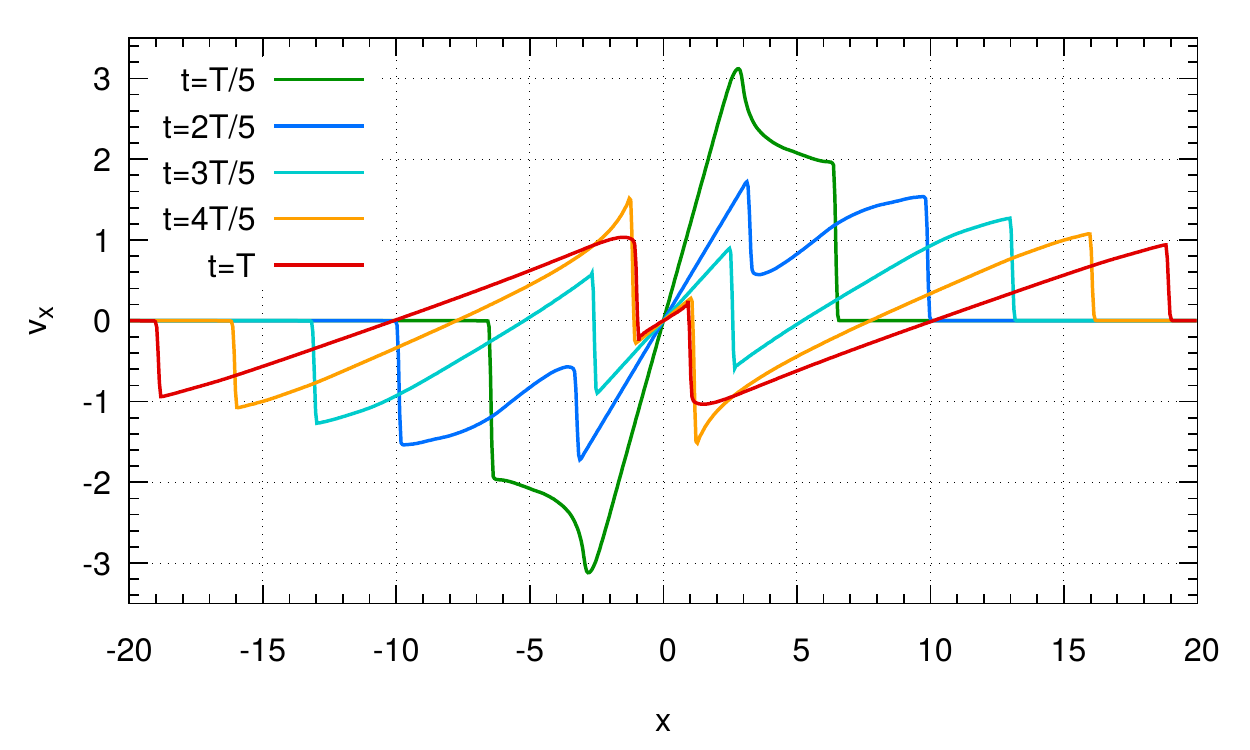}
\caption{Circular dam-break problem. Height and radial velocity obtained at different times along the line $x_2=0$ with the stabilized second-order scheme and a $800 \times 800$ mesh.}
\label{fig:cdb2}
\end{figure}

\begin{figure}[htbp]
\includegraphics[width=0.75\textwidth]{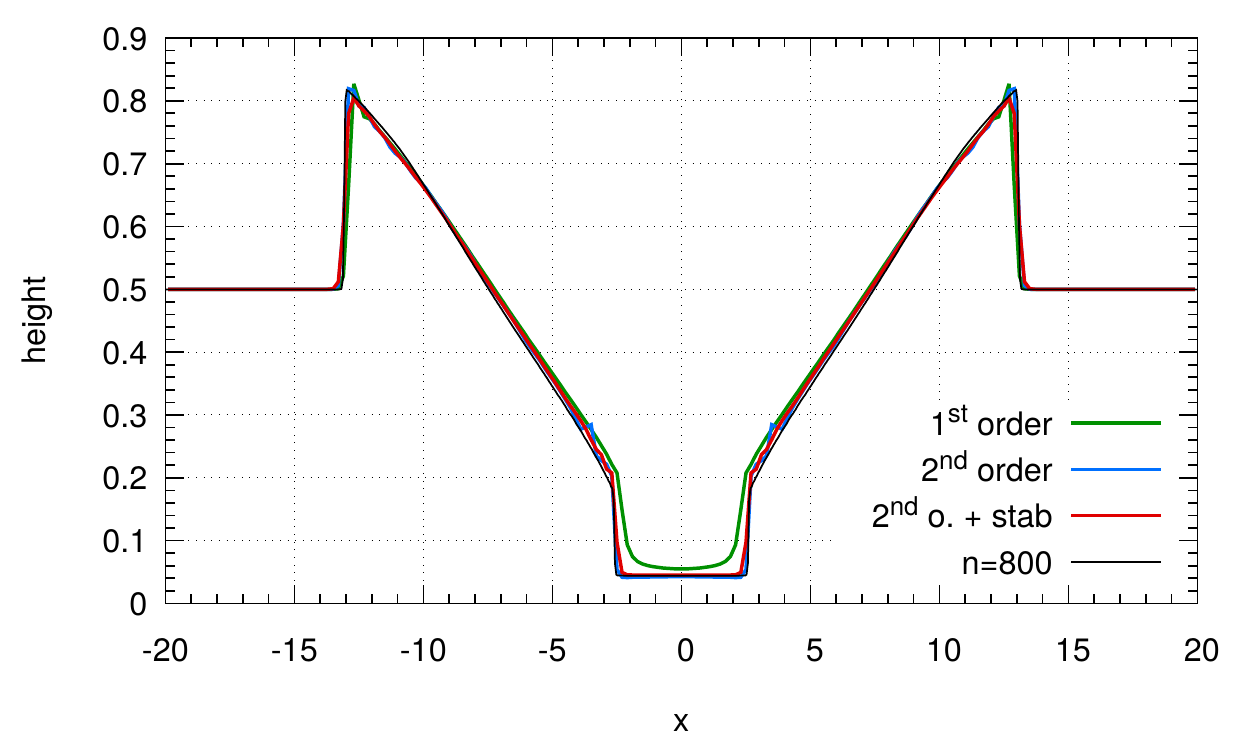}\\
\includegraphics[width=0.75\textwidth]{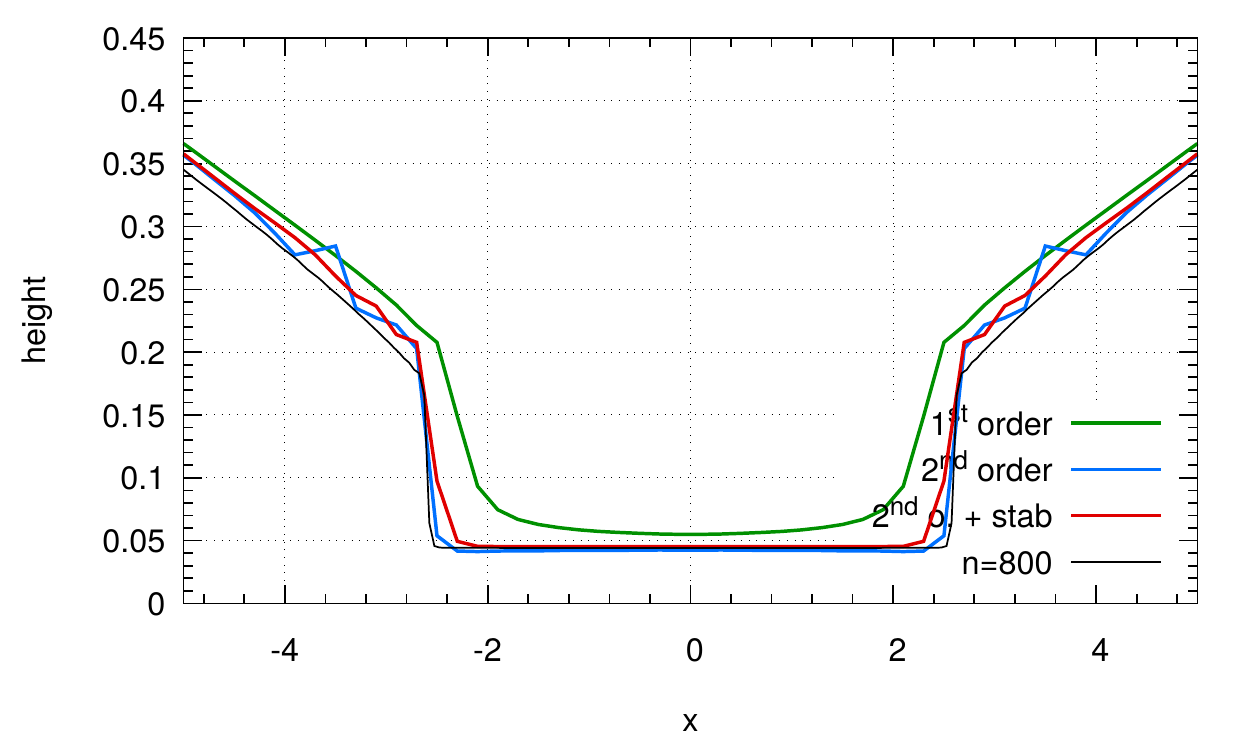}
\caption{Circular dam-break problem. Height obtained at $t=3T/5$ with the first-order scheme and the second-order scheme with and without stabilization, with a $200 \times 200$ mesh.}
\label{fig:cdb3}
\end{figure}
%
%
\subsection{A so-called partial dam-break problem} \label{subsec:part_dam_break}

We now turn to a test consisting in a partial dam-break problem with reflection phenomena, and with a non-flat bathymetry.
In this test, the computational domain is $\Omega=(0,200)\times(0,200) \setminus \Omega_w$ with $\Omega_w=(95,105)\times(0,95)\cup (95,105)\times(170,200)$.
The fluid is supposed to be initially at rest, the initial water height is $h=10$ for $x_1 \leq 100$ and $h=5-0.04\,(x_1-100)$ otherwise, and the bathymetry is $z=0$ if $x_1 \leq 100$ and $z=0.04\,(x_1-100)$ otherwise.
A zero normal velocity is prescribed at all the boundaries of the computational domain.
The computation is performed with a mesh obtained from a $1000 \times 1000$ regular grid by removing the cells included in $\Omega_w$.
The time step is $\delta t = \delta_\mesh /40$ (the maximal speed of sound and the maximal velocity are both close to $10$).
A stabilization with $\zeta=0.25$ (so two orders of magnitude lower than the artificial viscosity generated by the upwind scheme in high momentum zones) is added to damp oscillations appearing in the zones at rest, where no numerical diffusion is generated by our schemes.
Results obtained at $t=20$ with the first order in time and space and the present scheme are compared on Figure \ref{fig:pdb}.
One can observe that the second-order scheme is clearly less diffusive.
In addition, these results illustrate the capacity of the staggered scheme to deal with reflection conditions by simply imposing the normal velocity to the boundary at zero.

\begin{figure}[htbp]
\includegraphics[width=0.75\textwidth, clip=true, bb= 12.5cm  0.3cm  39cm  20.8cm]{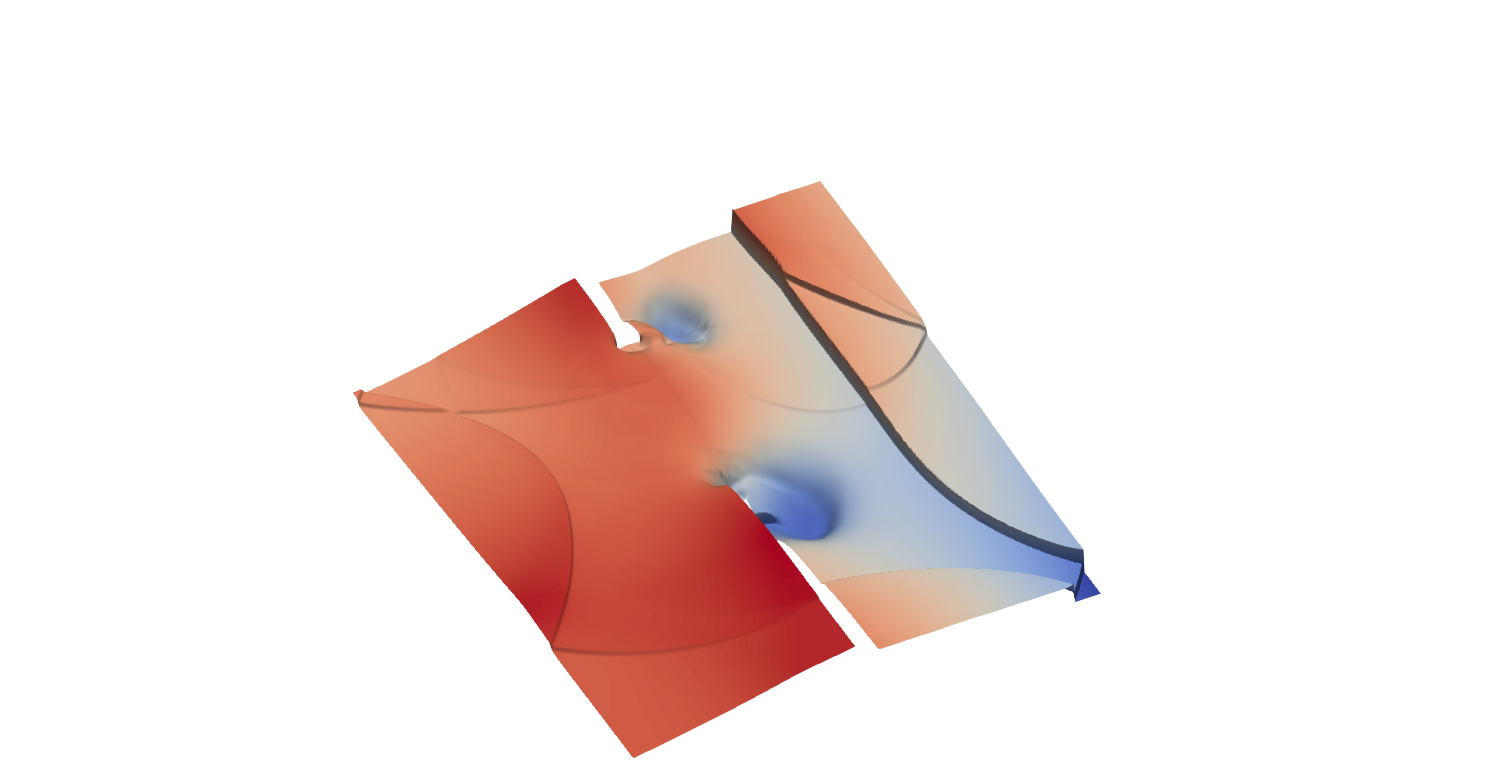} \\
\includegraphics[width=0.75\textwidth, clip=true, bb=  5.5cm  0.3cm  32cm  20.8cm]{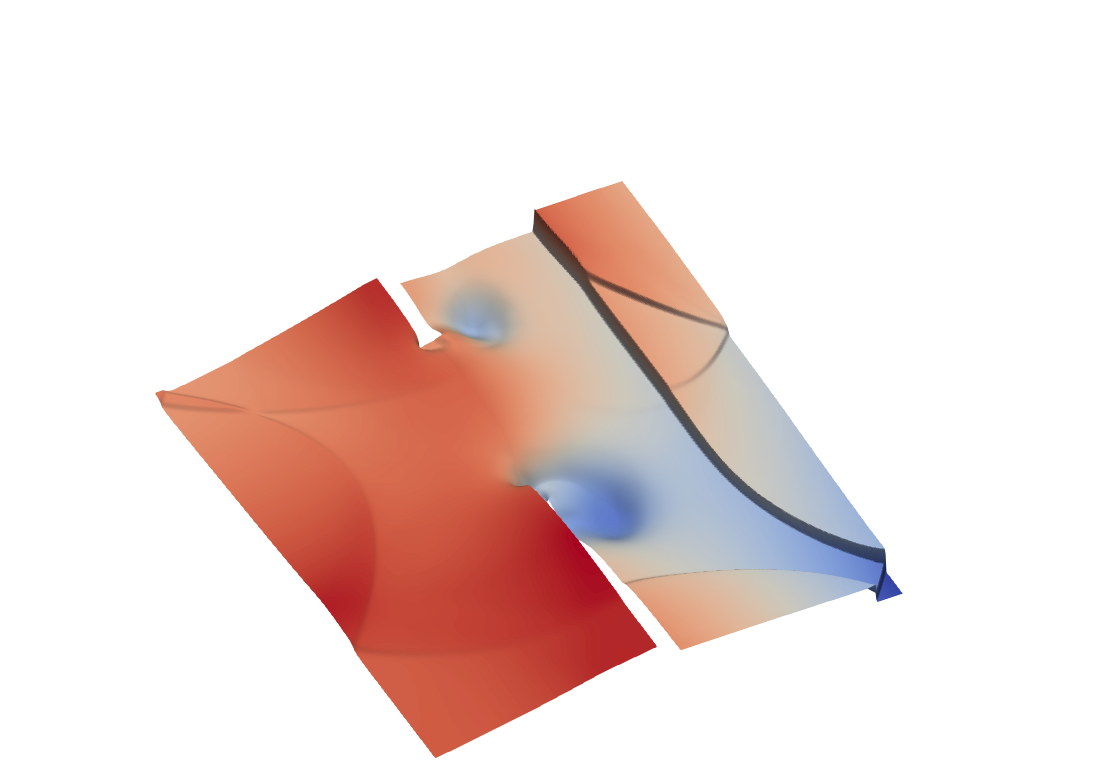}
\caption{Partial dam-break flow. Top: MUSCL scheme -- Bottom: upwind scheme.}
\label{fig:pdb}
\end{figure}

%
%
\subsection{Uniform circular motion in a paraboloid} \label{subsec:drop}

We address in this section a classical test which admits a closed-form solution and corresponds to the uniform rotation of a drop of liquid on a paraboloid-shaped support.
The solution is very regular (at a given time, the velocity field is constant and $h+z$ is affine outside the dry zones), and the essential interest of this test is to check whether the scheme is able to cope with dry zones, {\it i.e.}\ zones where the height is zero (in the continuous setting) or very close to zero, as we shall use numerically.
The computational domain is $\Omega=(0,L)\times(0,L)$ and the topography is given by
\[
z=- \frac{h_0}{a^2} \Bigl(a^2 - (x-\frac L 2)^2 - (y-\frac L 2)^2 \Bigr),
\]
with $h_0$ and $a$ parameters which are given below.
The height is:
\[
h = \max(0, \bar h) \mbox{ with } \bar h= \eta \frac{h_0}{a^2} \Bigl( 2\,(x - \frac L 2) \,\cos(\omega t) + 2 (y-\frac L 2) \,\sin(\omega t) - \eta \Bigr) - z,
\]
with $\eta$ a parameter and $\omega$ (the angular rotation velocity of the drop) given by
\[
\omega = \frac{(2 g h_0)^{1/2}} a.
\]
Finally, the velocity is
\[
\bfu = \eta\, \omega \begin{bmatrix} - \sin(\omega t)\\ \cos(\omega) t \end{bmatrix}.
\]
The computation is run up to $T= 6 \,\pi/\omega$, so the drop is supposed to perform 3 turns and to lie at the final time at its initial position.
The parameters are fixed here to $L=4$, $h_0=0.1$, $a=1$ and $\eta=0.5$.

For numerical tests, we bound $h$ from below by $10^{-8}$, {\it i.e.}\ we set $h=\max(10^{-8}, \bar h)$, in particular to avoid divisions by zero in the averaging steps of the Heun scheme (Equations \eqref{heun:last_mass} and \eqref{heun:last_mom}).
The computation are performed with a uniform $100 \times 100$ mesh, with $\delta t = \delta_\mesh/16$, without changing anything to the numerical fluxes to cope with dry zones.
This is clearly dangerous, since a non-upwind approximation of the water height at a face separating two cells with a large ratio of water height may lead to a huge outflow mass flux in view of the cell mass inventory (or, in other words, a very large CFL number).
This probably explains the rather small time step used here (the CFL number with respect to the celerity of the fastest waves is in the range of $1/8$); the first-order scheme, which uses upwind fluxes, works with time steps four times larger.
This problem would be probably cured by a more careful limitation of the mass fluxes outward an almost dry cell.

Results obtained with the first order, the segregated and the second order scheme at $t=6\,\pi /\omega$ are plotted on Figure \ref{fig:Ucm}.
All schemes give good results, which, for the first-order scheme, is probably due to the regularity of the solution.
For the momentum, one observes that the second-order scheme is less accurate than the other ones; this seems to be due to the time-stepping procedure, which perhaps generates some diffusion at the interface between dry and wet zones, especially in the last averaging step, since the segregated scheme is the most accurate one (and superimposed to the exact solution on Figure \ref{fig:Ucm}).

\begin{figure}[htbp]
\includegraphics[width=0.65\textwidth]{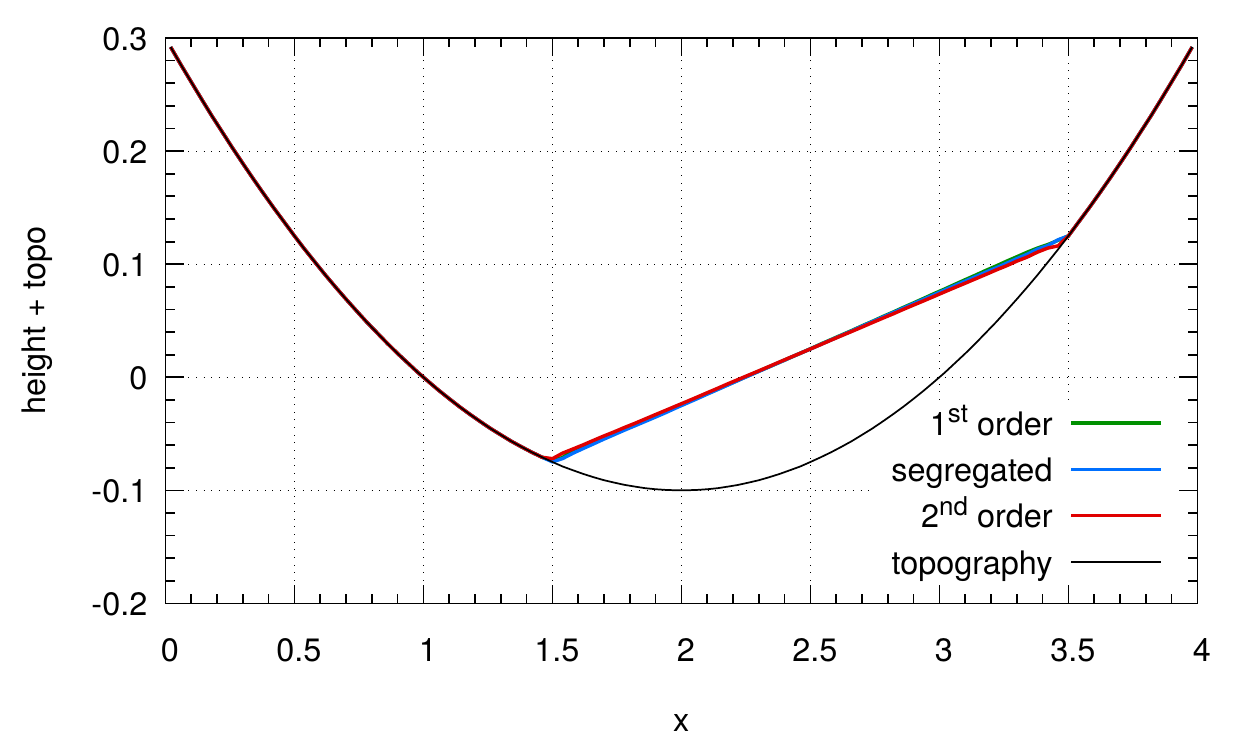}\\
\includegraphics[width=0.65\textwidth]{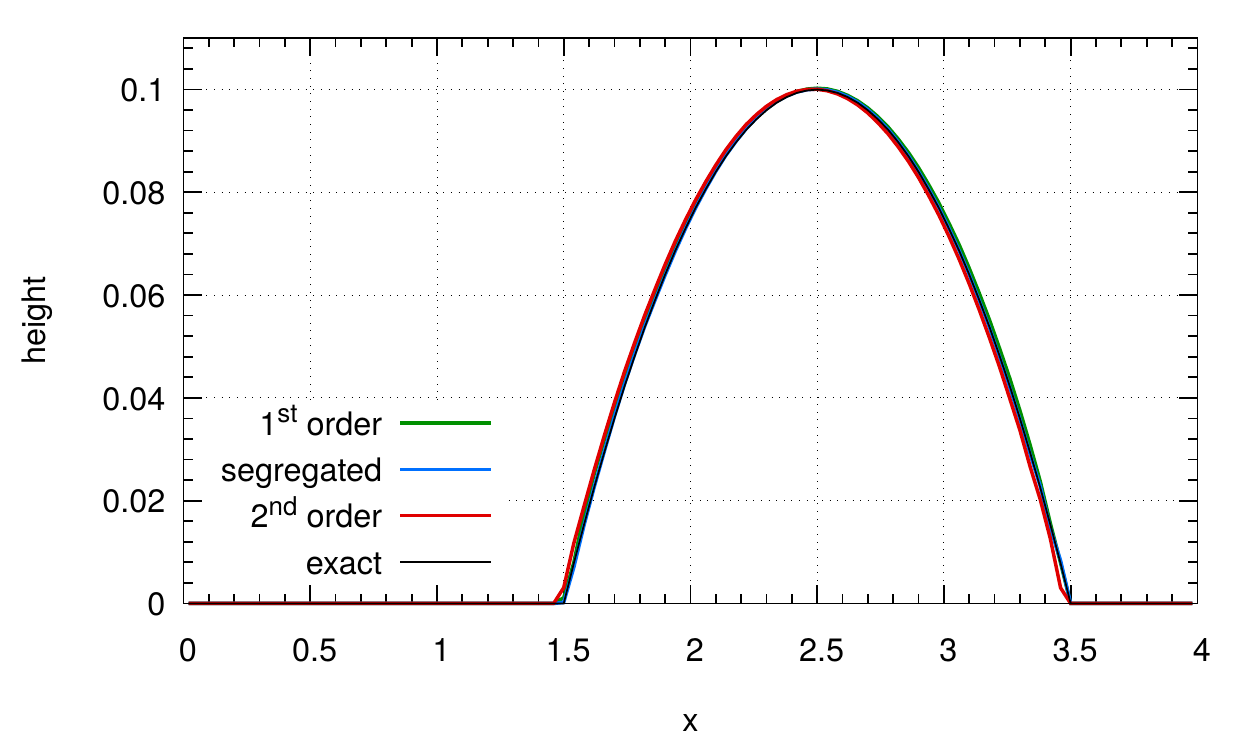}\\
\includegraphics[width=0.65\textwidth]{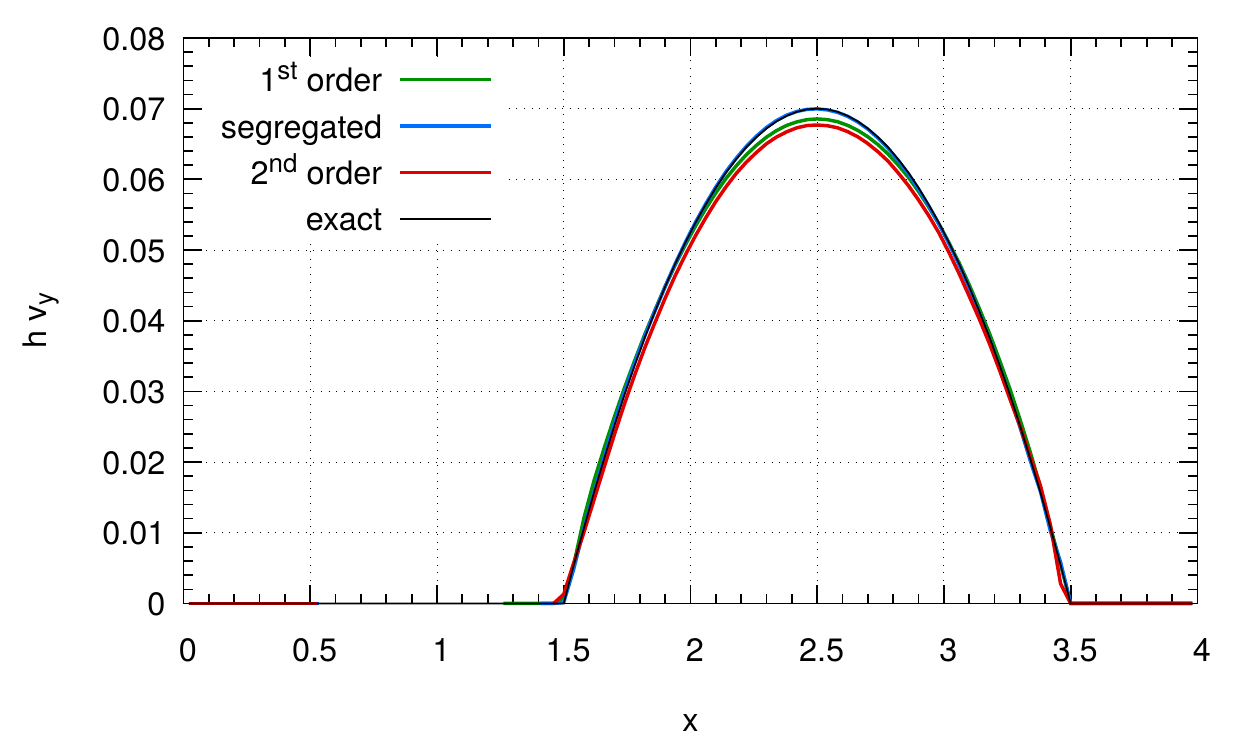}
\caption{Circular motion of a drop over a paraboloid-shaped topography. Sum of the height and the topography, height alone and second component of the momentum along the $y=L/2$ line at $t=6\,\pi$.}
\label{fig:Ucm}
\end{figure}

\bibliographystyle{abbrv}
\bibliography{sw}
\appendix
%
%
\section{Appendix}
This appendix gathers known results which are used in the stability or consistency proofs of the paper.

\smallskip
\textbf{Multiplying a finite volume convection operator by a function of the unknown} --
We begin with a property of the convection operator $\mathcal{C} : \rho \mapsto \partial_t(\rho) + \dive( \rho \bfu)$; at the continuous level, this property reads as follows (see \cite{her-18-con} for the detailed derivation).
Let $\psi$ be a regular function from $(0, +\infty)$ to $\xR$; then:
\begin{equation}
\psi'(\rho) \ \mathcal{C}(\rho)
 =\partial_t\bigl(\psi(\rho) \bigr) + \dive \bigl(\psi(\rho) \bfu \bigr)
+ \bigl(\rho \psi'(\rho) - \psi(\rho)\bigr)\ \dive \bfu.
\label{eq:renorm}
\end{equation}
This computation is of course completely formal and only valid for regular functions $\rho$ and $\bfu$.
The following lemma states a discrete analogue to \eqref{eq:renorm}.

\begin{lemma}\label{lem:A1}[On the discrete convection operator, \cite[Lemma A1]{her-13-exp}]
Let $P$ be a polygonal (resp. polyhedral) bounded set of $\xR^2$ (resp. $\xR^3$), and let $\edges(P)$ be the set of its edges (resp. faces).
Let $\psi$ be a twice continuously differentiable function defined over $(0,+\infty)$.
Let $\rho^\ast_P >0$, $\rho_P >0 $, $\delta t >0$; consider three families $(\rho^*_\eta)_{\eta \in \edges(P)} \subset \xR_+\setminus\{0\}, (V^*_\eta)_{\eta \in \edges(P)}\subset \xR$ and $(F^*_\eta)_{\eta \in \edges(P)} \subset \xR$ such that 
\[
\forall \eta \in \edges(P), \qquad F^*_\eta= \rho^*_\eta\ V^*_\eta.
\]
Let $R_{P,\delta t}$ be defined by:
\begin{multline*}
R_{P,\delta t}= \Bigl[\frac{|P|}{\delta t}\ (\rho_P - \rho_P^*) + \sum_{\eta \in \edges(P)} F^*_\eta \Bigr]\ \psi'(\rho_P)
\\
- \Big[\frac{|P|}{\delta t}\ [\psi(\rho_P) - \psi(\rho^*_P)] + \sum_{\eta \in \edges(P)} \psi(\rho^*_\eta) V^*_\eta
+ [\rho^*_P \psi'(\rho^*_P) - \psi(\rho^*_P)] \sum_{\eta \in \edges(P)} V^*_\eta \Big].
\end{multline*}
Then this quantity may be expressed as follows:
\[
R_{P,\delta t} = \frac 1 2 \frac{|P|}{\delta t}\, (\rho_P - \rho^*_P)^2\,\psi''(\overline \rho^{(1)}_P)
- \frac 1 2 \sum_{\eta \in \edges(P)} V^*_\eta\, (\rho^*_P - \rho^*_\eta)^2\, \psi''(\overline \rho^*_\eta)
+ \sum_{\eta \in \edges(P)} V^*_\eta \rho^*_\eta\, (\rho_P - \rho^*_P)\, \psi''(\overline \rho^{(2)}_P),
\]
where $\overline \rho^{(1)}_P,\ \overline \rho^{(2)}_P \in \llbracket \rho_P, \rho_P^*\rrbracket$ and $\forall \eta \in \edges(P)$, $\overline \rho^*_\eta \in \llbracket \rho^*_P, \rho^*_\eta \rrbracket$.
We recall that, for $a,\ b \in \xR$, we denote by $\llbracket a,b \rrbracket$ the interval $\llbracket a,b \rrbracket =\{\theta a + (1-\theta)b,\ \theta \in [0,1]\}$.
\end{lemma}

\bigskip
\textbf{Tools for the LW-consistency} -- 
We now turn to some results of \cite{gal-21-wea}; the first one generalises the Lax-Wendroff theorem to multidimensional problems with a general conservative operator applying to a set of unknowns belonging to a  quite general discretisation space; in particular, the discrete functions associated to the unknowns do not need to be piecewise constant or, as occurs with staggered discretisations, may be piecewise constant on different meshes.
The second result concerns the convergence of the space translates. 
Let us suppose that:
\begin{subequations} \label{pb-general}
\begin{align} &
\Omega \subset \xR^d, \; d= 1, 2,3, \; T \in (0,+\infty)
\\	&
p \in \xN^\ast, \;\beta \in C^1(\xR^p, \xR),\; \bff \in C^1(\xR^p, \xR^d),\; U \in \xL^\infty(\Omega\times(0,T), \xR^p), 
\end{align}
\end{subequations}
and consider the conservative convection operator defined (in the distributional sense) by: 
\begin{align} \label{def:op-conv} \nonumber
\mathcal C(U) : &\quad \Omega\times(0,T) 	\to \xR, 
\\ &
\quad (\bfx,t) \mapsto \partial_t(\beta(U))(\bfx,t) + \dive(\bff(U))(\bfx,t).
\end{align}

The following theorem is a straigthforward consequence of \cite[Theorem 2.1]{gal-21-wea}. 
\begin{theorem}[LW-consistency for a multi-dimensional conservative convection operator] \label{theo:lw}
Under the assumptions \eqref{pb-general}, let $(U\m)_\mnn \subset \xL^\infty(\Omega\times(0,T), \xR^p)$ be a sequence of functions such that:
\begin{align} \label{lemgen:linfbound} 	&
\exists \ C^u \in \xR_+^\ast \ :\ \Vert U\m \Vert_\infty \le C^u \ \forall \mnn,
\\ \label{lemgen:l1conv} &
\exists \ \bar U \in \xL^\infty(\Omega\times(0,T), \xR^p) : \Vert U\m - \bar U \Vert_{\xL^1(\Omega\times(0,T), \xR^p)} \to 0 \mbox{ as } m \to +\infty.
\end{align}
Let $(\mathcal P_m)_\mnn$ be a sequence of polygonal or polyhedral conforming meshes of $\ \Omega$ such that 
\[
\delta(\mathcal P_m) = \max_{P \in \mathcal P_m} \mathrm{diam} (P) \to 0 \mbox{ as } m \to +\infty.
\] 
Let $\edgespart\m$ denote the set of edges (or faces) of the mesh, and for a given polygon (or polyhedron) $P \in \mathcal P\m$, let $\edgespart\m(P)$ be the set of faces (or edges) of $P$.
For $\mnn$, let $t_0\m = 0 < t_1\m < \ldots < t\m_{N\m} = T$ be a uniform discretisation of $(0,T)$ with $\delta t\m = t_{k+1}\m - t_k\m \to 0 $ as $m\to + \infty$, and consider the discrete convection operator 
\begin{align*} 
\mathcal C\m(U\m) : & \quad \Omega\times(0,T) 	\to \xR,
\\ & \quad
(\bfx,t) \mapsto (\eth_t \beta\m)_P^n +\frac 1 {|P|} \sum_{\zeta \in \edgespart\m(P)} | \zeta | (\bfF\m)_\zeta^n \cdot \bfn_{P,\zeta} \quad \mbox{for } \bfx \in P \mbox{ and } t \in (t_n,t_{n+1})
\end{align*}
with
\[
(\eth_t \beta\m)_P^n = \frac 1 {\delta t} ((\beta\m)_P^{n+1} - (\beta\m)_P^n).
\]
We suppose that the families $\{(\beta\m)_P^n, P \in \mathcal P\m, \ n \in \llbracket 0,N\m \rrbracket \}$ of real numbers and $\{(\bfF\m)_{\zeta}^n, \zeta \in \edgespart\m, \ n \in \llbracket 0,N\m -1\rrbracket \}$ of real vectors are such that 
\begin{align} \label{hyp:condi}	&
\sum_{P \in \mathcal P\m}  \int_P \bigl|(\beta\m)_P^0 - \beta(U_0(\bfx)) \bigr| \dx \to 0 \mbox{ as } m \to + \infty, \mbox{ with } U_0 \in \xL^\infty(\Omega,\xR^p),
\\[1ex]	\label{hyp:t} &
\sum_{n=0}^{N\m-1} \sum_{P \in \mathcal P\m} \int_{t_n}^{t_{n+1}} \int_P \bigl|(\beta\m)_P^n - \beta(U\m(\bfx,t)) \bigr| \dx \dt \to 0 \mbox{ as } m \to + \infty,
\\[1ex] \label{hyp:x} &
\sum_{n=0}^{N\m-1} \sum_{P \in \mathcal P\m} \frac{\mathrm{diam} (P)}{|P|} \int_{t_n}^{t_{n+1}} \int_P
\sum_{\zeta \in \edgespart(P)} |\zeta|\ \bigl|\bigl((\bfF\m)_{\zeta}^n -\bff(U^m(\bfx,t) \bigr)\cdot \bfn_{P,\zeta} \bigr| \dx \dt \to 0 \mbox{ as } m \to + \infty. 
\end{align}
Let $\varphi \in C_c^\infty(\Omega\times[0,t))$, then
\begin{multline} \label{lw}
\int_0^T \int_\Omega \mathcal C\m(U\m)(\bfx,t) \varphi (\bfx,t) \dx \dt \to -\int_\Omega \beta(U_0(\bfx)) \varphi (\bfx,0) \dx
\\
- \int_0^T \int_\Omega \big(\beta(\bar U)(\bfx,t)\partial_t \varphi (\bfx,t) + \ \bff(\bar U)(\bfx,t) \cdot \gradi \varphi (\bfx,t) \big) \dx \dt \mbox{ as } m \to + \infty.
\end{multline}
\end{theorem}
The proof of this result relies on the next two lemmas, which are straigthforward consequences of \cite[Lemma 2.7 and Lemma 2.8]{gal-21-wea}.
\begin{lemma}[LW-consistency, time derivative] \label{lem:time-cons}
Under the assumptions and notations of Theorem \ref{theo:lw},
\begin{multline*}
\sum_{n=0}^{N\m-1} \sum_{P \in \mathcal P\m} \int_{t_n}^{t_{n+1}} \int_P (\eth_t \beta\m)_P^n\ \varphi (\bfx,t) \dx \dt
\\
\to 
- \int_\Omega \beta\bigl(U_0(\bfx)\bigr)\, \varphi (\bfx,0) \dx
- \int_0^T \int_\Omega \beta(\bar U)(\bfx,t)\ \partial_t \varphi (\bfx,t) \dx \dt \quad \mbox{ as } m \to + \infty.
\end{multline*}
\end{lemma}
\begin{lemma}[LW-consistency, space derivative] \label{lem:space-cons}
Under the assumptions and notations of Theorem \ref{theo:lw},
\begin{multline*}
\sum_{n=0}^{N\m-1} \sum_{P \in \mathcal P\m} \frac 1 {|P|} \int_{t_n}^{t_{n+1}} \int_P
\bigl(\sum_{\zeta \in \edgespart(P)} |\zeta| (\bfF\m)_\zeta^n \cdot \bfn_{P,\zeta}\bigr)\ \varphi (\bfx,t) \dx \dt
\\
\to 
- \int_0^T \! \!\int_\Omega \bff(\bar U)(\bfx,t) \cdot \gradi \varphi (\bfx,t) \dx \dt \quad \mbox{ as } m \to + \infty.
\end{multline*}
\end{lemma}

\begin{remark}[Disregarding the boundary cells]
From the proof of Theorem \ref{theo:lw}, it is clear that boundary cells may be disregarded in the sums appearing in assumptions \eqref{hyp:condi}-\eqref{hyp:x}.
It is due to the fact that, in this proof, all the terms are multiplied by the test function $\varphi$ and, since the support of $\varphi$ is compact in $\Omega \times (0,T)$, the function $\varphi$ vanishes in all the cells close to the boundary for $m$ large enough.
In this paper, when checking assumptions \eqref{hyp:condi}-\eqref{hyp:x}, we often use this remark to restrict the summation to the internal cells.
\end{remark}

We end with a result on the space translates, which was used several times (and systematically in this paper) to prove the LW-consistency of the schemes.
Note that the convergence of the time translates is a direct consequence of the Kolmogorov theorem, since in the present case the time step is constant; we refer to \cite[Lemma A.3]{gal-21-wea} for varying time steps and possibly multi-point in time schemes.
Let $\mathcal P$ be a polygonal or polyhedral mesh of $\Omega$.
For $(P,Q) \in \mathcal P^2$, let $\mathfrak d (\{P,Q\}) = \max_{(\bfx,\bfy) \in P \times Q} |\bfy-\bfx|$.
Let $\mathcal S_x$ be a set of cardinal 2 - subsets of $\mathcal P$ and let $\mathfrak d(\mathcal P) = \max_{\{P,Q\} \in \mathcal S_x} \mathfrak d (\{P,Q\})$.
Let $(\omega_{P,Q})_{\{P,Q\} \in \mathcal S_x}$ be a set of non-negative weights, and define
\begin{equation} \label{eq:reg_gene}
\begin{array}{l} \displaystyle
\theta_{\mathcal P}=\max_{P \in \mathcal P} \frac 1 {|P|}\ \sum_{\substack{Q \in \mathcal P\\ \{P,Q\} \in \mathcal S_x}} \omega_{P,Q}.
\end{array}
\end{equation}
A time step $\delta t$ is defined as $\delta t = \dfrac T N$ with $N \in \xN, N>1$; for $n = 0, \ldots, N$, se set $t_n = n \delta t$.

For a function $u \in L^1(\Omega \times(0,T))$, let $u_P^n$ be the mean value of $u$ on $P \times (t_n,t_{n+1})$ and let $T_{\mathcal P,\delta t}^{(x)}\, u$  and  $T_{\mathcal P,\delta t}^{(t)}\, u$ be defined by
\begin{equation}\label{transT-u_gene}
T_{\mathcal P,\delta t}^{(x)}\, u= \sum_{n=0}^{N-1}  \delta t \sum_{\{P,Q\} \in \mathcal S_x} \omega_{P,Q}\ |u^n_Q - u^n_P|
\quad \mbox{and } T_{\mathcal P,\delta t}^{(t)} = \sum_{n=0}^{N-1}\ \delta t \sum_{P\in\mathcal P} |P|\ |u^{n+1}_P - u^n_P|.
\end{equation}
Then the following convergence result holds.

\begin{lemma}[Limit of space and time translates]\label{lem:translates}
Let $(\mathcal P\m)_{m \in \xN}$ be a given sequence of meshes and $(\delta t\m)_{m \in \xN}$ a sequence of time steps.
Let us suppose that there exists $\theta > 0$ such that $\theta_{\mathcal P\m} \le \theta$  with $\theta_{\mathcal P\m}$  given by Equation \eqref{eq:reg_gene}.
Let us assume that $\mathfrak d(\mathcal P\m)$ and $\delta t\m$ tend to zero when $m$ tends to $+\infty$.
Let $u \in \xL^1(\Omega\times(0,T))$ and $(u_p)_{p \in \xN}$ be a sequence of functions of $\xL^1(\Omega\times(0,T))$ such that $u_p \to u$ in $\xL^1(\Omega\times(0,T))$ as $p \to +\infty$.\\[0.5ex]
Then $T_{\mathcal P\m,\delta t\m}^{(x)}\, u_p$  and $T_{\mathcal P\m,\delta t\m}^{(t)}\, u_p$ defined by \eqref{transT-u_gene} tend to zero when $m$ tends to $+\infty$, uniformly with respect to $p \in \xN$.
\end{lemma}
\end{document}